%% file: vf_mimo_paper.tex
\documentclass[final,10pt,a4paper]{siamltex}
\usepackage{amsmath,amstext,amsbsy,amsopn,amscd,amsxtra,upref,amssymb}
\usepackage[dvips]{graphicx}
\usepackage{color}
\usepackage{algorithmic}
\usepackage{algorithm}
\usepackage{pictex}
\usepackage{cyrillic}
\usepackage{mathrsfs}
\usepackage{url}
\input{vf_def}

\title{Vector Fitting for Matrix-valued Rational Approximation}
\author{Z.~Drma\v{c}\thanks{Faculty of Science, Department of Mathematics, University of Zagreb,
Bijeni\v{c}ka 30, 10000 Zagreb, Croatia. The work of this author was supported by the grant
HRZZ-9345 from the Croatian Science Foundation.} \and S.~Gugercin\thanks{Department of Mathematics,
Virginia Polytechnic Institute and State University,
460 McBryde, Virginia Tech,
Blacksburg, VA 24061-0123. The work of C. Beattie and S. Gugercin was supported in
part by NSF through Grant DMS-1217156.} \and C.~Beattie$^\dagger$}
\begin{document}
\maketitle

\begin{abstract}
%Vector Fitting (\textsf{VF}) is a popular method of constructing rational approximants that provides a least squares fit to frequency response measurements. In an earlier work, we provided an analysis of \textsf{VF} for scalar-valued rational functions and established a connection with the optimal $\Hardy_2$ approximation problem.   We build on this work and extend our earlier framework to produce effective rational approximations to matrix-valued rational functions of dimension $p\times m$.  Transfer functions associated with finite-dimensional multi-input/multi-output (MIMO) dynamical systems typify the class of functions we consider here.   
%%having predetermined McMillan degree $r$.  
%Previous work by Gustavsen and Semlyen considered extensions of \textsf{VF} to matrix-valued rational functions. 
%We continue this theme here, introducing new mechanisms for controlling the McMillan degree of the final approximant and providing both a careful implementation and analysis of the numerical issues that arise. 
%Our approach is motivated in part by recent innovations for computing optimal $\Hardy_2$ approximations.   
%We provide several numerical examples to support the theoretical discussion.
%\textcolor{red}{\bf The arrogant version of Abstract} 
Vector Fitting (\textsf{VF}) is a popular method of constructing rational approximants that provides a least squares fit to frequency response measurements. In an earlier work, we provided an analysis of \textsf{VF} for scalar-valued rational functions and established a connection with optimal $\Hardy_2$ approximation.   We build on this work and extend the previous framework to include the construction of effective rational approximations to \emph{matrix-valued} functions, a problem which presents significant challenges that do not appear in the scalar case.  Transfer functions associated with multi-input/multi-output (MIMO) dynamical systems typify the class of functions that we consider here.     
%Previous work by Gustavsen and Semlyen 
Others have also considered extensions of \textsf{VF} to matrix-valued functions and related numerical implementations are readily available.  However to our knowledge, a detailed analysis of numerical issues that arise does not yet exist.  We offer such an analysis  including critical implementation details here.  
%and an analysis from a systems theoretic and numerical linear algebra perspective of the resulting procedure do not exist even though a numerical implementation already does. The first contribution of this paper is to explain the fundamentals of  \textsf{VF} for matrix-valued functions in detail by making clear its connection to the Sanathanan-Koerner iterations. We use this analysis, then, to provide both a careful implementation and analysis of the numerical issues that arise.

 One important issue that arises for \textsf{VF} on matrix-valued functions that has remained largely unaddressed is the control of the McMillan degree of the resulting rational approximant; the McMillan degree can grow very high in the case of large input/output dimensions.   We introduce two new mechanisms for controlling the McMillan degree of the final approximant, one based on alternating least-squares minimization and one based on ancillary system-theoretic reduction methods.
  Motivated in part by our earlier work on the scalar \textsf{VF} problem as well as by recent innovations for computing optimal $\Hardy_2$ approximation, we establish a connection with optimal $\Hardy_2$ approximation, and are able to improve significantly the fidelity of  \textsf{VF} through numerical quadrature, with virtually no increase in cost or complexity.  We provide several numerical examples to support the theoretical discussion and proposed algorithms.

\end{abstract}

\begin{keywords}
least squares, frequency response, model order reduction, MIMO vector fitting, transfer function
\end{keywords}

\begin{AMS}
34C20, 41A05, 49K15, 49M05, 93A15, 93C05, 93C15
\end{AMS}

\section{Introduction}
%Background on earlier work on MIMO vector fitting - discussion of difficulty of producing a model of given McMillan degree with these methods.  
%\cite{Deschrijver-Mrozowski-Dhaene-Zutter-2008, gustavsen2002computer} Are there others ? 
Rational functions provide significant advantages over other classes of approximating functions, such as polynomials or trigonometric functions, that are important in the approximation of functions that occur in engineering and scientific applications.  \emph{Matrix-valued} rational functions offer substantial additional flexibility and broaden the domain of applicability by providing the potential for the interpolation and approximation of parameterized families of multi-dimensional observations.  
For example, in a variety of engineering applications, the dynamics arising from multi-input/multi-output (MIMO) dynamical systems may be inaccessible to direct modeling, yet input-output relationships often may be observed as a function of frequency, yielding 
an enormous amount of data.  In such cases, one may wish to deduce an \emph{empirical dynamical system model} nominally represented as a matrix-valued rational function, that fits the measured frequency response data.  This derived model may then be used as a surrogate in order to predict system behavior or to determine suitable control strategies.
See \cite{antoulas2005approximation,BGW2013,schilders2008model} for examples of rational approximation in action.

%\marginpar{\raggedright\tiny{
%(MORPaLS example\\ for parametric \\ inversion ? \\ Other PMOR \\ examples \\
%PB,SG\&KW \\ review ?)
% }}
%  \marginpar{\raggedright\tiny{
%SG: Added citations \\ without details
% }}
  
The \emph{McMillan degree} of a matrix-valued rational function, $\bfH(s)$, is the sum of pole multiplicities over the (extended) complex plane, or equivalently, the dimension of the state space in a minimal realization of $\bfH(s)$.
It is convenient to think of $\bfH(s)$ as a transfer function matrix associated with a stable MIMO linear time-invariant system having $m$ inputs and $p$ outputs (although this interpretation is not necessary for what follows); 
McMillan degree is a useful proxy for the level of complexity associated with $\bfH(s)$.   
Let $r>0$ be an integer denoting the desired McMillan degree for our approximant:
%  We assume a rational matrix approximant having the form 
  $\bfH_r(s)=\bfN(s) / d(s)$, where $\bfN(s)$ is a $p\times m$ matrix having elements that are polynomials  in $s$ of order $r-1$ or less, and $d(s)$ is a (scalar) polynomial function having exact order $r$.   Denote by $\mathcal{R}_r$ the set of matrix-valued functions of this form.  Note that $\mathcal{R}_r$ consists of matrix-valued functions having entries that are strictly proper rational functions of order $r$;  the McMillan degree of $\bfH_r\in\mathcal{R}_r$ could range as high as $r\cdot\min(m,p)$, which could be significantly larger than our target, $r$.   
  We assume that observations (evaluations) of $\bfH(s)$ are available at predetermined points in the complex plane, $s=\xi_1,\,\ldots,\,\xi_{\ell}$.  
 Only observed values at these points: $\bfH(\xi_j)$, for $j=1,\ldots,\ell$, will be necessary 
% the only information presumed available with which to 
 in order to derive our approximations.  
 We proceed by seeking a solution to
\begin{equation}\label{basicOptProb}
\min_{\bfH_r\in\mathcal{R}_r} \sum_{i=1}^\ell \qw_i \left\| \bfH_r(\xi_i) - \bfH(\xi_i) \right\|_F^2 \; .
\end{equation}

Sanathanan and Koerner \cite{Sanathanan-Koerner-1963} proposed an 
approach to solving (\ref{basicOptProb})  that produces a sequence of 
 rational matrix approximants $\bfH_r^{(k)}(s)$ having the form $\bfH_r^{(k)}(s)=\bfN^{(k)}(s) / d^{(k)}(s)$, where $\bfN^{(k)}(s)$ is a $p\times m$ matrix of polynomials of degree $r-1$ or less and  $d^{(k)}(s)$ is a (scalar-valued) polynomial of degree $r$. 
For each $k$, the coefficients of $\bfN^{(k)}$ and $d^{(k)}$ are adjusted by solving a weighted linear least squares problem with a weight determined by $d^{(k-1)}$.   An important reformulation of the Sanathanan-Koerner (\textsf{SK}) iteration was introduced by Gustavsen and Semlyen \cite{Gustavsen-Semlyen-1999}, which became known under the name \emph{Vector Fitting} (\textsf{VF}). The term ``vector fitting" is appropriate in light of the interpretation of the Frobenius norm, $\| \cdot \|_F$, that appears in (\ref{basicOptProb}) as a standard Euclidean vector norm,  $\| \cdot \|_2$, of an `unraveled' matrix listed column-wise as a vector in $\IC^{mp}$.  The Gustavsen-Semlyen \textsf{VF} method produces a similar sequence of rational approximants $\bfH_r^{(k)}(s)=\bfN^{(k)}(s) / d^{(k)}(s)$ but now with $\bfN^{(k)}$ and $d^{(k)}$ defined as rational functions represented in barycentric form.  Making central use of a clever change of representation at each step, the Gustavsen-Semlyen \textsf{VF} method achieves greater numerically stability and efficiency than the original Sanathanan-Koerner iteration.   

We describe both the \textsf{SK}  and \textsf{VF} iteration in \S \ref{S=Background} and make some observations that contribute to our analysis of it in \S \ref{S=NumericalDetails}.    In particular, we note that the change of representation implicit in the \textsf{VF} iteration can be related to a change of representation from barycentric to pole--residue form. This change of representation is made explicit in \S \ref{S=BC2PR}, where we provide formulas  \textcolor{black}{that appear to be new.} These formulas can be useful at any step of either the \textsf{SK} or \textsf{VF} iteration in order to estimate the contribution that each pole makes to the current approximation.  This, in turn, is useful in determining whether, $r$, the initial estimate of McMillan degree, is unnecessarily large relative to the information contained in the observed data.  
We note that many authors have applied, modified, and analyzed \textsf{VF}, see e.g. 
\cite{Gustavsen-2006}, \cite{Hendrickx-Dhaene-2006}, \cite{Deschrijver-Haegeman-Dhaene-2007},
\cite{Deschrijver-Gustavsen-2007}, \cite{Deschrijver-Mrozowski-Dhaene-Zutter-2008}, \cite{Deschrijver-Knockaert-Dhaene-2010}, \cite{Deschrijver-Knockaert-Dhaene-2013}. 
\textcolor{black}{A  \textsc{matlab} implementation \textsf{vecfit3} is provided at \cite{vectfit3} and is widely used.}
 When applied to MIMO problems, high fidelity rational approximations are sought and generally achieved at the expense of a relatively high McMillan degree for the final approximant. By way of contrast, the approaches we develop here are capable of providing systematic estimates to the McMillan degree of the original function, $\bfH(s)$, and can produce high fidelity rational approximations of any desired McMillan degree, to the extent possible. 

In \S \ref{S=NumericalDetails}, we analyze \textsf{VF} within a numerical linear algebra framework and discuss several important issues that are essential for a numerically sound implementation of the method. Although \textsf{VF} is based on successive solution of least squares (\textsf{LS}) problems, which is a well understood procedure, subtleties enter in the \textsf{VF} context.  We review commonly used numerical \textsf{LS} solution procedures in \S \ref{SS=LS+QR}  and point out some details that become significant in the \textsf{VF} setting.   In \S \ref{SS=noise}, we argue and illustrate by way of example that, as iterations proceed, the \textsf{VF} coefficient matrices in the pole identification phase tend to become noisy with a significant drop in column norms that can also coincide with a reduction in numerical rank. This prompts us to advise caution when rescaling columns in order to improve the condition number, since rescaling columns that have been computed through massive cancellations will preclude inferring an accurate numerical rank.

%First, we discuss the relevance of numerical rank revealing pivoting in the QR factorization that is used in 
%\textsf{LS} solvers, such as the "backslash" operator in \textsc{matlab}. We also stress the important difference, in the rank deficient cases, between the minimum norm solution (as computed using the Moore--Penrose pseudo-inverse) and a solution with certain sparsity pattern. Furthermore, we argue and illustrate by an example that the \textsf{LS} coefficient matrices in the pole identification phase throughout iterations become noisy with significant drop in their column norms, and with possible reduction of their numerical ranks. This prompted us to recommend caution when scaling the columns to improve the condition number, because scaling the columns that are computed by massive cancellations precludes inferring the numerical rank. 

Ill-conditioning is intrinsic to rational approximation and 
manifested through the potentially high condition numbers of Cauchy matrices that naturally 
arise.  We introduce in \S \ref{SS=HALA} 
another approach toward curbing  ill-conditioning. 
Using recent results on high accuracy matrix computations, we show that successful computation is possible independent of the condition number of the underlying Cauchy matrix, and we open possibilities for application of Tichonov regularization and the Morozov discrepancy principle, which provide additional useful tools when working with noisy data.  In \S \ref{SSS:LessQRs}, we offer some suggestions for efficient software implementation and provide some algorithmic details that reduce computational cost.
{For the sake of brevity, we omit discussion of the algorithmic details that allow computation to proceed using only real arithmetic.}

 Typically,  the weights in (\ref{basicOptProb}) are $\qw_i=1$ and the sample points, $\xi_i$, are either uniformly or logarithmically spaced according to sampling expedience.   We do consider this case in detail, but also propose, in \S 4, an alternate strategy for choosing $\qw_i$ and $\xi_i$ that is guided by numerical quadrature. 
This connects the ``vector fitting" process with optimal $\Hardy_2$ approximation and follows up on our recent work, \cite{Drmac-Gugercin-Beattie:VF-2014-SISC}. Indeed, with proper choice of the nodes $\xi_i$ and the weights $\qw_i$, we can interpret the weighted \textsf{LS} approximation of (\ref{basicOptProb}) as rational approximation in a discretized $\Hardy_2$ norm in the Hardy space of matrix functions that are analytic in the open right half-plane. This leads to a significant improvement in the performance of the \textsf{VF} approximation and also sets the stage for a new post-processing stage proposed in \S \ref{sec_rank1_VFMIMO}, where we address the question of computing an approximation having predetermined McMillan degree. 
% 
% \begin{itemize}
%\item VFMIMO vs vecfit3 comparison, emphasize the discrete LS measure
%\begin{itemize}
%\item Less computation
%\item QR with pivoting 
%\item Scaling 
%\end{itemize}
%\item Do not ignore the denominator (stopping criterion???)
%\item Order revealing via VFMIMO; scaling of noise masks this
%\item Order $rm$ to $r$ reduction
%\end{itemize}

%
%\marginpar{\raggedright\textcolor{blue}{\tiny Check changes\\ to paragraph\\ $\longleftarrow$   \\} } 
%

Rational data fitting has a long history (going back at least to Kalman \cite{Kalman-1958}) and continues to be studied in a variety of settings: For example, Gonnet, Pach\'{o}n, and Trefethen \cite{gonnet2011robust} recently provided a robust approach to rational approximation through linearized \textsf{LS} problems on the unit disk.  Hokanson in \cite{hokanson2013numerically} presented a detailed study of exponential fitting, which is closely related to rational approximation. The Loewner framework developed by Mayo and Antoulas \cite{Mayo-2007,lefteriu2010new} is an effective and numerically efficient method to 
construct rational interpolants directly from  measurements. This approach also has been extended successfully to parametric
\cite{AIL11,IonitaAntoulas2014}  and weakly nonlinear  \cite{Antoulas2014,Antoulas2015} problems.
In this paper, we focus solely on matrix-valued rational least-squares approximation by \textsf{VF}.

\section{Background and Problem Setting}\label{S=Background}
%\section{MIMO Vector Fitting -- derivation of the algorithm and the blueprints}
   
\textsf{VF} iteration is built upon the Sanathanan--Koerner (\textsf{SK}) procedure \cite{Sanathanan-Koerner-1963}, which we describe briefly as follows: 
%assume that after the $k$th step, we have a rational matrix approximant with the form 
%$\bfH_r^{(k)}(s)=\bfN^{(k)}(s) / d^{(k)}(s)$, where $\bfN^{(k)}(s)$ is a $p\times m$ matrix,  $d^{(k)}(s)$ is a scalar function, and both  $\bfN^{(k)}$  and $d^{(k)}$ have an affine dependence on parameters described in more detail below. 
A sequence of approximations having the form $\bfH_r^{(k)}=\bfN^{(k)}(s) / d^{(k)}(s) \in \mathcal{R}_r$, is developed by
taking $d^{(0)}(s)\equiv 1$, and then solving successively for 
 $\bfN^{(k+1)}$ and $d^{(k+1)}$ via the weighted linear least squares problem:
\begin{equation}\label{mimoSK}
\epsilon^{(k)} = \displaystyle \sum_{i=1}^\ell \frac{\qw_i}{|d^{(k)}(\xi_i)|^2}\left\| {\bfN^{(k+1)}(\xi_i) - d^{(k+1)}(\xi_i)\bfH(\xi_i)} \right\|_F^2 \longrightarrow \min.\;\ \mbox{ for } k=0,1,2,\ldots
\end{equation}
Since $\bfN^{(k)}$ and $d^{(k)}$ have a presumed affine dependence on parameters, (\ref{mimoSK}) does indeed define a linear least squares problem with respect to those parameters. 
%Once $\bfN^{(k+1)}(s)$ and $d^{(k+1)}(s)$ have been computed, only $d^{(k+1)}(s)$ is 
%used in the next iteration, as a new weighting function. If the iterations converge in the sense that, at some $k_*$, $d^{(k_*+1)}(s)$ "is close" to $d^{(k_*)}(s)$, 
%the process is halted and we take $\bfH_r(s) = \bfN^{(k_*+1)}(s)/d^{(k_*+1)}(s)$. Note that before this last step ($k<k_*$) any effort to compute the $\bfN^{(k+1)}$s is wasted. 
%
%, starting with $k=0$ and $d^{(0)}(s)\equiv 1$, solve successively the weighted linear least squares problems for $\bfN^{(k+1)}$ and $d^{(k+1)}$: 
%\begin{equation}\label{SK}
%\epsilon^{(k)} = \displaystyle \sum_{i=1}^\ell  \qw_i\left\| \frac{\bfN^{(k+1)}(\xi_i) - d^{(k+1)}(\xi_i)\bfH(\xi_i)}{d^{(k)}(\xi_i)} \right\|_F^2 \longrightarrow \min.\;\ k=0,1,2,\ldots
%\end{equation}
%\begin{remark}\label{REM-SK-iterations}
%{\em
Notice also that once $\bfN^{(k+1)}(s)$ and $d^{(k+1)}(s)$ have been computed, only $d^{(k+1)}(s)$ is 
used in the next iteration as a new weighting function. 
Although convergence of this process remains an open question,  
if the iterations do converge at least 
in the sense that, at some $k_*$, $d^{(k_*+1)}(s)$ ``is close to" $d^{(k_*)}(s)$ at the sample points $\xi_i$, 
the process is halted and we take $\bfH_r(s) = \bfN^{(k_*+1)}(s)/d^{(k_*+1)}(s)$.  
Note also that before this last step (i.e., for $k<k_*$), any effort to compute the value of $\bfN^{(k+1)}$ is wasted. 

The error in (\ref{mimoSK}) may be rewritten as
$
\epsilon^{(k_*)} =  \sum_{i=1}^\ell  \qw_i \left| d^{(k_*+1)}(\xi_i)/d^{(k_*)}(\xi_i)\right|^2\left\| {\bfH_r(\xi_i) - \bfH(\xi_i)}\right\|_F^2 , 
$
which corresponds to (\ref{basicOptProb}) up to an error that depends on the deviation of $|d^{(k_*+1)}/d^{(k_*)}|$ from $1$. 
This deviation becomes small as convergence occurs and can be associated with stopping criteria that we introduce 
in \S \ref{SS=StopCrit}.
%stopping criterion used to declare (numerical) convergence. 
%If the iteration index reaches the maximal allowed number of iterations without declaring numerical convergence, 
If the iteration is halted prematurely, then the last step leaves
$d^{(k)}$ unchanged and redefines $\bfN^{(k+1)}$ as
the solution to the \textsf{LS} problem  $\sum_{i=1}^\ell \qw_i \| \bfN^{(k+1)}(\xi_i) /d^{(k)}(\xi_i) - \bfH(\xi_i)\|_F^2\rightarrow\min$.
The implementation of this procedure depends on specific choices for the parameterization of $\bfN^{(k)}(s)$
and $d^{(k)}(s)$ used in representing $\bfH_r^{(k)}(s)$; barycentric representations will offer clear advantage. 

\subsection{\textsf{VF}=\textsf{SK}+barycentric representation}
Suppose that at the $k$th iteration step, we choose a set of $r$ mutually distinct (but otherwise arbitrary) nodes, 
$\{\pol_j^{(k)}\}_{j=1}^r$.  Consider a barycentric representation for $\bfH_r^{(k)}(s)$:
\begin{equation}\label{eq:H_r(k)}
\bfH_r^{(k)}(s) = \frac{\bfN^{(k)}(s)}{d^{(k)}(s)} \equiv \frac{\sum_{j=1}^r \respm_j^{(k)}/(s-\pol_j^{(k)})}{1 + \sum_{j=1}^r \dres_j^{(k)}/(s-\pol_j^{(k)})},\;\;
\respm_j^{(k)}\in\Cplx^{p\times m},\;\; \dres_j^{(k)}, \pol_j^{(k)} \in\Cplx.
\end{equation}
Observe that if $\dres_j^{(k)}\neq 0$, $\bfH_r^{(k)}(\pol_j^{(k)})={\respm_j^{(k)}}/{\dres_j^{(k)}}$. 
On the other hand, if $\dres_j^{(k)}= 0$ then $\bfH_r^{(k)}(s)$ has a simple pole at $\pol_j^{(k)}$ with an associated residue given explicitly in (\ref{eq:residues}) of Proposition \ref{PROP1}.
%In the \textsf{SK} iterations framework (with both the polynomial and with the barycentric representation),
%the \textsf{LS} coefficient matrix changes only by changing the weighting factors $1/|d^{(k)}(\xi_i)|^2$ -- in that spirit

 A new approximant is sought having the form, 
$$
\bfH_r^{(k+1)}(s) =  \frac{\sum_{j=1}^r \widehat{\respm}_j^{(k+1)}/(s-\pol_j^{(k)})}{1 + \sum_{j=1}^r \widehat{\dres}_j^{(k+1)}/(s-\pol_j^{(k)})},\;\;
\widehat{\respm}_j^{(k+1)}\in\Cplx^{p\times m},\;\; \widehat{\dres}_j^{(k+1)} \in\Cplx .
$$
%(Note that we use, for the moment, 
%a barycentric representation with respect to the previous nodes, $\pol_j^{(k)}$, $j=1,\ldots, r$).
The weighted least squares error as described in (\ref{mimoSK}) is written explicitly as
\begin{equation}\label{eq:SK-baryc}
\epsilon^{(k)} = \sum_{i=1}^\ell \frac{\qw_i}{|d^{(k)}(\xi_i)|^2} \left\| \sum_{j=1}^r \frac{\widehat{\respm}_j^{(k+1)}}{\xi_i-\pol_j^{(k)}} - \bfH(\xi_i) \left(1 + \sum_{j=1}^r \frac{\widehat{\dres}_j^{(k+1)}}{\xi_i-\pol_j^{(k)}}\right)\right\|_F^2 .
\end{equation}

\noindent One of the distinguishing features of the Gustavsen-Semlyen \textsf{VF} method \cite{Gustavsen-Semlyen-1999} emerges at this point: the explicit weighting factors $1/|d^{(k)}(\xi_i)|^2$ are eliminated through  ``pole relocation", i.e.,  the nodes in the barycentric representation are changed in such a way so as to absorb the weighting factors.
% %\footnote{Check whether this has preconditioning effects!} 
Note that if $\{\pol_j^{(k+1)}\}$ are the zeros of $d^{(k)}(s)$, then we can write 
${\displaystyle d^{(k)}(s) = {\prod_{q=1}^r (s-\pol_q^{(k+1)})}/{\prod_{q=1}^r (s-\pol_q^{(k)})}}$, and then 
%&=& \left\{ \mbox{let}\;\; 1 + \sum_{j=1}^r \frac{\dres_j^{(k+1)}}{s-\pol_j^{(k)}}
%= \frac{p}{\prod_{q=1}^r (s-\pol_j^{(k)})} \right\} \\
\begin{eqnarray}
&& \sum_{i=1}^\ell \!\qw_i \!\left| \frac{{\displaystyle \prod_{q=1}^r (\xi_i-\pol_q^{(k)})}}{{\displaystyle \prod_{q=1}^r (\xi_i-\pol_q^{(k+1)})}}\right|^2 \!\left\| \!
\sum_{j=1}^r \frac{{\displaystyle \prod_{q\neq j} (\xi_i-\pol_q^{(k)})\widehat{\respm}_j^{(k+1)}}}{\prod_{q=1}^r (\xi_i-\pol_q^{(k)})} - \bfH(\xi_i) \frac{{\displaystyle \prod_{q=1}^r (\xi_i-\pol_q^{(k)}) + \sum_{j=1}^r
\widehat{\dres}_j^{(k+1)} \prod_{q\neq j} (\xi_i-\pol_q^{(k)}) }}{\prod_{q=1}^r (\xi_i-\pol_q^{(k)})} \right\|_F^2 \nonumber \\
&&= \sum_{i=1}^\ell \qw_i\left\| 
\sum_{j=1}^r \frac{{\displaystyle \prod_{q\neq j} (\xi_i-\pol_q^{(k)})\widehat{\respm}_j^{(k+1)}}}{\prod_{q=1}^r (\xi_i-\pol_q^{(k+1)})} - \bfH(\xi_i) \frac{{\displaystyle \prod_{q=1}^r (\xi_i-\pol_q^{(k)}) + \sum_{j=1}^r
\widehat{\dres}_j^{(k+1)} \prod_{q\neq j} (\xi_i-\pol_q^{(k)}) }}{\prod_{q=1}^r (\xi_i-\pol_q^{(k+1)})} \right\|_F^2 \equiv
\epsilon^{(k)}.
\end{eqnarray}
The next iterate, $\bfH_r^{(k+1)}$, will be represented in barycentric form using nodes $\pol_q^{(k+1)}$, $q=1,\ldots,r$.
We assume simple zeros for simplicity.  We introduce new variables, ${\dres}_j^{(k+1)}$,  so that
$$
\frac{{\displaystyle \prod_{q=1}^r (s-\pol_q^{(k)}) + \sum_{j=1}^r
\widehat{\dres}_j^{(k+1)} \prod_{q\neq j} (s-\pol_q^{(k)}) }}{\prod_{q=1}^r (s-\pol_q^{(k+1)})} = 1 + \sum_{j=1}^r \frac{{\dres}_j^{(k+1)}}{s-\pol_j^{(k+1)}}.
$$
 In the same way, we introduce new unknowns,
${\respm}_j^{(k+1)}$, so that
%\footnote{\textcolor{red}{Is this considered obvious or an argumentation is needed?}} 
\begin{equation}\label{(*)}
\sum_{j=1}^r \frac{{\displaystyle \prod_{q\neq j} (s-\pol_q^{(k)})\widehat{\respm}_j^{(k+1)}}}{\prod_{q=1}^r (s-\pol_q^{(k+1)})} = 
\sum_{j=1}^r \frac{{\respm}_j^{(k+1)}}{s-\pol_j^{(k+1)}} .
\end{equation}
After this change of variables, the \textsf{LS} objective function from (\ref{eq:SK-baryc}) appears as
\begin{equation}\label{eq:VF-LS-problem}
\epsilon^{(k)} = \sum_{i=1}^\ell \qw_i\left\| \sum_{j=1}^r \frac{{\respm}_j^{(k+1)}}{\xi_i-\pol_j^{(k+1)}} - \bfH(\xi_i) \left(1 + \sum_{j=1}^r \frac{{\dres}_j^{(k+1)}}{\xi_i-\pol_j^{(k+1)}}\right)\right\|_F^2.
\end{equation}
%In general, $\bfH(\xi_i)$  will be measured in an experiment or computed from a mathematical model
%and may not be known exactly.   
%In such cases, we would use a sampled value, $\bfS^{(i)}=\bfH(\xi_i)+\bfcalE^{(i)}$, instead, taking into account 
%%\begin{equation}\label{eq:Si}
%%S^{(i)} = \left(\begin{smallmatrix} S^{(i)}_{11} & \ldots & S^{(i)}_{1m} \cr
%%\vdots & \ddots & \vdots\cr S^{(i)}_{p1} & \ldots & S^{(i)}_{pm}
%% \end{smallmatrix}\right)  = \left(\begin{smallmatrix} \bfH_{11}(\xi_i) & \ldots & \bfH_{1m}(\xi_i) \cr \vdots & \ddots & \vdots\cr \bfH_{p1}(\xi_i) & \ldots & \bfH_{pm}(\xi_i)\end{smallmatrix}\right) + \mathcal{E}^{(i)} .
%%\end{equation} 
%In some statistical information on the noise $\mathcal{E}^{(i)}$ is available. In this paper we consider the noise-free case, $S^{(i)}=\bfH(\xi_i)$.

Now, the values for $\{\respm_j^{(k+1)}\}_{j=1}^r$ and $\{ \dres_j^{(k+1)}\}_{j=1}^r$ that minimize $\epsilon^{(k)}$ in (\ref{eq:VF-LS-problem}) will determine the next iterate $\bfH_r^{(k+1)}(s)$.  The iteration continues by defining 
$d^{(k+1)}(s)=1+\sum_{j=1}^r{{\dres}_j^{(k+1)}}/{(s-\pol_j^{(k+1)})}$ 
and new poles $\pol_j^{(k+2)}$ will be the computed zeros of $d^{(k+1)}(s)$.

%
%\marginpar{\raggedright\textcolor{blue}{\tiny  Is convergence\\ criterion in the \\ %\textsf{VF} literature ? \\ Include a citation ? \\} } 
%
{
Numerical convergence is declared and the iteration terminates at an index $k_*$ when %$d^{(k_*)}(s)\approx 1$,  i.e., when 
$\max_j |{\dres}_j^{(k_*)}|$ is ``small enough". 
%In that case, there is a matching permutation $\pi$ such that 
%$\lambda_{\pi(j)}^{(k+1)} \approx \lambda_j^{(k+2)}$, $j=1,\ldots ,r$. 
The zeros of $d^{(k_*)}(s)$ are extracted as the eigenvalues of the matrix $\mathrm{diag}(\pol_j^{(k_*)})_{j=1}^r + (1,\cdots, 1)^T ({\dres}_1^{(k_*)}, \cdots, {\dres}_r^{(k_*)})$ and assigned to $\{\pol_j^{(k_*+1)}\}_{j=1}^r$.} 
The denominator $d^{(k_*)}(s)$ is set to the 
constant $1$ {indicating numerical pole--zero cancellation,} and the \textsf{LS} problem (\ref{eq:VF-LS-problem}) is solved 
for $\respm_j^{(k_*)}$ (assigning $k+1=k_*$) with 
all $\dres_j^{(k_*)}$ terms replaced with zeros.
The resulting \textsf{VF} approximant is
\begin{equation}\label{eq:H_r(k)-final}
\bfH_r(s) = \sum_{j=1}^r \frac{\respm_j^{(k_*)}}{s-\pol_j^{(k_*+1)}}.
%,\;\;
%\respm_j^{(k)}\in\Cplx^{p\times m},\;\; \pol_j^{(k+1)} \in\Cplx, \;\;j=1,\ldots, r.
\end{equation}
%is computed by solving (\ref{eq:VF-LS-problem})
%for $\respm_j^{(k+1)}$ with all $\varphi_j^{(k+1)}$ set to zero. 
It is usually reported that numerical convergence takes place within a few iterations, provided that initial poles are well chosen.   However in difficult cases, numerical convergence 
may not be achieved, and the iterate described in (\ref{eq:H_r(k)}) may have a denominator that is far from constant. 
Even if the stopping criterion is not satisfied, the above procedure that yields (\ref{eq:H_r(k)-final}) will be valid at any $k$, and may be interpreted as taking the last computed barycentric nodes, viewing them as poles, and then computing a best least-squares rational approximation in pole-residue form.  
We discuss this closing step of the iteration further at the end of \S \ref{S=LSmimo-solution}.
The stopping criterion will be discussed in \S \ref{SS=StopCrit}.

\subsection{Pole-residue form of barycentric approximant}\label{S=BC2PR}
The barycentric representation offers many advantages for both the \textsf{VF} iteration and the \textsf{SK} iteration (i.e., with the explicit weighting by $1/|d^{(k)}(\xi_i)|^2$ as in (\ref{eq:SK-baryc})). 
Nonetheless, the pole-residue representation is more convenient for analysis and further usage. For example, a pole-residue representation of a (rational) transfer function leads immediately to a state space realization of the underlying dynamical system.  Furthermore, we may wish to inspect, in the course of iterations, an
approximant (\ref{eq:H_r(k)}) with the goal of estimating the importance of the contribution of certain of its poles and the possible effect of discarding them.  
Hence, it would be useful in general to have an efficient procedure to transform a barycentric representation of a function into a pole-residue representation. 
Toward that end, consider a $p\times m$ matrix rational function 
\begin{eqnarray} \label{eq:G_for_pr}
\mathbf{G}(s) = \frac{\sum_{j=1}^{r}\frac{\bfPhi_j}{s-\lambda_j}}
{1+\sum_{j=1}^{r}\frac{\varphi_j}{s-\lambda_j}},\;\; \bfPhi_j\in\Cplx^{p\times m},\;\;\varphi_j, \lambda_j\in\Cplx, \;\;
\mbox{and}\;\; |\varphi_j| + \|\bfPhi_j\|_F > 0.
\end{eqnarray}
 Assume that the barycentric nodes, 
$\lambda_1,\ldots,\lambda_r$ are distinct and closed under conjugation (non-real values appear in complex conjugate pairs), and that the corresponding $\varphi_j$s and $\bfPhi_j$s have a compatible conjugation symmetry ($\varphi_j$, $\bfPhi_j$ real if $\lambda_j$ real, $\varphi_i=\cc{\varphi_j}$, $\bfPhi_i=\cc{\bfPhi_j}$ if $\lambda_i=\cc{\lambda_j}$). 
Define complementary index sets: 
$$
\mathcal{J}_0 = \{ j\; : \; \varphi_j=0\}\quad \mbox{and}\quad \mathcal{J}_1= \{ j\; : \; \varphi_j \neq 0\}.
$$ 
Then $\{\lambda_j \; :\; j\in\mathcal{J}_k\}$, $\{\varphi_j \; :\; j\in\mathcal{J}_k\}$ are closed under complex conjugation for $k=0,1$. Note that 
$\{\lambda_j\; :\; j\in\mathcal{J}_0\}$ are among the poles of $\mathbf{G}(s)$.
The remaining poles can be indexed as 
$\{\widehat{\lambda}_j \; :\; j\in\mathcal{J}_1 \}$. 

\begin{proposition}\label{PROP1}
	Assume, in addition to the above,  that all poles of $\mathbf{G}(s)$ in (\ref{eq:G_for_pr}) are
	simple.\footnote{The general case of multiple poles is just more technical and it follows by standard residue calculus.} Then,
	$$
	\mathbf{G}(s) = \sum_{j\in\mathcal{J}_1} \frac{\widehat{{{R}_j}}}{s-\widehat{\lambda}_j} + 
	\sum_{j\in\mathcal{J}_0} \frac{{{{R}_j}}}{s-{\lambda}_j} , %\;\;\mbox{where}
	$$
	where the residues can be efficiently calculated as 
	\begin{equation}\label{eq:residues}
	R_j = \frac{\bfPhi_j}{1+\sum_{i\in\mathcal{J}_1} \varphi_i/(\lambda_j-\lambda_i)},\;\;j\in\mathcal{J}_0 ; \;\;\;\;
	\widehat{R}_j = \frac{\prod_{i\in\mathcal{J}_1} (\widehat{\lambda}_j-\lambda_i)}{\prod_{i\neq j}(\widehat{\lambda}_j-\widehat{\lambda}_i)} \sum_{i=1}^r \frac{\bfPhi_i}{\widehat{\lambda}_j -\lambda_i},\;\; j\in\mathcal{J}_1 .
	\end{equation}

\end{proposition}
\begin{proof}
The proof immediately follows from a calculation of residues, e.g. for $j\in\mathcal{J}_0$, $R_j=\lim_{s\rightarrow\lambda_j}(s-\lambda_j)\mathbf{G}(s)$. 	
\end{proof}

\noindent The first formula in (\ref{eq:residues}) (for $j\in\mathcal{J}_0$) is a special case of the second one, which is actually given in (\ref{(*)}).
This reflects the fact that the change of variables (\ref{(*)}) in the transition from \textsf{SK} to \textsf{VF} iterations implicitly seeks a pole-residue representation.

%The final approximant in (\ref{eq:H_r(k)-final}) is constructed under the assumption that the iterations have converged in the sense that the denominators $d^{(k)}$ have 
%approached the constant unity. 
%
%... we now justify it .... 

\subsection{Computing $\respm_j^{(k+1)}$ and $ \dres_j^{(k+1)}$}\label{S=LSmimo-solution}
Measurements are naturally kept in a tensor $\mathbb{S}\in\Cplx^{p\times m\times\ell}$ 
with $\mathbb{S}(:,:,i)=\bfS^{(i)}=\bfH(\xi_i)+ \bfcE ^{(i)}\in\Cplx^{p\times m}$.  
$\bfS^{(i)}$ denotes a sampling of the transfer function at the node $\xi_i$, allowing also for measurement errors to be
represented through $\bfcE ^{(i)}$.  Notice that the 
%complementary $\mathbb{S}$-slice,
{mode-$3$ fiber}
$\mathbb{S}(u,v,:)=\begin{pmatrix} 
S^{(1)}_{uv},  S^{(2)}_{uv}, \ldots , S^{(\ell-1)}_{uv}, S^{(\ell)}_{uv}
\end{pmatrix}^T\in\Cplx^\ell$, represents connections between input $v$ and output $u$ over all frequency samples. 

The expression for $\epsilon^{(k)}$ can be matricized in several natural ways. The one followed in the \textsf{VF} literature is derived from a point-wise matching of input-output relations over all measurements. This is natural as it decouples the problem into $p\cdot m$ rational function approximations having a common set of poles. In terms of the matrix entries, the least squares error (\ref{eq:VF-LS-problem}) can be written as
%
%\marginpar{\raggedleft\textcolor{blue}{\scriptsize \\  Changed\\ $F-$norm\\ to %$2-$norm \\ $\longrightarrow$
%  \\} } 
%
\begin{eqnarray}
\epsilon^{(k)} &=& \sum_{i=1}^\ell \qw_i \sum_{u=1}^p \sum_{v=1}^m \left| \sum_{j=1}^r
\left( \frac{{(\respm}_j^{(k+1)})_{uv}}{\xi_i-\pol_j^{(k+1)}} - S^{(i)}_{uv}\frac{{\dres}_j^{(k+1)}}{\xi_i-\pol_j^{(k+1)}} \right) - S^{(i)}_{uv}\right|^2  \nonumber \\
%&=& 
%\sum_{v=1}^m \sum_{u=1}^p  \sum_{i=1}^\ell\qw_i \left| \sum_{j=1}^r
%\left( \frac{{(\respm}_j^{(k+1)})_{uv}}{\xi_i-\pol_j^{(k+1)}} - (S^{(i)})_{uv}\frac{{\dres}_j^{(k+1)}}{\xi_i-\pol_j^{(k+1)}} \right) - (S^{(i)})_{uv}\right|^2 \\
&=& \sum_{v=1}^m \sum_{u=1}^p \left\| \dqw \begin{pmatrix} \Cauchy^{(k+1)}, & - D^{(uv)}{\Cauchy}^{(k+1)}\end{pmatrix} \begin{pmatrix}\varPhi^{(k+1)}(u,v,:) \cr \bfvarphi^{(k+1)} \end{pmatrix} -\dqw \mathbb{S}(u,v,:) \right\|_2^2 , \label{eq:LS1}
%\left| \begin{pmatrix} {\displaystyle \frac{1}{\xi_i-\pol_1^{(k+1)}}} & \ldots & {\displaystyle \frac{1}{\xi_i-\pol_r^{(k+1)}}}\end{pmatrix}   \right|^2 , 
\end{eqnarray}
%
%\marginpar{\raggedleft\textcolor{blue}{\tiny Changed from \\ $\dqw = %\mathrm{diag}(\qw_i)$ 
%  \\} } 
%
where {$\dqw = \mathrm{diag}(\sqrt{\qw_i})$  } , 
$D^{(uv)}=\mathrm{diag}(S^{(i)}_{uv})_{i=1}^\ell$, $\bfvarphi^{(k+1)}=(\dres^{(k+1)}_1,\ldots,\dres^{(k+1)}_r)^T$,  and
\begin{equation} \label{eq:A=PC}
\Cauchy^{(k+1)} = \left(\begin{smallmatrix}
\frac{1}{\xi_1-\pol_1^{(k+1)}} & \frac{1}{\xi_1-\pol_2^{(k+1)}} & \cdots & \frac{1}{\xi_1-\pol_r^{(k+1)}}  \\[0.3em]
\frac{1}{\xi_2-\pol_1^{(k+1)}} & \frac{1}{\xi_2-\pol_2^{(k+1)}} & \cdots & \frac{1}{\xi_2-\pol_r^{(k+1)}}  \cr
\vdots & \vdots & \vdots & \vdots  \cr
\frac{1}{\xi_{\ell-1}-\pol_1^{(k+1)}} & \frac{1}{\xi_{\ell-1}-\pol_2^{(k+1)}} & \cdots & \frac{1}{\xi_{\ell-1}-\pol_r^{(k+1)}} \cr
\frac{1}{\xi_\ell-\pol_1^{(k+1)}} & \frac{1}{\xi_\ell-\pol_2^{(k+1)}} & \cdots & \frac{1}{\xi_\ell-\pol_r^{(k+1)}} 
\end{smallmatrix}\right),\;\;
%\mathbb{S}(u,v,:)=\left( \begin{smallmatrix} 
%S^{(1)}_{uv} \cr S^{(2)}_{uv} \cr \vdots \cr S^{(\ell-1)}_{uv} \cr S^{(\ell)}_{uv}
%\end{smallmatrix}\right),\;\;
\varPhi^{(k+1)}(u,v,:)=\left( \begin{smallmatrix} 
(\Phi_1^{(k+1)})_{uv} \cr (\Phi_2^{(k+1)})_{uv} \cr \vdots \cr (\Phi_{r-1}^{(k+1)})_{uv} \cr (\Phi_r^{(k+1)})_{uv}
\end{smallmatrix}
\right) .
\end{equation}
To minimize (\ref{eq:LS1}), it is convenient to introduce the QR factorizations for
$1\leq u \leq p$, $1\leq v\leq m$:
\begin{align}\label{eq:concurrent-QRs}
\dqw
\begin{pmatrix} \Cauchy^{(k+1)},  - D^{(uv)}{\Cauchy}^{(k+1)}\end{pmatrix} & = \!\!
\left( \begin{smallmatrix}
\div & \div & \div & \divideontimes & \divideontimes & \divideontimes \\
\div & \div & \div & \divideontimes & \divideontimes & \divideontimes \\
\div & \div & \div & \divideontimes & \divideontimes & \divideontimes \\
\div & \div & \div & \divideontimes & \divideontimes & \divideontimes \\
\div & \div & \div & \divideontimes & \divideontimes & \divideontimes \\
\div & \div & \div & \divideontimes & \divideontimes & \divideontimes \\
 \div & \div & \div & \divideontimes & \divideontimes & \divideontimes \\
 \end{smallmatrix}\right)\!\!
% = & Q^{(k+1)}_{uv} \!\begin{pmatrix} (R^{(k+1)}_{uv})_{11} & (R^{(k+1)}_{uv})_{12}\cr
 =  Q^{(k+1)}_{uv} \!\begin{pmatrix} (R^{(k+1)})_{11} & (R^{(k+1)}_{uv})_{12}\cr
 0 & (R^{(k+1)}_{uv})_{22} \cr 0 & 0 \end{pmatrix} \\
 & \!\! =\!\! 
 \begin{pmatrix} (Q^{(k+1)})_1 & (Q^{(k+1)}_{uv})_2 &  (Q^{(k+1)}_{uv})_3 \end{pmatrix}
 \left( \begin{smallmatrix}
 \ast & \ast & \ast & \times & \times & \times \\
   0  & \ast & \ast & \times & \times & \times \\
   0  &   0  & \ast & \times & \times & \times \\
   0  &   0  &   0  & \star          & \star          & \star          \\
   0  &   0  &   0  &       0        & \star          & \star \\
   0  &   0  &   0  &       0        &       0        & \star   \\
   0  &   0  &   0  &       0        &       0        &        0       \\[-.5em]
    \mbox{\small{$\vdots$}}  &  & & \mbox{\small{$\vdots$}}     &       &   \mbox{\small{$\vdots$}}        \\
   0  &   0  &   0  &       0        &       0        &        0       \\
  \end{smallmatrix}\right),
\end{align}
where the unitary matrix $Q^{(k+1)}_{uv}$ has been partitioned as 
$Q^{(k+1)}_{uv} = \begin{pmatrix} (Q^{(k+1)})_1 & (Q^{(k+1)}_{uv})_2 &  (Q^{(k+1)}_{uv})_3 \end{pmatrix}$, with block columns of sizes $\ell\times r$,
$\ell\times r$, $\ell\times (\ell-2r)$, respectively.  The leading $r$ columns of (\ref{eq:concurrent-QRs})
are \emph{independent} of $(u,v)$ hence the initial part of the factorization, 
$(Q^{(k+1)})_1\,(R^{(k+1)})_{11} = \dqw \Cauchy^{(k+1)}$, 
 need only be done once.  

The \textsf{LS} residual norm may be decomposed as 
%\begin{eqnarray}  \label{eq:eps}
$\epsilon^{(k)}  = \epsilon^{(k)}_{1} + \epsilon^{(k)}_{2} + \epsilon^{(k)}_{3}$,
%\end{eqnarray}
where 
\begin{eqnarray}
\epsilon_1^{(k)} &=& \sum_{v=1}^m \sum_{u=1}^p \left\| (R^{(k+1)})_{11}\varPhi^{(k+1)}(u,v,:) + (R^{(k+1)}_{uv})_{12} \bfvarphi^{(k+1)} - (Q^{(k+1)})_1^* \dqw\mathbb{S}(u,v,:)\right\|_2^2 \nonumber\\
\epsilon_2^{(k)}
&=& \sum_{v=1}^m \sum_{u=1}^p \left\| (R^{(k+1)}_{uv})_{22}\bfvarphi^{(k+1)} -  (Q^{(k+1)}_{uv})_2^* \dqw\mathbb{S}(u,v,:) \right\|_2^2,\qquad\mbox{and} \label{eq:eps123}\\
 \epsilon_3^{(k)} & =  & \sum_{v=1}^m \sum_{u=1}^p \left\| (Q^{(k+1)}_{uv})_3^* \dqw\mathbb{S}(u,v,:)\right\|_2^2, \nonumber
\end{eqnarray}
where we have used $M^*$ to denote the conjugate transpose of a matrix $M$.
Here, $\epsilon^{(k)}_{3}$ is a part of the residual that is beyond the reach of the unknowns $\respm_j^{(k+1)}, \dres_j^{(k+1)}$ -- it corresponds to the component in the data that is orthogonal to the subspace of rational functions available with the current barycentric nodes.
Only a set of new (better) barycentric nodes will incline this subspace toward the data in such a way as to reduce this component of the error.  Extracting optimal information from a given subspace associated with the current barycentric nodes is achieved by minimizing  $\epsilon^{(k)}_{1} + \epsilon^{(k)}_{2}$.
%\begin{eqnarray}
%\epsilon^{(k)} &=& \sum_{v=1}^m \sum_{u=1}^p \left\| (R^{(k+1)}_{uv})_{11}\varPhi^{(k+1)}(u,v,:) + (R^{(k+1)}_{uv})_{12} \bfvarphi^{(k+1)} - (Q^{(k+1)}_{uv})_1^* \dqw\mathbb{S}(u,v,:)\right\|_F^2 + \nonumber\\
%&+& \sum_{v=1}^m \sum_{u=1}^p \left\| (R^{(k+1)}_{uv})_{22}\bfvarphi^{(k+1)} -  (Q^{(k+1)}_{uv})_2^* \dqw\mathbb{S}(u,v,:) \right\|_F^2 + \sum_{v=1}^m \sum_{u=1}^p \left\| (Q^{(k+1)}_{uv})_3^* \dqw\mathbb{S}(u,v,:)\right\|_F^2 \nonumber\\
%&\equiv& \epsilon^{(k)}_{1} + \epsilon^{(k)}_{2} + \epsilon^{(k)}_{3}.\label{eq:eps123}
%\end{eqnarray}
%Here, $\epsilon^{(k)}_{3}$ is a part of the residual that is beyond the reach of the unknowns $\respm_j^{(k+1)}, \dres_j^{(k+1)}$ -- it corresponds to the component in the data that is orthogonal to the subspace of rational functions available with the current poles.
%Only a set of new, better, poles can incline this subspace toward the data and thus reduce this component of the error. Extracting optimal information from a given subspace is achieved by minimizing  $\epsilon^{(k)}_{1} + \epsilon^{(k)}_{2}$.

To that end, first note that we can assume that $(R^{(k+1)})_{11}$ is nonsingular since it participates in a QR factorization of  $\dqw\Cauchy^{(k+1)}$, which in turn must have full column rank since $\{\pol_j^{(k)}\}_{j=1}^r$ are presumed distinct and we (tacitly) assume that $\xi_i\neq \pol_j^{(k)}$ throughout the iteration. 
%Then $(R^{(k+1)}_{uv})_{11}$
% is the triangular factor in a QR factorization of the full column-rank matrix $\Cauchy^{(k+1)}$.  
 Thus, we can make $\epsilon^{(k)}_{1}$ exactly zero for any choice of
$\dres^{(k+1)}$ by taking, concurrently for all input-output pairs $(u,v)$, with $1\leq u\leq p$ and $1\leq v\leq m$, 
\begin{equation}\label{eq:PHIuv}
\varPhi^{(k+1)}(u,v,:) = (R^{(k+1)})_{11}^{-1} \left((Q^{(k+1)})_1^*\dqw \mathbb{S}(u,v,:) -(R^{(k+1)}_{uv})_{12} \bfvarphi^{(k+1)} \right).
\end{equation}
When the poles of the approximant are 
assigned to the barycentric nodes, $\{\pol_j^{(k+1)}\}_{j=1}^r$, then
the residues in (\ref{eq:H_r(k)-final}) may be found by solving (\ref{eq:PHIuv}) with $\bfvarphi^{(k+1)}=0$.
For other cases, the optimal $\bfvarphi^{(k+1)}$ will be found by minimizing $\epsilon_k^{(2)}$, and this process is uncoupled from any information about $\varPhi^{(k+1)}$. However, determining 
$\bfvarphi^{(k+1)}$ does involve a coupled
 minimization across all input-to-output $(u,v)$ pairs.  
{The computation (\ref{eq:PHIuv}) is skipped if we choose to proceed with the \textsf{SK} iterations.}

Once the maximal number of iterations is reached without reducing $\bfvarphi^{(k)}$
enough to be neglected, then $\bfvarphi^{(k+1)}$ in (\ref{eq:PHIuv}) also cannot be 
%
%\marginpar{\raggedleft\textcolor{blue}{\tiny  This still is not \\ clear to me...\\ 
%Will this produce \\ a smaller \textsf{LS}\\ residual ?  \\} } 
%
neglected without losing information. That is why the approximant defined in
(\ref{eq:H_r(k)-final}) uses the poles $\lambda_j^{(k_*+1)}$. It follows from \S \ref{S=BC2PR}
that this is equivalent to solving (\ref{eq:PHIuv}) with $\bfvarphi^{(k_*)}$ 
included in the right-hand sides and transforming the computed rational approximant 
from barycentric into pole-residue form. This is another elegant feature built into the Gustavsen-Semlyen \textsf{VF} framework. {(If we use the original \textsf{SK} iterations with diagonal scalings and fixed poles, then the barycentric form can be transformed to pole-residue representation using Proposition \ref{PROP1}.)  }

% 
%\begin{remark}
%{\em
%\cite{Deschrijver-Mrozowski-Dhaene-Zutter-2008}
%}
%\end{remark}
\subsection{Global structure}
This procedure can be represented as $\|\mathbf{A}^{(k+1)}x-\mathbf{b}\|_2 \rightarrow \min$, in the usual form, as follows. First, we specify that the indices $(u,v)$ in the $p\times m$ array will be vectorized in a column-by-column fashion, $(u,v)\leadsto \iota_{uv} = p(v-1) + u$. Define $\mathbf{A}^{(k+1)}$ to be a block matrix, with $pm\times (pm+1)$ block structure, each block of dimensions $\ell\times r$. Only $2 p m$ out of $pm(pm+1)$ blocks are \emph{a priori} nonzero. Initially, set $\mathbf{A}^{(k+1)}$ to zero and
update it as follows: in the block row $\iota_{uv}$ set the diagonal block to
$\dqw\Cauchy^{(k+1)}$ and the last block in  the row to $-\dqw D^{(uv)}\Cauchy^{(k+1)}$. 
The right-hand side is the vector $\mathbf{b}$ with the $pm\times 1$ block structure, where
the $\iota_{uv}$th block is set to $\dqw\mathbb{S}(u,v,:)$. The concurrent QR
factorizations (\ref{eq:concurrent-QRs}) can be then represented by pre-multiplying $\mathbf{A}^{(k+1)}$ with unitary block-diagonal matrix $(\mathcal{Q}^{(k+1)})^*$,  with the 
$\iota_{uv}$th diagonal block set to $(Q^{(k+1)}_{uv})^*$. It is easily seen that the rows of $(\mathcal{Q}^{(k+1)})^*\mathbf{A}^{(k+1)}$ can be permuted to obtain the structure illustrated in (\ref{eq:structure-partialQR}).\footnote{Elements not displayed are zeros.}

\begin{equation}\label{eq:structure-partialQR}
%A = 
\underbrace{ \left( \begin{smallmatrix}  
\div & \div & \div &  &   &   &   &  &  &  &  &    &  \divideontimes & \divideontimes & \divideontimes \\ 
\div & \div & \div &  &   &   &   &  &  &  &  &    &  \divideontimes & \divideontimes & \divideontimes \\
\div & \div & \div &  &   &   &   &  &  &  &  &    &  \divideontimes & \divideontimes & \divideontimes \\
\div & \div & \div &  &   &   &   &  &  &  &  &    &  \divideontimes & \divideontimes & \divideontimes \\
\div & \div & \div &  &   &   &   &  &  &  &  &    &  \divideontimes & \divideontimes & \divideontimes \\
\div & \div & \div &  &   &   &   &  &  &  &  &    &  \divideontimes & \divideontimes & \divideontimes \\
\div & \div & \div &  &   &   &   &  &  &  &  &    &  \divideontimes & \divideontimes & \divideontimes \\ \hline
     &      &      & \div & \div & \div &  &  &  &  &  &  & \divideontimes & \divideontimes & \divideontimes \\ 
     &      &      & \div & \div & \div &  &  &  &  &  &  & \divideontimes & \divideontimes & \divideontimes \\
     &      &      & \div & \div & \div &  &  &  &  &  &  & \divideontimes & \divideontimes & \divideontimes \\
     &      &      & \div & \div & \div &  &  &  &  &  &  & \divideontimes & \divideontimes & \divideontimes \\
     &      &      & \div & \div & \div &  &  &  &  &  &  & \divideontimes & \divideontimes & \divideontimes \\
          &      &      & \div & \div & \div &  &  &  &  &  &  & \divideontimes & \divideontimes & \divideontimes \\
          &      &      & \div & \div & \div &  &  &  &  &  &  & \divideontimes & \divideontimes & \divideontimes \\\hline 
     &  &  &  &  &  & \div & \div & \div &  &  &  & \divideontimes & \divideontimes & \divideontimes \\
     &  &  &  &  &  & \div & \div & \div &  &  &  & \divideontimes & \divideontimes & \divideontimes \\
     &  &  &  &  &  & \div & \div & \div &  &  &  & \divideontimes & \divideontimes & \divideontimes \\
     &  &  &  &  &  & \div & \div & \div &  &  &  & \divideontimes & \divideontimes & \divideontimes \\
     &  &  &  &  &  & \div & \div & \div &  &  &  & \divideontimes & \divideontimes & \divideontimes \\
      &  &  &  &  &  & \div & \div & \div &  &  &  & \divideontimes & \divideontimes & \divideontimes \\
      &  &  &  &  &  & \div & \div & \div &  &  &  & \divideontimes & \divideontimes & \divideontimes \\  \hline   
     &  &  &  &  &  &  &  &  & \div & \div & \div & \divideontimes & \divideontimes & \divideontimes \\
     &  &  &  &  &  &  &  &  & \div & \div & \div & \divideontimes & \divideontimes & \divideontimes \\
     &  &  &  &  &  &  &  &  & \div & \div & \div & \divideontimes & \divideontimes & \divideontimes \\
     &  &  &  &  &  &  &  &  & \div & \div & \div & \divideontimes & \divideontimes & \divideontimes \\
     &  &  &  &  &  &  &  &  & \div & \div & \div & \divideontimes & \divideontimes & \divideontimes \\
          &  &  &  &  &  &  &  &  & \div & \div & \div & \divideontimes & \divideontimes & \divideontimes \\
          &  &  &  &  &  &  &  &  & \div & \div & \div & \divideontimes & \divideontimes & \divideontimes \\
 \end{smallmatrix}\right)}_{\mathbf{A}^{(k+1)}, \ell=7, r=3, p=m=2},\;\; 
% % %
%Q^* A = 
\underbrace{
\left( \begin{smallmatrix}  
\ast & \ast & \ast &  &   &   &   &  &  &  &  &    &  \times & \times & \times \\ 
0    & \ast & \ast &  &   &   &   &  &  &  &  &    &  \times & \times & \times \\
0    & 0    & \ast &  &   &   &   &  &  &  &  &    &  \times & \times & \times \\
0    & 0    & 0    &  &   &   &   &  &  &  &  &    &  \star & \star & \star \\
0    & 0    & 0    &  &   &   &   &  &  &  &  &    &         0       & \star & \star \\
0    & 0    & 0    &  &   &   &   &  &  &  &  &    &         0       &        0       & \star \\
0    & 0    & 0    &  &   &   &   &  &  &  &  &    &         0       &        0       &   0    \\ \hline
 &      &      & \ast & \ast & \ast &  &  &  &  &  &  & \times & \times & \times \\ 
 &      &      & 0    & \ast & \ast &  &  &  &  &  &  & \times & \times & \times \\
 &      &      & 0    & 0    & \ast &  &  &  &  &  &  & \times & \times & \times \\
 &      &      & 0    & 0    & 0    &  &  &  &  &  &  & \star & \star & \star \\
 &      &      & 0    & 0    & 0    &  &  &  &  &  &  &      0     & \star & \star \\
 &      &      & 0    & 0    & 0    &  &  &  &  &  &  &      0     &    0       & \star \\
 &      &      & 0    & 0    & 0    &  &  &  &  &  &  &      0     &        0       &    0 \\\hline 
     &  &  &  &  &  & \ast & \ast & \ast &  &  &  & \times & \times & \times \\
     &  &  &  &  &  & 0    & \ast & \ast &  &  &  & \times & \times & \times \\
     &  &  &  &  &  & 0    & 0    & \ast &  &  &  & \times & \times & \times \\
     &  &  &  &  &  & 0    & 0    & 0    &  &  &  & \star           & \star         & \star \\
     &  &  &  &  &  & 0    & 0    & 0    &  &  &  &      0         & \star           & \star \\
      &  &  &  &  &  & 0   & 0    & 0    &  &  &  &      0         &       0        & \star \\
      &  &  &  &  &  & 0   & 0    & 0    &  &  &  &      0         &       0        &  0  \\  \hline   
     &  &  &  &  &  &  &  &  & \ast & \ast & \ast & \times & \times & \times \\
     &  &  &  &  &  &  &  &  & 0    & \ast & \ast & \times & \times & \times \\
     &  &  &  &  &  &  &  &  & 0    & 0    & \ast & \times & \times & \times \\
     &  &  &  &  &  &  &  &  & 0    & 0    & 0    & \star          & \star          & \star \\
     &  &  &  &  &  &  &  &  & 0    & 0    & 0    &  0             & \star          & \star \\
     &  &  &  &  &  &  &  &  & 0    & 0    & 0    & 0              &        0       & \star \\
     &  &  &  &  &  &  &  &  & 0    & 0    & 0    &     0          &        0       &       0        \\
 \end{smallmatrix}\right)}_{(\mathcal{Q}^{(k+1)})^* \mathbf{A}^{(k+1)}} ,
 \underbrace{ \left( \begin{smallmatrix}  
 \ast & \ast & \ast &  &   &   &   &  &  &  &  &    &  \times & \times & \times \\ 
 0    & \ast & \ast &  &   &   &   &  &  &  &  &    &  \times & \times & \times \\
 0    & 0    & \ast &  &   &   &   &  &  &  &  &    &  \times & \times & \times \\
   &      &      & \ast & \ast & \ast &  &  &  &  &  &  & \times & \times & \times \\ 
   &      &      & 0    & \ast & \ast &  &  &  &  &  &  & \times & \times & \times \\
   &      &      & 0    & 0    & \ast &  &  &  &  &  &  & \times & \times & \times \\
         &  &  &  &  &  & \ast & \ast & \ast &  &  &  & \times & \times & \times \\
         &  &  &  &  &  & 0    & \ast & \ast &  &  &  & \times & \times & \times \\
         &  &  &  &  &  & 0    & 0    & \ast &  &  &  & \times & \times & \times \\
               &  &  &  &  &  &  &  &  & \ast & \ast & \ast & \times & \times & \times \\
               &  &  &  &  &  &  &  &  & 0    & \ast & \ast & \times & \times & \times \\
               &  &  &  &  &  &  &  &  & 0    & 0    & \ast & \times & \times & \times \\ \hline\hline
      &      &      &  &   &   &   &  &  &  &  &    &  \star  & \star  & \star \\
      &      &      &  &   &   &   &  &  &  &  &    & 0       & \star  & \star \\
      &      &      &  &   &   &   &  &  &  &  &    &  0      &  0     & \star \\
  &      &      &      &      &      &  &  &  &  &  &  & \star & \star & \star \\
  &      &      &      &      &      &  &  &  &  &  &  &      0     & \star & \star \\
  &      &      &      &      &     &  &  &  &  &  &  &      0     &    0       & \star \\
 %\hline 
      &  &  &  &  &  &      &      &      &  &  &  & \star           & \star         & \star \\
      &  &  &  &  &  &      &     &      &  &  &  &   0    & \star           & \star \\
       &  &  &  &  &  &     &      &      &  &  &  &   0    &  0    & \star \\
      &  &  &  &  &  &  &  &  &      &      &      & \star          & \star          & \star \\
      &  &  &  &  &  &  &  &  &      &      &      &  0             & \star          & \star \\
      &  &  &  &  &  &  &  &  &      &      &      & 0              &        0       & \star \\
             &  &  &  &  &  &    &     &     &  &  &  &   0    &  0    &  0  \\   % \hline   
           &     &     &  &   &   &   &  &  &  &  &    &  0      &  0     &   0    \\  %\hline
        &      &      &     &    &    &  &  &  &  &  &  &      0     &        0       &    0 \\ 
      &  &  &  &  &  &  &  &  &     &     &     &     0          &        0       &       0        \\
  \end{smallmatrix}\right)}_{(\mathcal{Q}^{(k+1)})^* \mathbf{A}^{(k+1)},\;\;\mbox{row permuted}} .
\end{equation}
Of course, the above matrices will not be used as an actual data structure in a computational routine. 
But, this global view of the \textsf{LS} problem is useful for conceptual considerations.
For instance, the $\epsilon^{(k)}_{3}$ part of the residual corresponds to the zero rows of $(\mathcal{Q}^{(k+1)})^*\mathbf{A}^{(k+1)}$ -- 
they build the block of zero rows at the bottom of the row-permuted $\Pi(\mathcal{Q}^{(k+1)})^*\mathbf{A}^{(k+1)}$, see (\ref{eq:structure-partialQR}). The corresponding entries in the transformed right-hand side amount to $\epsilon^{(k)}_{3}$ in the Euclidean norm.

The \textsf{LS} problem with the block upper triangular permuted $\mathbf{B}^{(k+1)}=\Pi (\mathcal{Q}^{(k+1)})^* \mathbf{A}^{(k+1)}$ 
and the corresponding partitioned right-hand side $\mathbf{s}^{(k+1)}= \Pi (\mathcal{Q}^{(k+1)})^* \mathbf{b}$ can be written as, see (\ref{eq:structure-partialQR}),
\begin{equation}\label{zd:eq:VF-LS1}
\begin{pmatrix} 
\mathbf{B}_{[11]}^{(k+1)} & \mathbf{B}_{[12]}^{(k+1)} \cr 0 & \mathbf{B}_{[22]}^{(k+1)}  \cr 0 & 0
\end{pmatrix}
\begin{pmatrix} \bfPhi^{(k+1)} \cr \bfvarphi^{(k+1)} \end{pmatrix} \approxeq 
\begin{pmatrix} \mathbf{s}_1^{(k+1)} \cr \mathbf{s}_2^{(k+1)}  \cr \mathbf{s}_3^{(k+1)}
\end{pmatrix} .
\end{equation} 
%\subsection{Further details}
\begin{algorithm}[hh]
	\caption{Vector Fitting - Basic Iterations} \label{zd:ALG:MIMOVF-basic}
	\begin{algorithmic}[1]
		\STATE Given: The sampling data $\bfH(\xi_i)$ for $i=1,\ldots, \ell$ ; maximal number of iterations $k_{\max}$.
		\STATE Set $k \leftarrow 0$ and make an initial pole selection $\bfpol^{(k+1)}\in\Cplx^r$ .
		\WHILE{ \{ stopping criterion not satisfied and $k\leq k_{\max}$ \}}
		\STATE Form $\mathbf{A}^{(k+1)}$ and $\mathbf{b}$.
		\STATE Compute $\mathbf{B}^{(k+1)}=\Pi (\mathcal{Q}^{(k+1)})^* \mathbf{A}^{(k+1)}$ and $\mathbf{s}^{(k+1)}= \Pi (\mathcal{Q}^{(k+1)})^* \mathbf{b}$ and partition as in (\ref{zd:eq:VF-LS1}).
		\STATE Solve $\| \mathbf{B}_{[22]}^{(k+1)} \bfvarphi^{(k+1)} - \mathbf{s}_2^{(k+1)}\|_2\longrightarrow\min$ for $\bfvarphi^{(k+1)}$.
		\STATE Set $k\leftarrow k+1$ and compute $\bfpol^{(k+1)} = zeros(1+\sum_{j=1}^r \dres_j^{(k)}/(s-\lambda_j^{(k)}))$.
		\ENDWHILE
		\STATE $\bfPhi = (\mathbf{B}_{[11]}^{(k)})^{-1} \mathbf{s}_1^{(k)}$.
	\end{algorithmic}
\end{algorithm}
\begin{remark}\label{REM:OnLine9}
\em{ Observe that the iteration on $\bfvarphi^{(k)}$ proceeds independently of $\bfPhi$ and indeed, $\bfPhi$ is obtained only in the final step, Line 9, which accomplishes the simultaneous determination of 
residues by minimizing  
$\|\dqw \left(\Cauchy^{(k+1)}\varPhi^{(k+1)}(u,v,:) -  \mathbb{S}(u,v,:)\right)\|_2$, for $u=1,\ldots,p$, $v=1,\ldots,m$.	 This observation  was first exploited in \cite{Deschrijver-Mrozowski-Dhaene-Zutter-2008}.  This may be implemented as a solution of an  \textsf{LS} problem with multiple right-hand sides,
and additional measures can be taken to compute more accurate residues, see \S \ref{SS=HALA}. 
Stopping criteria (Line 3) will be discussed in \S \ref{SS=StopCrit}. }
\end{remark}
%
%%%%%%%%%%%%%%%%%%%%%%%%%%%%%%%%%%%%%%%%%%%%%%%%%%%%%%%%%%%%%%%%%%%%
%
%\section{Numerical implementation details}\label{S=NumericalDetails}
\section{Numerical issues that arise in standard \textsf{VF}}\label{S=NumericalDetails}
The key variables in \textsf{VF} are computed as solutions of \textsf{LS} problems, where the coefficient matrices are built from
Cauchy and diagonally scaled Cauchy matrices, thus potentially highly ill-conditioned.
Further, as we hope to capture the data by reducing the residual, we also expect cancellation to take place. These issues pose tough challenges to numerical analyst during the finite precision implementation of the algorithm. In this section, we discuss several important details that are at the core of a robust implementation of  \textsf{VF}.

%© 2015 The MathWorks, Inc. \textsc{matlab} and Simulink are registered trademarks of The MathWorks, Inc. See www.mathworks.com/trademarks for a list of additional trademarks. Other product or brand names may be trademarks or registered trademarks of their respective holders.

\subsection{Least squares solution and rank revealing QR factorization}\label{SS=LS+QR}
To fully understand  the global behavior of \textsf{VF} iterations in finite precision arithmetic,  %with likely numerical rank deficiency, 
it is crucial to investigate all the details of an \textsf{LS} solver used in a robust software implementation.
For example, consider  Line 6. in Algorithm \ref{zd:ALG:MIMOVF-basic}, i.e., consider the  \textsf{LS} problem $\| \mathbf{B}_{[22]} \bfvarphi - \mathbf{s}\|_2\rightarrow\min$ where we now drop all superfluous indices to ease notation.  In a \textsc{matlab} implementation, the solution
is obtained using the backslash operator, i.e., $\bfvarphi = \mathbf{B}_{[22]} \backslash \mathbf{s}$, or using the pseudoinverse, i.e., $\bfvarphi = \textsf{pinv}(\mathbf{B}_{[22]}) \mathbf{s}$, computed using the SVD and an appropriate threshold for determining numerical rank.
The state of the art \textsf{LAPACK} library \cite{LAPACK} provides driver routines \textsf{xgelsy}, based on a complete orthogonal decomposition,  and \textsf{xgelss},
\textsf{xgelsd}, based on the SVD decomposition.  

 We briefly describe the decomposition approach:  
In the first step, the column pivoted QR factorization is computed and written in partitioned form 
\begin{equation}\label{eq:RRQR}
\mathbf{B}_{[22]} P = W T = \begin{pmatrix} W_1 & W_2 \end{pmatrix} \begin{pmatrix} T_{[11]} & T_{[12]} \cr 0 & T_{[22]}\end{pmatrix},\;\; W^* W = \Id , \;\; \|T_{[22]}\|_F \leq \epsilon \|T_{[11]}\|_F ,
\end{equation}
where $\epsilon$ is a threshold value, e.g. $\epsilon = n\roff$. 
As a consequence of the Businger--Golub pivoting \cite{bus-gol-65}, 
\begin{equation}\label{eq:Tii}
|T_{ii}| \geq \sqrt{\sum_{j=i}^k |T_{jk}|^2},\;\; 1 \leq i\leq k \leq r.
\end{equation}
%\marginpar{\raggedright\textcolor{blue}{\tiny  Presorting the rows \\ has the same \\ %effect, right ? \\} } 
%
\textcolor{black}{In the case of differently weighted rows of $\mathbf{B}_{[22]}$, numerical stability can be enhanced by using Powell-Reid complete pivoting \cite{Powell-Reid-1968} or by presorting the rows in order of decreasing $\infty$-norm \cite{cox-hig-98}}. 
The actual size of $T_{[11]}$ may be determined by an incremental condition number estimator, or by inspecting for gaps in the sequence $|T_{11}| \geq |T_{22}|\geq\cdots\geq |T_{nn}|$.  If no such partition is possible, then $T=T_{[11]}$ and the block $T_{[22]}$ is void. In ill-conditioned cases, as we could have in Algorithm \ref{zd:ALG:MIMOVF-basic}, such a partition is likely to be visible, see, e.g., Figure \ref{FIG-Structure-BQR-Heat-1}. Then, $T_{[22]}$ is deemed negligible noise and set to zero. This is justified by backward error analysis: there exists a small perturbation of the initial matrix $\mathbf{B}_{[22]}$ such that  this part of the triangular factor  is exactly zero. Then, an additional orthogonal reduction transformation $Z$ (similar to QR factorization, see, e.g., \textsf{LAPACK} routine \textsf{xtzrzf}) is deployed from the right leading to a URV decomposition
\begin{equation}\label{eq:LS:solution}
\mathbf{B}_{[22]} P \approx W \begin{pmatrix} \widehat{T}_{[11]} & 0 \cr 0 & 0 \end{pmatrix} Z,\;\;\mbox{and the solution}\;\; \bfvarphi = P Z^*\begin{pmatrix} \widehat{T}_{[11]}^{-1} W_1^* \mathbf{s} \cr 0 \end{pmatrix} .
\end{equation}
The vector $\bfvarphi$ in (\ref{eq:LS:solution}) is the minimal Euclidean norm solution of a nearby (backward perturbed) problem. However, in the rank deficient case, other particular choices from the solution manifold might be of interest. For instance, once we set $T_{[22]}$ in (\ref{eq:RRQR}) to zero, we can use
\begin{equation}\label{eq:LS:solution2}
\mathbf{B}_{[22]} P \approx W \begin{pmatrix} {T}_{[11]} &T_{[12]}  \cr 0 & 0 \end{pmatrix},\;\;\mbox{and the solution}\;\; \bfvarphi^{(0)} = P \begin{pmatrix} {T}_{[11]}^{-1} W_1^* \mathbf{s} \cr 0 \end{pmatrix} ,
\end{equation}
with the same residual norm as $\bfvarphi$ in (\ref{eq:LS:solution}). Note that while
the \textsc{matlab} command $\mathbf{B}_{[22]} \backslash \mathbf{s}$ computes $\bfvarphi^{(0)}$,  
the SVD based $\textsf{pinv}(\mathbf{B}_{[22]}) \mathbf{s}$ and 
the \textsf{LAPACK} routines return the minimal norm solution $\bfvarphi$.
\textcolor{black}{
These details may also significantly impact the computation in Line 9 of Algorithm \ref{zd:ALG:MIMOVF-basic} (cf. Remark
\ref{REM:OnLine9}). Numerical rank deficiency will trigger truncation and in the case of  the solution method in (\ref{eq:LS:solution2}), some of the residues in 
(\ref{eq:H_r(k)-final}) will be computed as $p\times m$ zero matrices, thus effectively removing the corresponding poles from $\bfH_r(s)$. We discuss this residue computation step in more detail in \S \ref{SS=HALA}.}
%\begin{remark}
%{\em

Determining the numerical rank is a delicate procedure and it should be tailored to
a particular application, based on all available information and interpretation of
the solution. For instance, what is a sensible choice for the threshold $\epsilon$ in
(\ref{eq:RRQR}) and how do we decide whether to prefer the solution of minimal Euclidean length or the solution with most zero entries? What can we infer from the numerical rank of 
$\mathbf{B}_{[22]}$? These issues are further
discussed in \S \ref{SS=noise}. 	
%}	
%\end{remark}
\subsection{Convergence introduces noise}\label{SS=noise}
In this section, we analyze and illustrate that as  \textsf{VF} proceeds, the coefficient matrices in the pole identification phase tend to become noisy with a significant drop in column norms that may coincide also with a reduction in numerical rank. This prompts us to advise caution when rescaling columns in order to improve the condition number, since rescaling columns that have been computed
through massive cancellations will preclude inferring an accurate numerical rank. Since \textsf{VF} simultaneously fits all input--output pairs using a common set of poles, 
it suffices  to focus our analysis on only one fixed input--output pair $(u,v)$. For simplicity of the notation, we drop the iteration index $k$, and take unit weights, i.e., $\dqw=\Id_\ell$. %
 If $r$ is large enough and the poles have settled, then, with some small $error$,
$$
\mathbb{S}(u,v,1:\ell) 
\approx  \Cauchy x + error = (Q)_1 (R)_{11} x + error,
$$
{where $x=\varPhi(u,v,1:r)$, see Remark \ref{REM:OnLine9}.}
% and $\respm\equiv\respm^{(k+1)}$ is as in (\ref{eq:H_r(k)-final}). 
Here we used the QR factorization (\ref{eq:concurrent-QRs}). Now, the right hand side in the error contribution $\epsilon_2$ in (\ref{eq:eps123}) that corresponds to the pair $(u,v)$ is
\begin{equation}\label{eq:smallRHS}
(Q_{uv})_2^* \mathbb{S}(u,v,1:\ell) =  (Q_{uv})_2^* (Q)_1 (R)_{11} x 
+ (Q_{uv})_2^* error = (Q_{uv})_2^* error .
\end{equation}
The vectors of the structure (\ref{eq:smallRHS}) are building the vector $\mathbf{s}_2$  in (\ref{zd:eq:VF-LS1}).
Furthermore, using (\ref{eq:concurrent-QRs}), 
$$
(R_{uv})_{22} = (Q_{uv})_2^* \underbrace{\mathrm{diag}(\mathbb{S}(u,v,1:\ell))}_{\ell\times\ell} \Cauchy = 
(Q_{uv})_2^* \left\{\left[(\Cauchy x + error)\begin{pmatrix} 1 & \ldots& 1\end{pmatrix}\right] \circ \Cauchy \right\} ,
$$
where $\circ$ denotes the Hadamard product. Hence, a $j$th column of $(R_{uv})_{22}$ reads
$${\displaystyle 
	(R_{uv})_{22}(:,j) =(Q_{uv})_2^* \left\{ \left[ (Q)_1 (R)_{11} x + error\right] \circ 
	\left(\begin{smallmatrix} {1}/{(\xi_1 - \pol_j)} \cr \vdots \cr {1}/{(\xi_\ell - \pol_j)}\end{smallmatrix}\right)\right\}},
$$
which means that $(R_{uv})_{22}(:,j)$ could also be small, depending on the
position of $\lambda_j$ relative to the $\xi_i$s. 
The \textsf{LS} coefficient matrix $\mathbf{B}_{[22]}$ in line 6. of Algorithm \ref{zd:ALG:MIMOVF-basic} is assembled from the matrices $(R_{uv})_{22}$, and we can expect that it will have many small entries that are (when computed in floating point arithmetic) mostly contaminated by the roundoff noise. We illustrate this on an example.
\begin{example}\label{Ex:Structure-B}
	{\em 
		We use the one-dimensional heat diffusion equation model \cite{NICONET-report}, obtained by spatial discretization
		of 
		$
		\frac{\partial}{\partial t}T(x,t) = \alpha \frac{\partial^2}{\partial x^2} T(x,t) + u(x,t),\;\;0<x<1,\;\; t>0
		$
		with the zero boundary and initial conditions. The discretized system is of order $n=197$. We generate $\ell=1000$ samples and set $r=80$. The structures  of $\mathbf{B}_{[22]}$ and its column pivoted triangular factor are given in Figure \ref{FIG-Structure-BQR-Heat-1}. 
		The column norms of $\mathbf{B}_{[22]}$ in the first step are so particularly ordered due to the ordering of the initial poles  $\lambda_j = \alpha_j \pm \imunit \beta_j$, where the $\beta_j$s are logarithmically spaced between the minimal and the maximal
		sampling frequency and $\alpha_j=-\beta_j$.		
		Note the sharp drop in the column norms in the second iteration, after the relocated poles induced better approximation.
		%\begin{figure}
		%	\begin{center}
		%		\includegraphics[width=2.9in,height=2.8in]{H500r80BQR1.eps}
		%		\includegraphics[width=2.9in,height=2.8in]{H500r80BQRP1.eps}
		%		\includegraphics[width=2.9in,height=2.8in]{H500r80BQR2.eps}
		%		\includegraphics[width=2.9in,height=2.8in]{H500r80BQRP2.eps}
		%	\end{center}
		%	\caption{\label{FIG-Structure-BQR-Heat} The structure of the triangular factor ($\log_{10}$ of the absolute values) in the QR factorization of the matrix $B$ in line 6. during the first two iterations in Algorithm \ref{zd:ALG:MIMOVF-basic}. In the first row, the triangular factor of $\mathbf{B}$ without (left) and with (right) column pivoting. The structure of the first plot reflects a particular order of the nodes $\lambda_j$, i.e. the columns of $\Cauchy$. The plots in the second row show the structure in the second iteration, with improved approximation of the input data.}
		%\end{figure}	
		\begin{figure}
			\begin{center}
				\includegraphics[width=2.9in,height=2.8in]{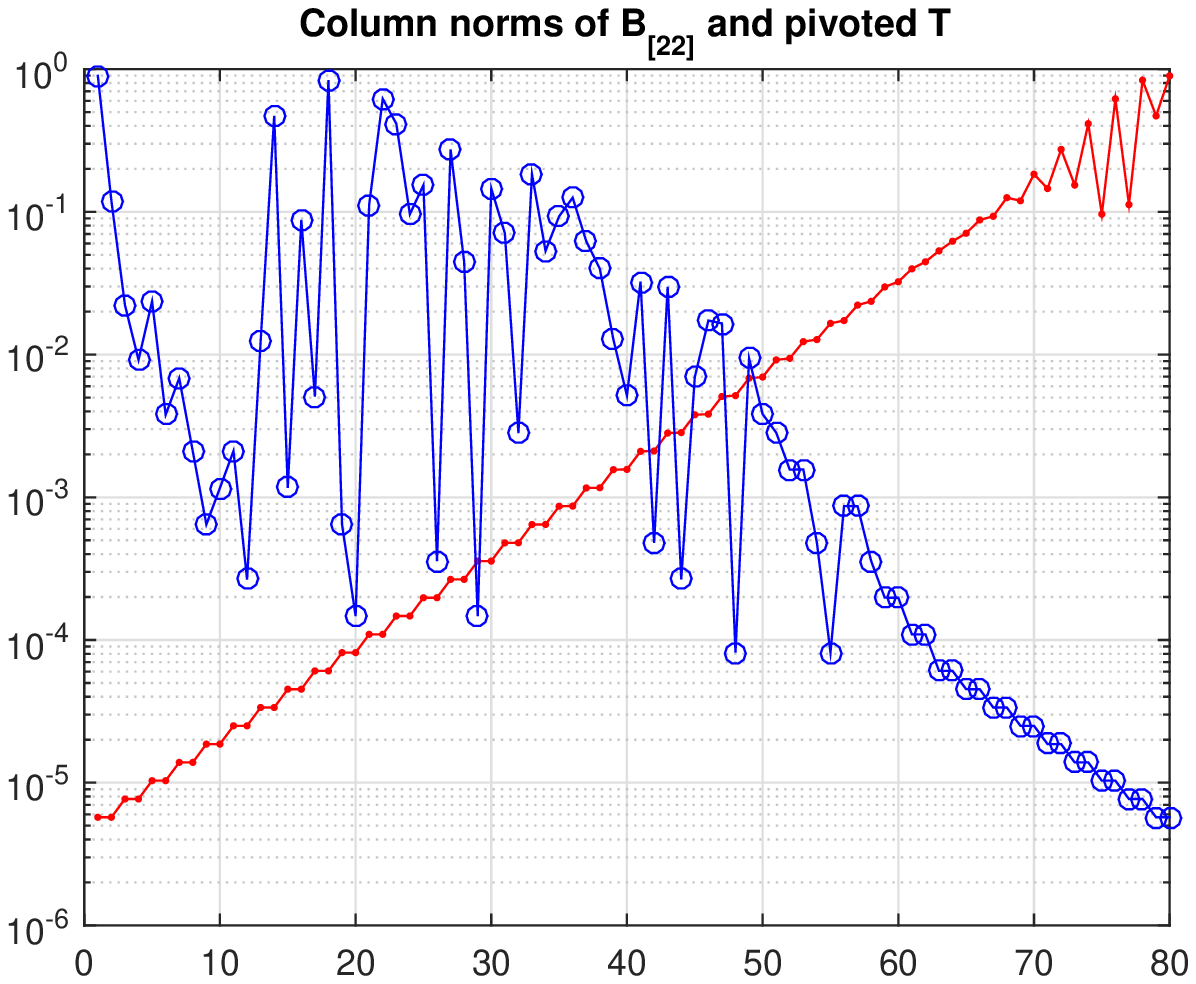}
				\includegraphics[width=2.9in,height=2.8in]{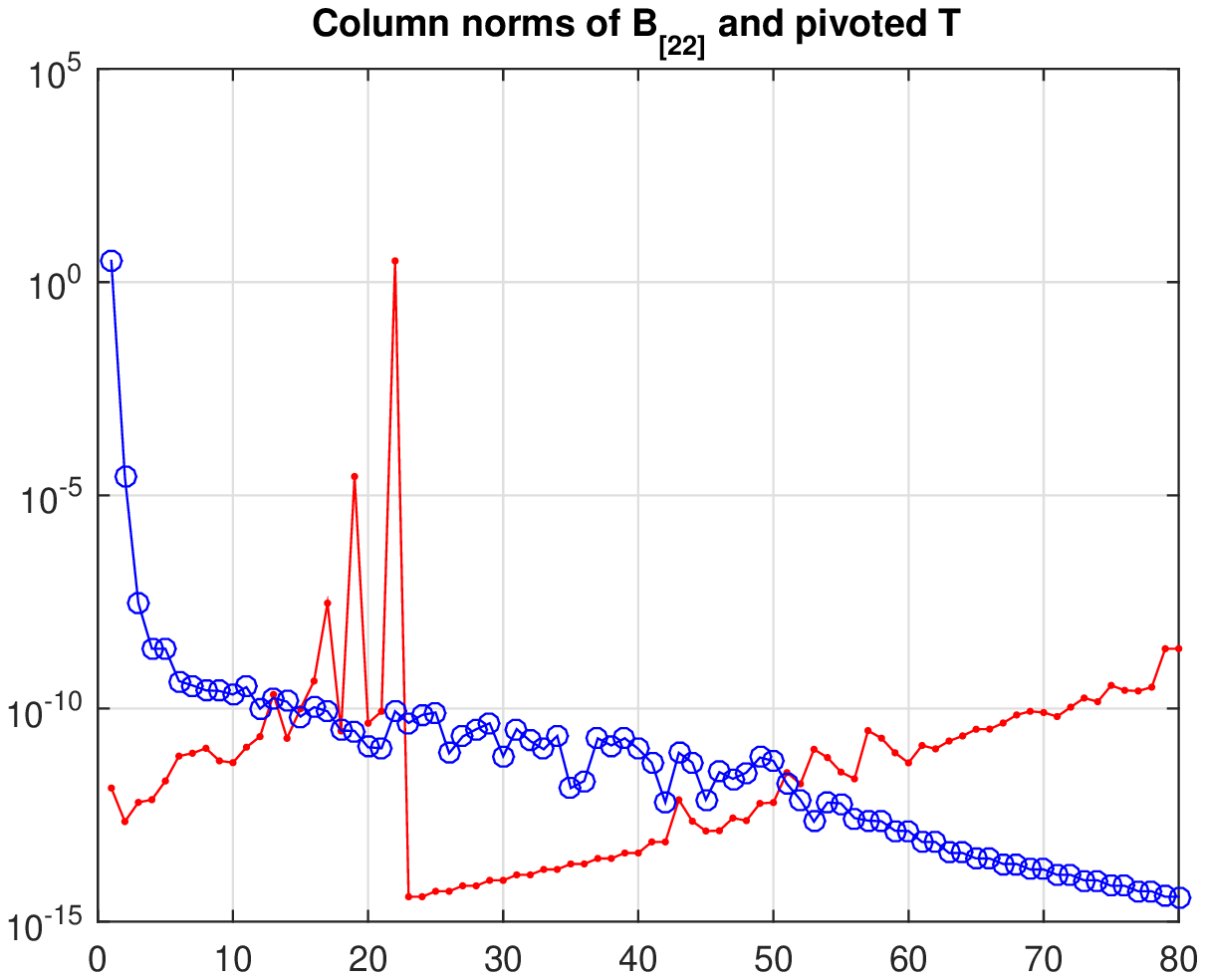}
			\end{center}
			\caption{\label{FIG-Structure-BQR-Heat-1} The structure of the matrix $\mathbf{B}_{[22]}$ and its column pivoted triangular factor in line 6. during the first two iterations in Algorithm \ref{zd:ALG:MIMOVF-basic}. In the first plot, showing the data in the first iteration, the column norms  of $\mathbf{B}_{[22]}$ are marked with (red) $\cdot-$ and the column norms its triangular factor in the Businger-Golub pivoted QR factorization are marked with (blue) $\circ-$. 
				The  second plot shows the same information, but in the second iteration.}
		\end{figure}
	}
\end{example}

\subsection{The Quandary of Column Scaling}\label{SS=ColumnScaling}
In the \textsf{VF} literature, it is often recommended to scale the columns of the \textsf{LS} coefficient matrix in line 6 of Algorithm \ref{zd:ALG:MIMOVF-basic}, to make them all of the same Euclidean length,  before deploying the \emph{backslash} solver, and then to rescale the solution, see, 
%
%\marginpar{\raggedright\textcolor{blue}{\tiny  There is a slot \\ for two references \\ Is there another ?  \\} } 
%
e.g., \cite{Gustavsen-EMC-2007}.  One desirable effect of this column equilibration step is to reduce the effective condition number of \textsf{LS} coefficient matrix (see e.g., \cite{lawson1974solving,sluis1969condition}).   While this can be beneficial, nonetheless this tactic also may have a variety of deleterious effects and, in our opinion, must be considered with caution.  Of foremost concern, following the discussion from \S \ref{SS=noise}, is that scaling noisy matrix columns effectively increases the influence of noise, and allows these noisy columns to participate in the column pivoting process of the QR factorization, with a possibility that some of them become drafted and taken upfront as important.  This, in turn, interferes with the rank revealing process and possibly precludes truncation based on a partition 
 as in (\ref{eq:RRQR}). For illuminating discussions related to this issue we refer to \cite{Golub:1976:RDL:892104}, \cite{Golub198041}, \cite{Stewart-Rank-1993}. In the following two examples we illustrate the potentially baleful effects of column scaling in the particular context of  
 \textsf{VF} iteration. 

\begin{example}\label{NumEx-2}
{\em 
We continue using the heat model from Example \ref{Ex:Structure-B}.
In Figure \ref{FIG-Structure-diagR-Heat-1} we show, for the first two iterations, the moduli of the diagonal entries ($|T_{ii}|$) of the triangular factors of $\mathbf{B}_{[22]}$ (appearing in line 6 of Algorithm \ref{zd:ALG:MIMOVF-basic}) without and with column equilibration. Note that, due to the diagonal dominance (\ref{eq:Tii}), the distribution of $|T_{ii}|$ is decisive for numerical rank revealing. 
%(The data are from Example \ref{Ex:Structure-B}.) 

\begin{figure}
	\begin{center}
		\includegraphics[width=2.9in,height=2.8in]{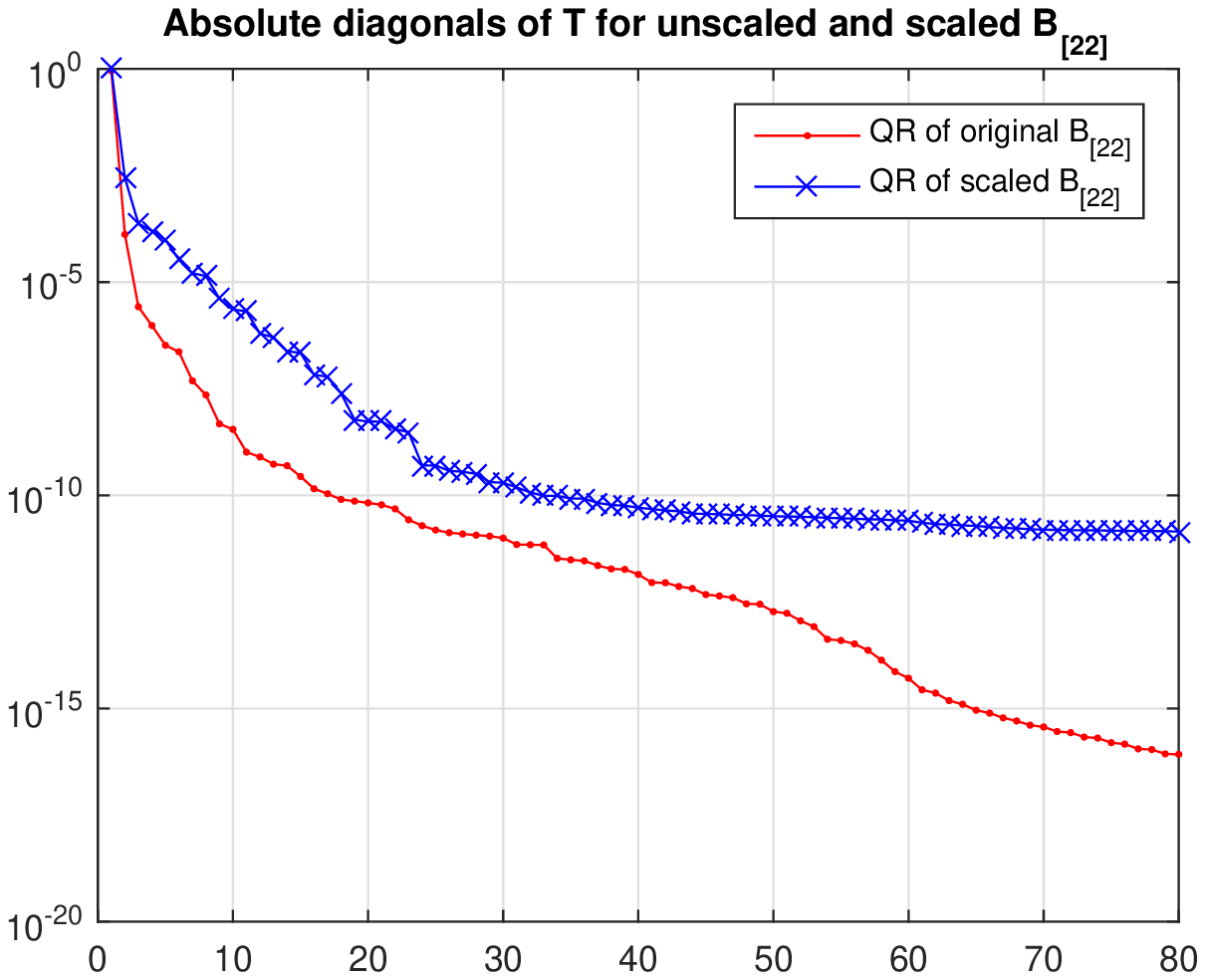}
		\includegraphics[width=2.9in,height=2.8in]{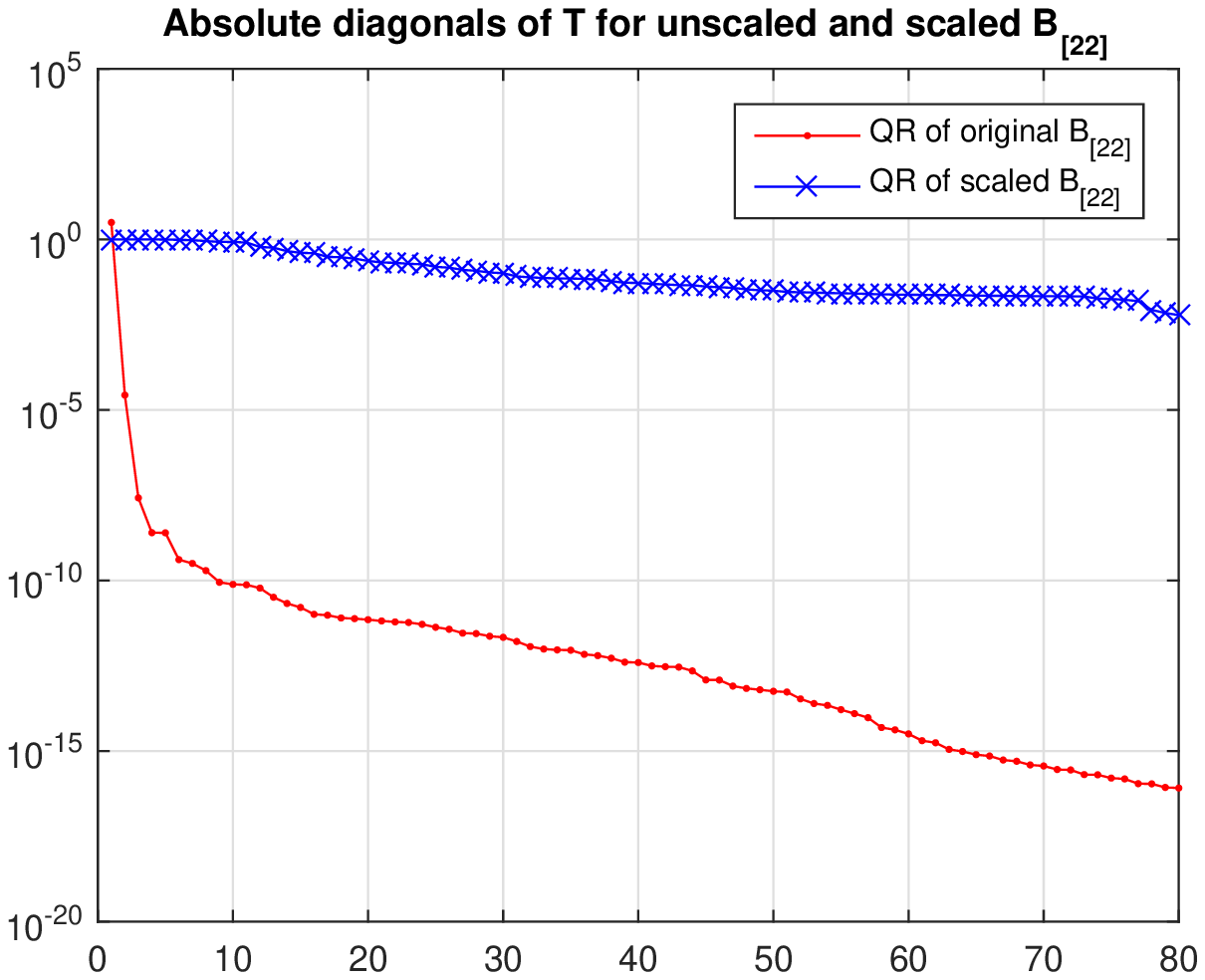}
	\end{center}
	\caption{\label{FIG-Structure-diagR-Heat-1} (Example \ref{Ex:Structure-B} cont.) The structure of the pivoted triangular factors (cf.(\ref{eq:Tii})) of $\mathbf{B}_{[22]}$ in the first two iterations in Algorithm \ref{zd:ALG:MIMOVF-basic}. The plot shows the values $|T_{ii}|$ of the unscaled $\mathbf{B}_{[22]}$ (marked with (red) $\cdot-$) and of the column equilibrated  $\mathbf{B}_{[22]}$ (marked with (blue) $\times-$). }
\end{figure}

We now illustrate how the numerical rank deficiency may be manifested during the \textsf{VF} iterations. 
%		State space realization of this system has the dimension $n=197$, but only about $40$ of its Hankel singular values are significant, and the rest are at the roundoff level, see  Figure
%		\ref{FIG-Heat-HSVD}.  Based on this, one could estimate that this system can be approximated with an order $34$ system up to a relative accuracy of $O(10^{-13})$ (as measured with the Hankel norm).
		We solve the \textsf{LS} problems for the $\bfvarphi^{(k)}$s using only a simple modification of \textsc{matlab}'s \emph{backslash} operator: first reorder the equations so that the rows of the coefficient matrix have decreasing $\ell_\infty$ norms, and then apply \emph{backslash}.\footnote{Recall the discussion in \S \ref{SS=LS+QR}.}
This stabilizes the \textsf{LS} solution process in much the same way as does Powell-Reid complete pivoting 
		(see \cite{cox-hig-98, higham2000qr}).		
		We  use $1000$ frequencies and $r=80$. The samples are matched perfectly with both our implementation of 
		\textsf{VF} and \textsf{vectfit3} \cite{vectfit3} (up to relative errors of the order of $10^{-13}$).  The first plot in Figure \ref{FIG-Heat-HSVD}
		shows the structure of the denominators $\bfvarphi^{(k)}$ throughout ten iterations. 
		Each $\bfvarphi^{(k)}$ is represented
		by the sorted vector of $\log_{10}(|\bfvarphi^{(k)}|/\|\bfvarphi^{(k)}\|_1)$, and the zero entries of $\bfvarphi^{(k)}$ are not shown. For $k=1,\ldots, 10$,
		the values of $\textsf{sum}(|\bfvarphi^{(k)}|/\|\bfvarphi^{(k)}\|_\infty > \roff)$ are,  respectively,
		$53$,    $43$,    $38$,    $30$,    $44$,    $44$,    $44$,    $40$,    $42$,    $44$. (If we restart the approximation with $r=53$, the number of nonzero coefficients throughout the iterations are $40$,    $36$,    $38$,    $41$,    $41$,     $42$,    $41$,    $41$,    $39$,    $41$.)
		To interpret these numbers, we compute the Hankel singular values $\sigma_1\geq\cdots\geq\sigma_{197}$ and superimpose them on the graph as 
		$\log_{10}(\sigma_i/\sigma_1)$ -- those values marked by $\diamond$.
		%, indicate the error (in the Hankel norm) of the best $i$-th order reduced model.
		 Since the $\sigma_i$s are forming a ``devil's staircase" and there is no clear cutoff index. 
		For instance,  $\sigma_{30}/\sigma_1\approx 1.28e-11$,  $\sigma_{36}/\sigma_1\approx
		2.17e-13$,  $\sigma_{44}/\sigma_1\approx 1.12e-16$,  $\sigma_{53}/\sigma_1\approx 9.00e-17$. 
		(If we use the backslash without the initial row pivoting, the numbers of nonzero coefficients throughout the iterations are $53$, $44$, $37$, $42$,    $50$, $50$, $50$, $49$, $49$, $50$.)

Recall from Proposition \ref{PROP1} that the barycentric nodes corresponding to $\varphi_{j}^{(k)}=0$ are the poles of the current approximant $\bfH_r^{(k)}$ and that the corresponding residues are accessible 
		by an explicit formula (\ref{eq:residues}), thus allowing an estimate of the contribution of the pole and perhaps discarding it and reducing $r$.
		The matching of the number of the untruncated entries in the \textsf{LS} solutions $\bfvarphi^{(k)}$ with the number of significant Hankel singular values is striking and may offer machinery to readjust the order of the  approximant, $r$, during the iterations. 
		This matching cannot be guaranteed in general but offers a great promise.
		%in particular for systems whose Hankel singular values show no significant decay -- in that case the problem of finding good approximation of lower order is particularly difficult.
		 Full understanding of the potential of the \textsf{VF} for determining the order of the underlying system (e.g., as in the Loewner framework \cite{Mayo-2007}) remains an important open problem. 
		The right-hand side plot  in Figure \ref{FIG-Heat-HSVD} also shows $\log_{10}(|\bfvarphi^{(k)}| /\|\bfvarphi^{(k)}\|_1)$, but with the column scaling
		of the least squares coefficient matrix and rescaling the solution, as in \textsf{vectfit3}. Note that, once the scaling is applied, 
		connection to the  Hankel singular values decay is lost; supporting our discussion on the undesirable  effects of column scaling in \textsf{VF}.
		\begin{figure}
			\begin{center}
				\includegraphics[width=2.9in,height=2.5in]{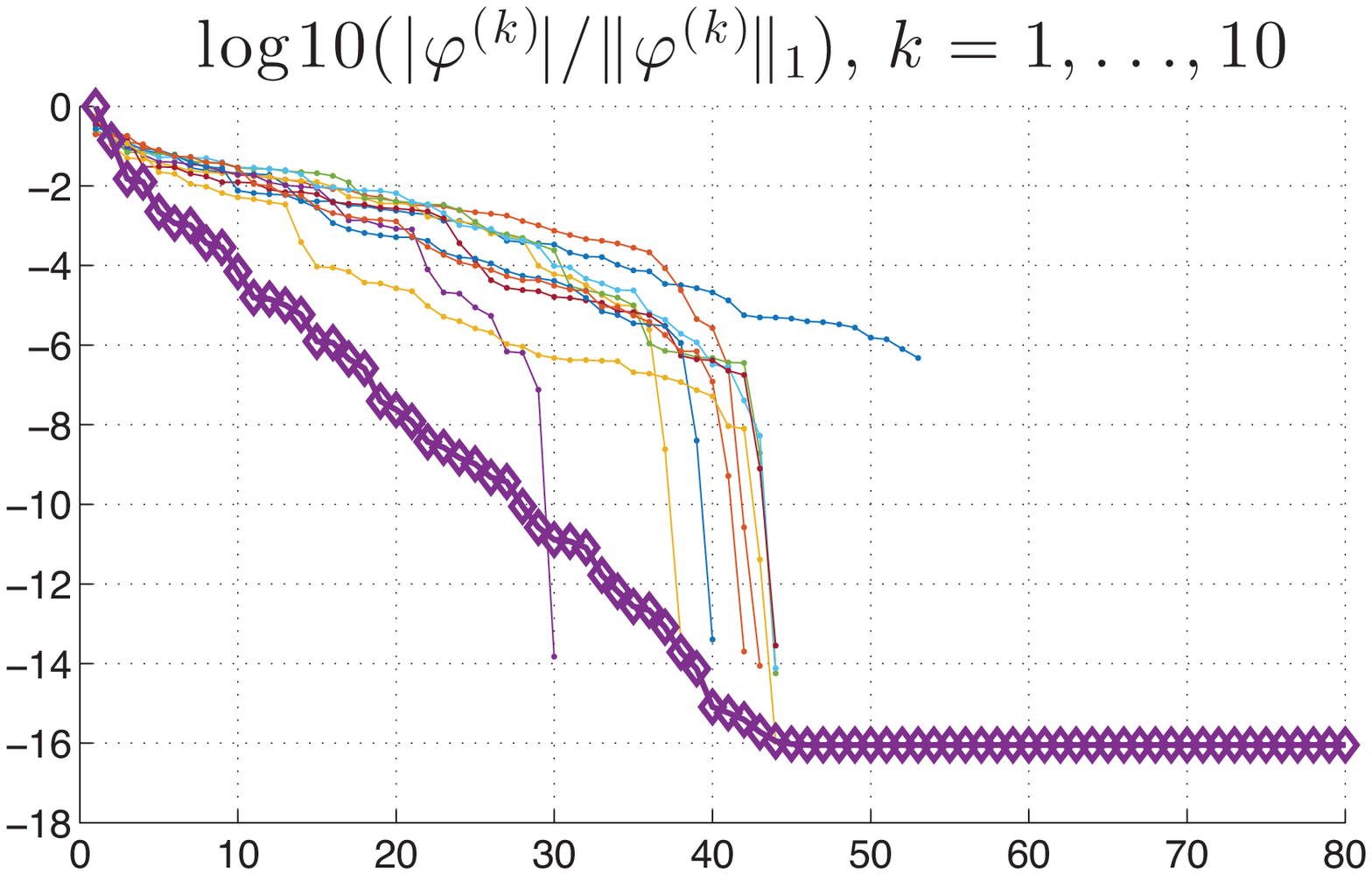}
				\includegraphics[width=2.9in,height=2.5in]{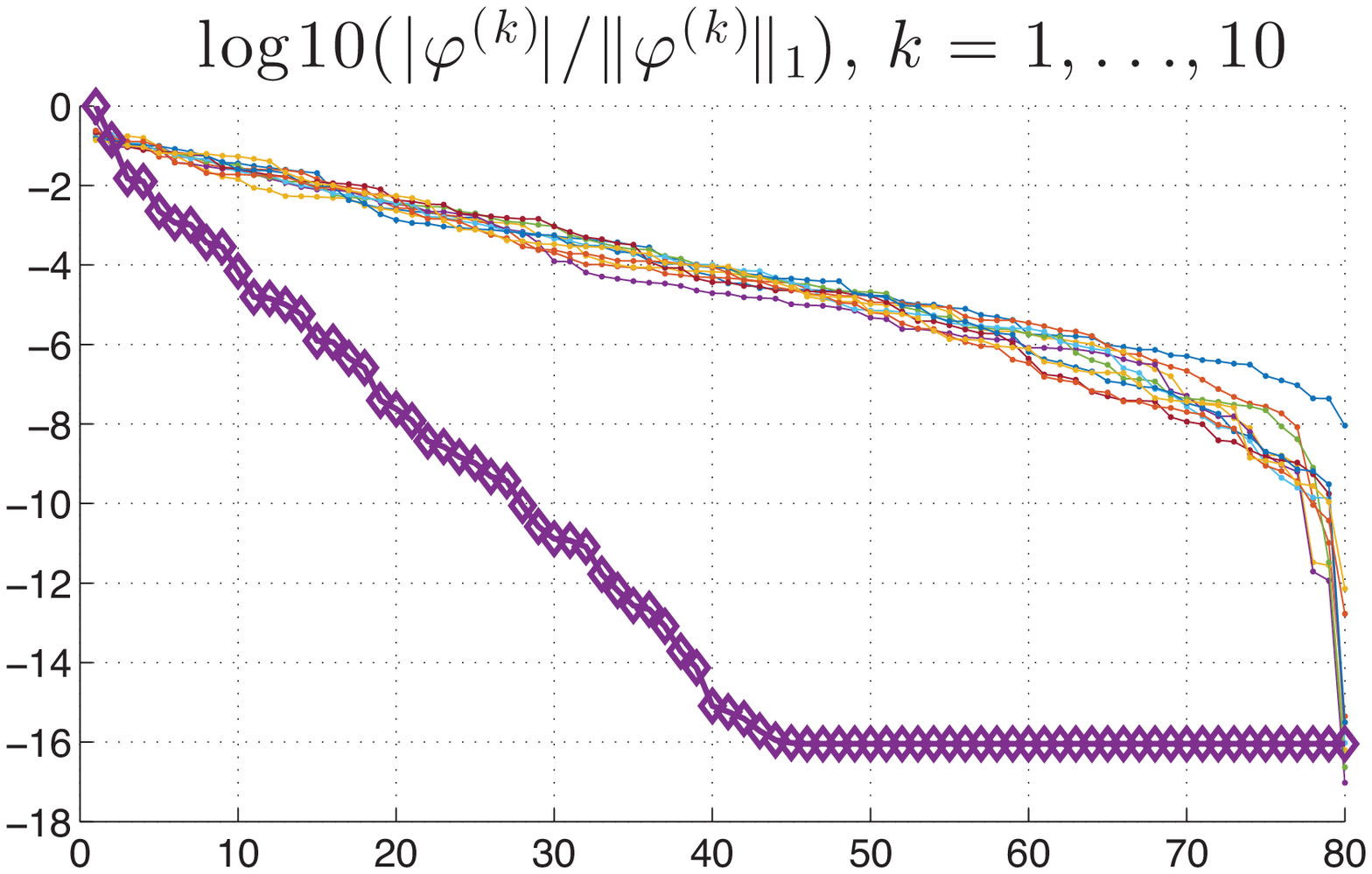}
			\end{center}
			%	\includegraphics[width=5.9in,height=2.8in]{VFZSSL1000r90_truncated.eps}
			%\vspace{-7mm}
			\caption{\label{FIG-Heat-HSVD} (Example \ref{NumEx-2}.)
History of the first 10 iterations, $k=1,\ldots,10$.  Plots show $log_{10}(|\varphi_{j}^{(k)}|//\|\bfvarphi^{(k)}\|_1)$ vs $j$;  only nonzero coefficients are shown. Normalized Hankel singular values are represented as diamonds. 
				}
\end{figure} 
	}
\end{example}		
%
%

%%\paragraph{\textbf{An example}}
\subsection{ \textsf{mimoVF}: Putting the pieces together}
In this section, based on our preceding analysis  of \textsf{VF} for matrix-valued rational approximation problem,
we start testing our new implementation of  \textsf{VF}. We will call  the new implementation and the corresponding   \textsc{matlab} toolbox \textsf{mimoVF}. This section will provide examples for verification and validation of \textsf{mimoVF}.
We will compare our implementation with the original \textsf{vectfit3} \cite{vectfit3}\footnote{The options used deploy the relaxed vector fitting technique \cite{Gustavsen-2006}.}  and show that our proposed modifications based on the theoretical analysis can substantially improve the results.  While on average \textsf{vectfit3}
performs well, in the ill-conditioned cases, it has difficulties with numerical issues addressed in this paper. 

For the resulting rational approximation $\bfH_r$, define the tensor $\HT_r(:,:,i)=\bfH_r(\xi_i)$, $i=1,\ldots, \ell$,
% %the \textsf{LS} error $\bfeps = \|\HT - \HT_r \|_F$, 
and the relative \textsf{LS} error as $\bfgamma = \|\HT - \HT_r \|_F / \|\HT\|_F$. 
Recall that $\HT(:,:,i)=\bfH(\xi_i)\in\Cplx^{p\times m}$, $i=1,\ldots, \ell$, contains the original samples that are 
either measurements, or computed from a state space realization of the underlying LTI dynamical system.

\subsubsection{A stress test}  \label{NumEx-01}
{
We consider a model for the ISS 1R module \cite{NICONET-report}  with $m=p=3$. The underlying dynamical system has dimension $n=270$. This example presents a  difficult test case with rather vivid dynamics. The model is very hard to approximate  and presents significant challenges to model reduction; see \cite{gugercin2001iss}. To that end, we choose $r=50$ and take $\ell=300$ samples. The initial barycentric nodes are chosen as the eigenvalues of a pseudo-random real stable $r\times r$ matrix, a potentially poor initialization. Using these initial nodes, two iteration steps are taken both in \textsf{vecfit3} and \textsf{mimoVF}.  The goal of this example is to illustrate that once the computations in each iteration step are made more robust (following our preceding analysis), high-fidelity rational approximants can still be achieved even with a small number iteration count or even in the cases of poorly initial choices of barycentric nodes. 
%		
%
%\marginpar{\raggedright\textcolor{blue}{\tiny  FRVF in the plot\\ means \textsf{vectfit3}\\ presumably.
%\\ There should be \\ more explanation \\ of the plots. 
%\\ Plotting errors \\ instead of \\ magnitudes would \\ be more\\ enlightening,\\ I think. \\} } 
%
The results of \textsf{vectfit3} and \textsf{mimoVF} are shown on Figure \ref{FIG-compare-to-vf3-ISS270_2iter}
where we depict the amplitude frequency response plots for the data and for the rational approximants. Note that this model has 
$m=3$ inputs and $p=3$ outputs, there are $9$ input/output channels corresponding to the different lines in Figure \ref{FIG-compare-to-vf3-ISS270_2iter}.
The figure clearly illustrates the \textsf{mimoVF} performs significantly better than \textsf{vecfit3} for this example.
Recall that both functions are given the same set of initial poles. Despite this unfavorable choice of initial poles, a restricted number of iterations, and 
ill-conditioned \textsf{LS} matrices, \textsf{mimoVF} succeeds to compute a model with relative \textsf{LS} error below $\gamma\approx 6.45\cdot 10^{-3}$. On the other hand, the relative error due to \textsf{vectfit3} is  
$\gamma \approx 10.41$; a significantly higher value than the error due to \textsf{mimoVF}. 
It is possible that 
 allowing \textsf{vectfit3} to iterate further might realign the poles better, leading to a smaller \textsf{LS} error and an accurate approximate. But, of course, this comes with additional costs since every step of the  iteration requires to solve a potentially large-scale \textsf{LS} problem; especially when $m$ and $p$ are large. Therefore, any reduction in the iteration count is a gain in terms of computational efficiency.
\begin{figure}
	\begin{center}
		\includegraphics[width=2.9in,height=2.8in]{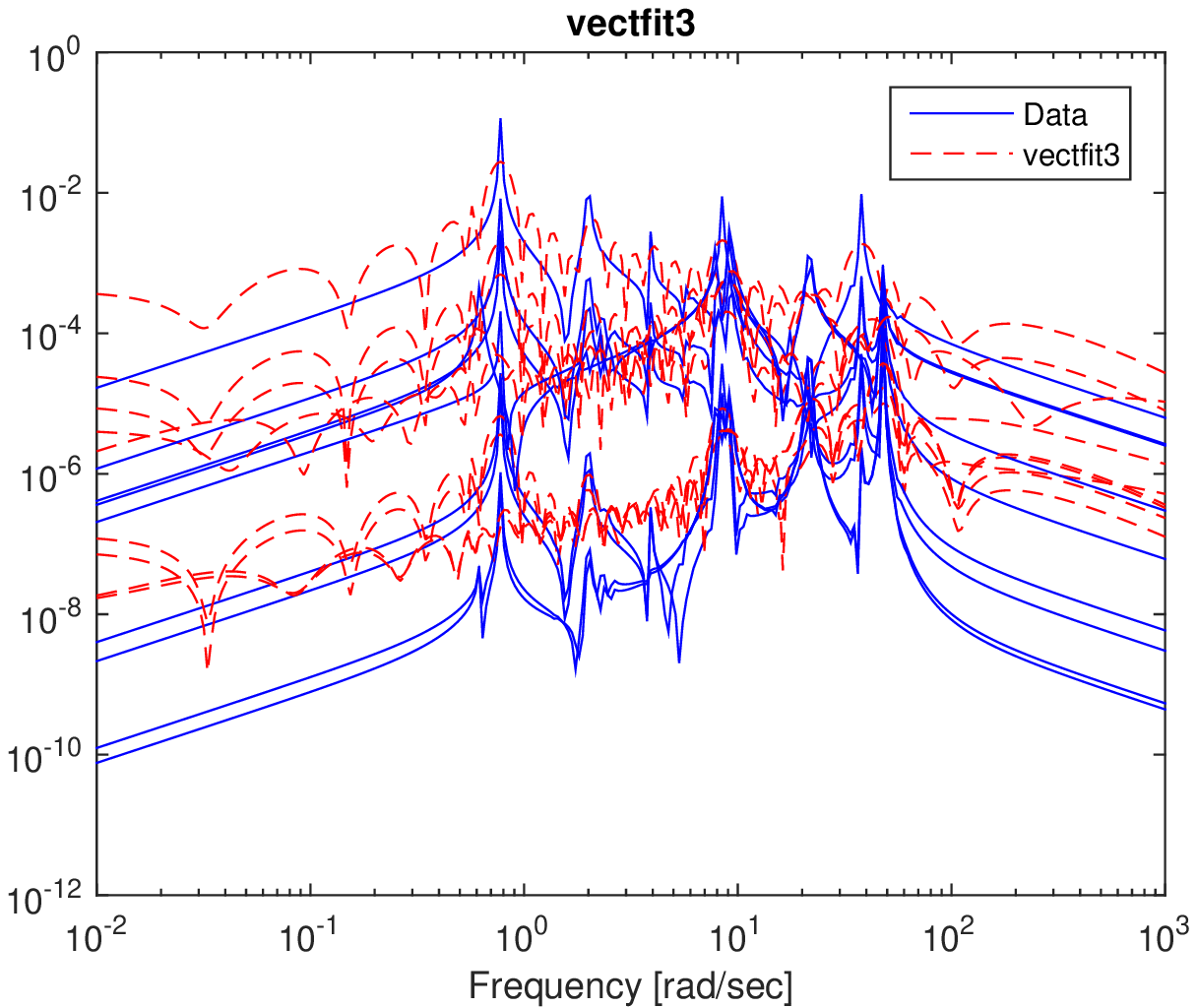}
		\includegraphics[width=2.9in,height=2.8in]{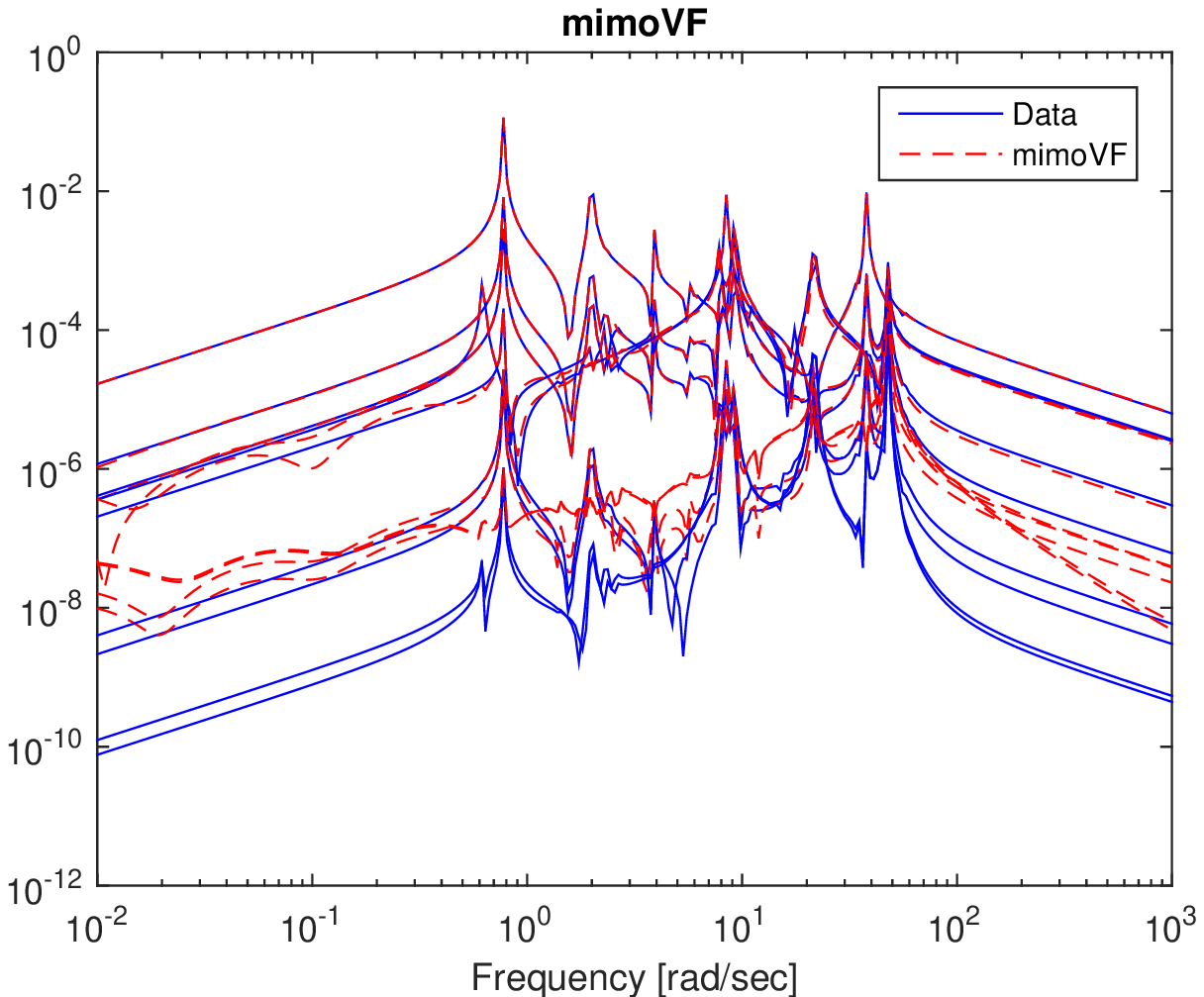}
	\end{center}
	%	\includegraphics[width=5.9in,height=2.8in]{VFZSSL1000r90_truncated.eps}
	%\vspace{-7mm}
	\caption{\label{FIG-compare-to-vf3-ISS270_2iter} 
		(Example of \S \ref{NumEx-01}.) 
		Comparison of  \textsf{mimoVF} and the \textsf{vectfit3} on the ISS 1R module with initial poles set as the eigenvalues of a pseudo-random stable	real matrix,  $\ell=300$ and $r=50$. 
		\textcolor{black}{(The frequency response magnitudes of each of the possible nine input/output pairings is plotted in solid blue; the corresponding frequency response magnitudes from rational approximations provided by \textsf{vectfit3} and \textsf{mimoVF} appear as dashed red lines.)  } 
The purpose of the experiment is to check the robustness of the numerical implementation in the case of unpropitious distribution of the barycentric nodes.  The first plot shows the output of \textsf{vectfit3}
		(with the relative error $\gamma > 10$), and the second of \textsf{mimoVF} ($\gamma < 10^{-2}$), both after two iterations. 
		%The relative $\Hardy_2$ error of $\bfH_r$ computed by \textsf{mmoVF} is
		% $\chi\approx 9.53e-02$.
	}
\end{figure} 
}

\subsubsection{How to compute the residues in the final step}\label{SS=HALA}
One of the advantages of the barycentric implementation of the \textsf{SK} iterations over the original approach using polynomial representations is in the avoidance of high powers of $\xi_i$ (which may cause overflow and underflow in finite precision arithmetic) as producing ill-conditioned Vandermonde matrices. The additional scalings by $1/|d^{(k)}(\xi_i)|^2$, which is another potential source of ill-conditioning, has been elegantly removed by the \textsf{VF} formulation and compensated by reallocating the barycentric nodes $\lambda_j^{(k)}$. However, once the \textsf{VF} 
iterations are completed, one needs to solve for the final residues $\bfPhi$ in Line 9. of Algorithm \ref{zd:ALG:MIMOVF-basic} for the converged poles. This step needs to be performed carefully as the coefficient matrix that determines the residues for given set of poles is a Cauchy matrix, which, together with  Vandermonde matrices, is among the most notoriously ill-conditioned matrices. To illustrate, the spectral condition number of an arbitrary $100\times 100$ real Vandermomde matrix is larger than $3\cdot 10^{28}$, and the condition number of the 
$100\times 100$ Hilbert (Cauchy) matrix is more than $10^{150}$. The column norms of the latter are between $0.07$ and $1.3$, thus no column scaling can substantially reduce the condition number. Furthermore, the additional weightings $\qw_i$ (whose values may spread many orders of magnitude) may further worsen the conditioning of the least squares coefficient matrix. 
This all is a menace to the final computed residues, in particular when the order $r$ is sufficiently high and in cases of unfavorably distributed nodes. This issue has to be addressed if the method is to be applied to truly challenging problems with complex dynamics and of high orders, for instance, for $m, p, r$ in hundreds. In this section we focus our attention to the very last step -- given poles of a rational approximant, how to best numerically extract  the residues.

\begin{example}\label{NumEx-01-A} 
	{\em We continue to use the ISS 1R example from \S \ref{NumEx-01}. However, in this case, we choose 
good initial poles of the form $\lambda_j = \alpha_j \pm \imunit\beta_j$, where the (positive) $\beta_j$s are log-spaced over the frequency sample interval and $\alpha_j=-\beta_j$ as often recommended in the \textsf{VF} literature for good initial pole selection. 
 We take $\ell=500$ and choose $r=100$. Recall that the underlying system has dimension $n=270$ with 
 $p=m=3$; therefore with $r=100$ and $\ell=500$ samples, one expects to obtain a very good approximation.
 The  amplitude frequency response plots are depicted in  Figure \ref{FIG-compare-to-vf3-ISS270_2iter_logpoles} for both \textsf{vecfit3} and \textsf{mimoVF}, once again illustrating that
 \textsf{mimoVF} outperforms \textsf{vecfit3}; the relative \textsf{LS} errors due to \textsf{mimoVF} and \textsf{vecfit3} are respectively, $9.47\cdot 10^{-1}$ and  $4.90\cdot 10^{-3}$.
  In this case, \textsf{vecfit3} suffers from the numerical ill-conditioning of the final residue computation.
Thus, this example  shows that even a plenty of good initial barycentric nodes, one does not necessarily guarantee a good approximation due to the numerical issues arising in the residue computation step.
\begin{figure}
	\begin{center}
		\includegraphics[width=2.9in,height=2.8in]{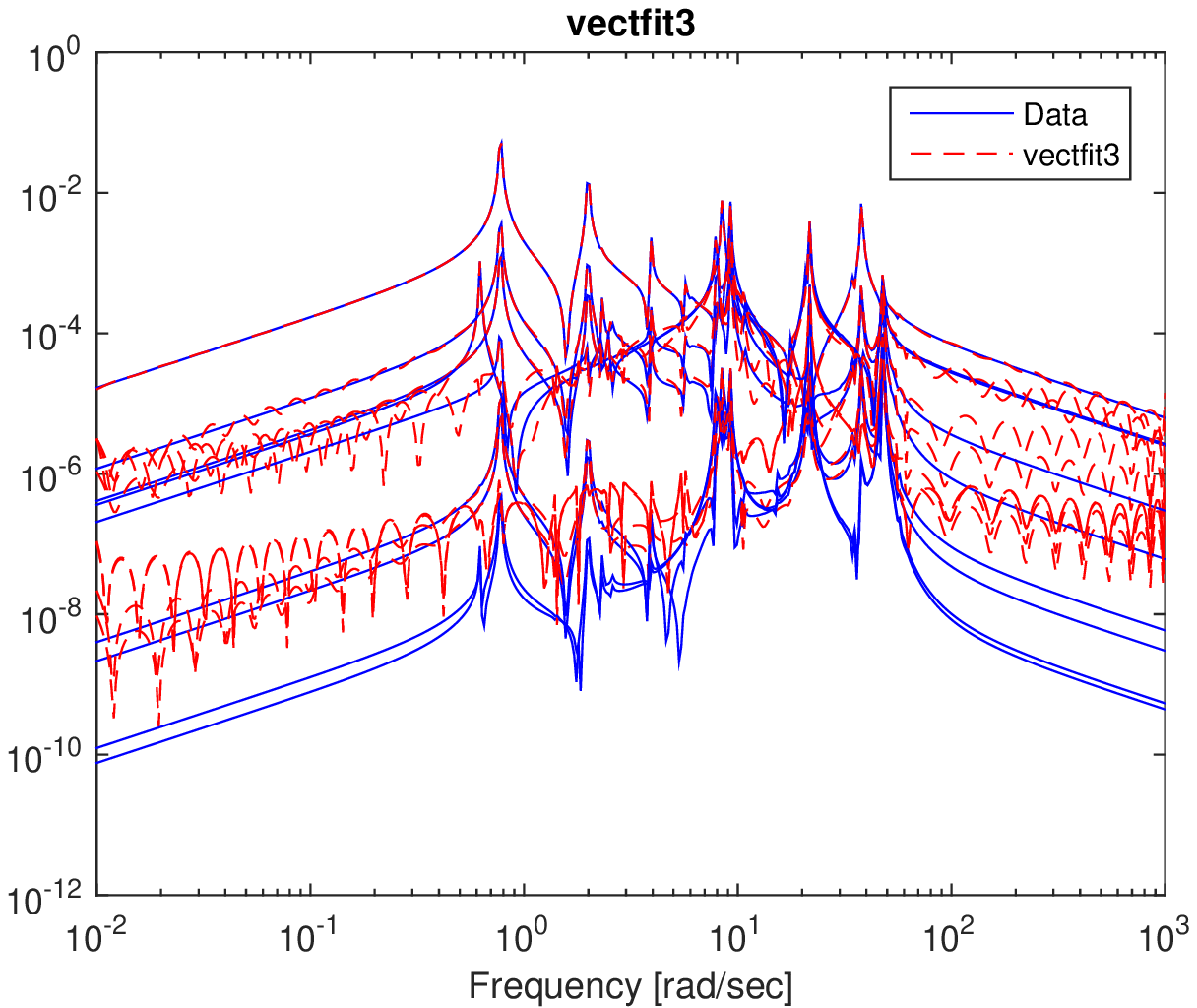}
		\includegraphics[width=2.9in,height=2.8in]{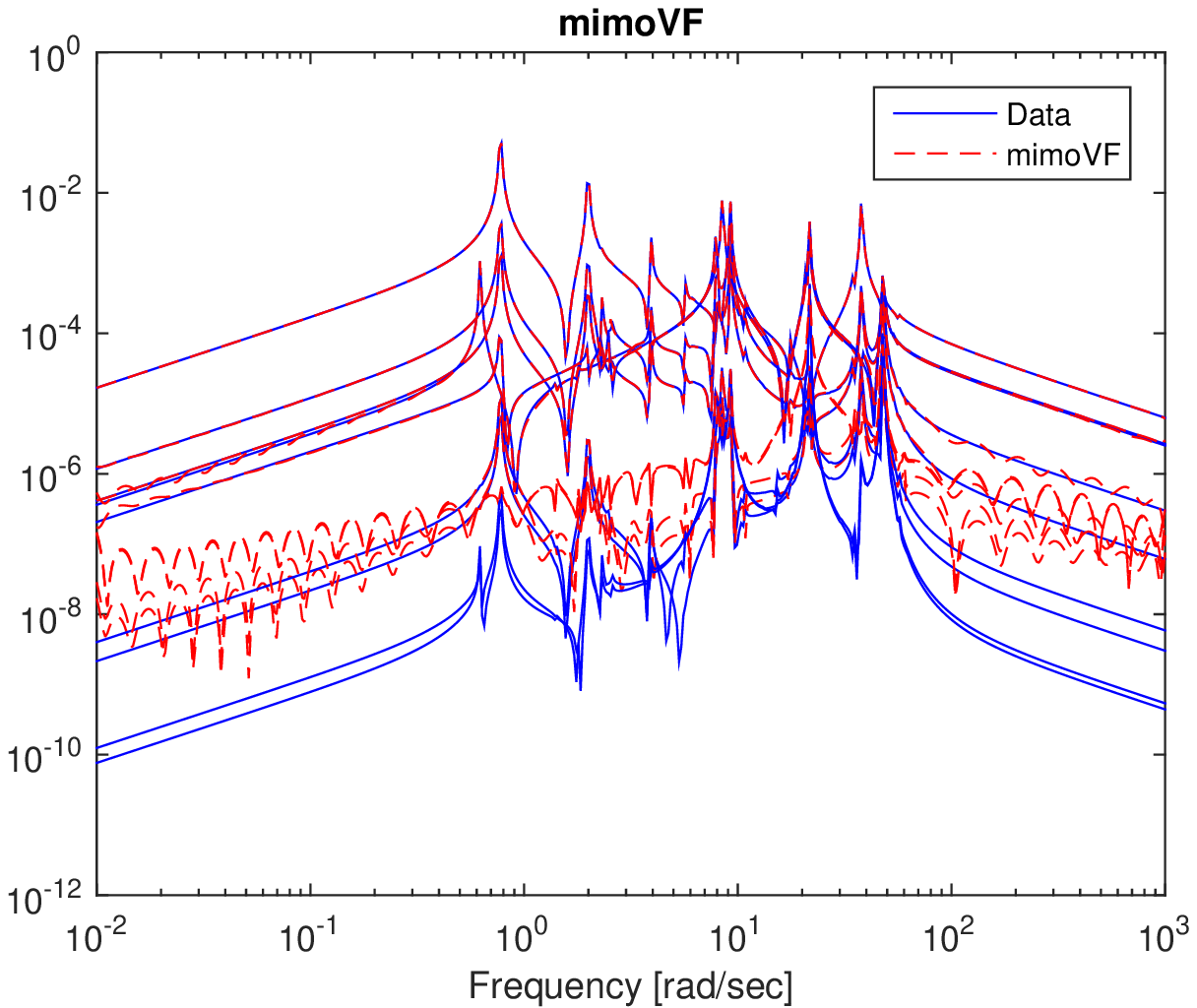}
		\includegraphics[width=2.9in,height=2.8in]{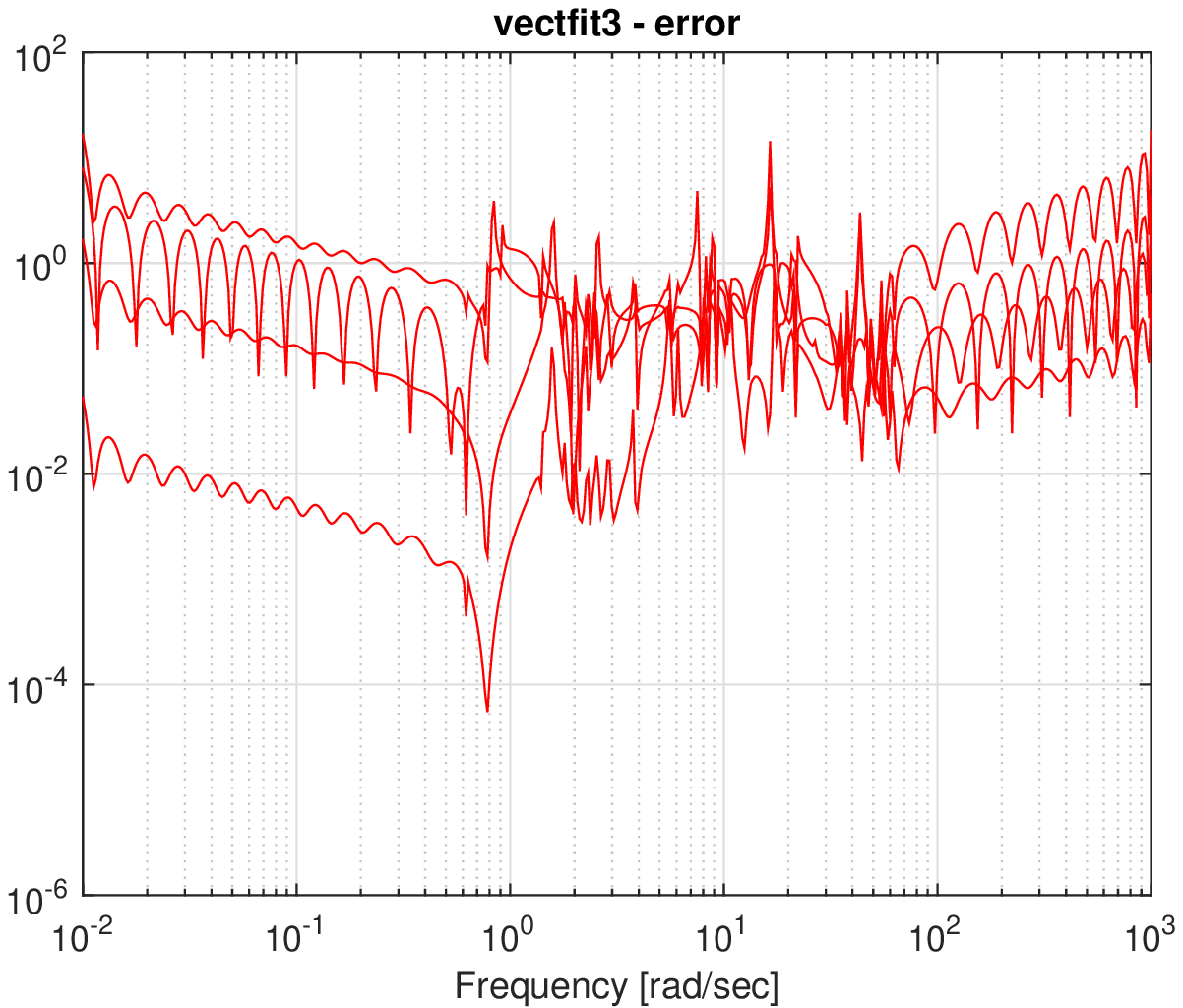}
		\includegraphics[width=2.9in,height=2.8in]{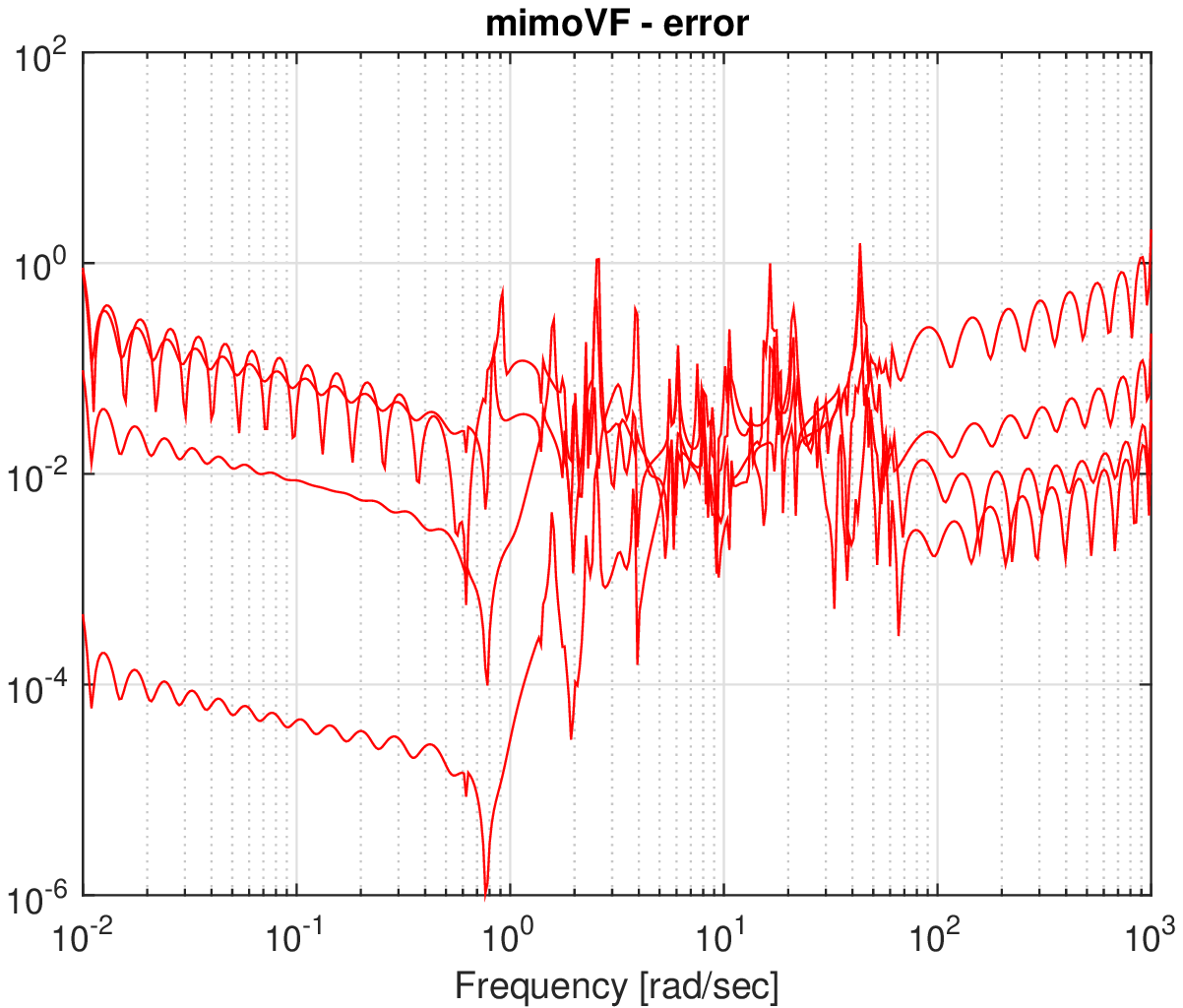}	
	\end{center}
	%	\includegraphics[width=5.9in,height=2.8in]{VFZSSL1000r90_truncated.eps}
	%\vspace{-7mm}
	\caption{\label{FIG-compare-to-vf3-ISS270_2iter_logpoles} 
		(Example \ref{NumEx-01-A}.) 
		Comparison of  \textsf{mimoVF} and the \textsf{vectfit3} on an ISS example with initial poles set as log-spaced (imaginary parts log-spaced over the frequency range). The first plot shows the output of \textsf{vectfit3}
		(with the relative error $\gamma \approx 9.47\cdot 10^{-1}$), and the second of \textsf{mimoVF} ($\gamma \approx 4.90\cdot 10^{-3}$) after one iteration. In this example, $\ell=500$ and $r=100$. The second row shows relative approximation error $|\bfH(\xi_i)_{uv}-\bfH_r(\xi)_{uv}|/|\bfH(\xi_i)_{uv}|$ for the four dominant input-output pairs over all samples.  Since the order of the underlying system is $n=270$ and $p=m=3$, $r=100$ should provide good approximation.
		%The relative $\Hardy_2$ error of $\bfH_r$ computed by \textsf{mmoVF} is
		% $\chi\approx 9.53e-02$.
	}
\end{figure} 
}
\end{example}
\subsubsection*{A regularization approach for residue extraction}
Even though \textsf{mimoVF} performed relatively well in Example \ref{NumEx-01-A},
 some of the less-dominant input/output pairs were not captured as accurately as one would prefer; see Figure \ref{FIG-compare-to-vf3-ISS270_2iter_logpoles}. In this section, we consider a regularization technique to further improve the final residue extraction step in \textsf{mimoVF}. 
Recall Line 9. of Algorithm 
 \ref{zd:ALG:MIMOVF-basic} to compute the final residues: solving $\bfPhi = \bfB_{[11]}^{-1} \mathbf{s}_1$. As noted in Remark \ref{REM:OnLine9} this corresponds to the simultaneous determination of the residue matrices by solving  $\|\dqw \left(\Cauchy^{(k+1)}\varPhi^{(k+1)}(u,v,:) -  \mathbb{S}(u,v,:)\right)\|_2\longrightarrow\min$, $u=1,\ldots,p$, $v=1,\ldots,m$.	
 To simplify the notation, denote this \textsf{LS} problem by $\|\dqw \Cauchy x -h\|_2\longrightarrow \min$, where
  $\Cauchy=\Cauchy_{\bfxi,\bflambda}$ is a Cauchy matrix as in (\ref{eq:A=PC}), $h$ is the corresponding scaled right-hand side,
  $\bflambda$ is closed under conjugation and  and the solution vector should also be closed under conjugation. 
Such a constrained problem can be replaced by an equivalent unconstrained \textsf{LS} problem
\begin{equation}
\left\| \begin{pmatrix} \dqw \Cauchy_{\bfxi,\bflambda} \cr
\dqw \Cauchy_{\cc{\bfxi},\bflambda} \end{pmatrix} x  - \begin{pmatrix} h \cr \cc{h} \end{pmatrix} \right\|_2 \equiv \|\widehat{\Cauchy} x - \widehat{h} \|_2 \longrightarrow\min 
\end{equation}
with the coefficient matrix again of the diagonally scaled Cauchy structure,
$\widehat{\Cauchy} = (\dqw\oplus\dqw) \Cauchy_{(\bfxi,\cc{\bfxi}),\lambda}$.

The SVD of $\widehat{\Cauchy}$ can be computed to high relative accuracy based on the  pivoted LU decomposition $\Pi_1 \widehat{\Cauchy}\Pi_2=LDU$, where
each entry (including the tiniest ones) of the computed factors $\widetilde{L}$, $\widetilde{D}$, $\widetilde{U}$ is computed to high relative accuracy, and $\widetilde{L}$, $\widetilde{U}$ are well conditioned. The ill-conditioning of
$\widehat{\Cauchy}$ is revealed in the diagonal matrix $\widetilde{D}$.
This decomposition can be used immediately in an \textsf{LS} solver \cite{Castro-GonzalezCDM13}, or it can be used to compute an accurate SVD 
\cite{dgesvd-99}, \cite{Demmel-99-AccurateSVD}, \cite{drm-ves-VW-1, drm-ves-VW-2}
which is then used to compute an approximate \textsf{LS} solution. 
$\widehat{\Cauchy}$ can be severely ill-conditioned so that small changes
of $\xi_i$s and $\lambda_j$s can cause significant perturbation of the SVD. However, we can consider the values of 	$\xi_i$s and $\lambda_j$s, as stored in the machine memory, as exact and attempting to compute accurate SVD is justified.

Let $\widehat{\Cauchy}=W\Sigma V^*$ be the SVD and let the unique\footnote{Since all nodes are distinct and the poles are assumed simple, the matrix is of full column rank.} \textsf{LS} solution be
$x=V\Sigma^\dagger W^* = \sum_{i=1}^r v_i (w_i^*\widehat{h})/\sigma_i$.
Unfortunately,  an accurate SVD is  not enough to have the \textsf{LS} solution computed to high relative accuracy, and additional regularization techniques must be deployed. This is in particular important if the right-hand side is contaminated by noise. In the Tichonov regularization, we choose $\mu\geq 0$ and use the solution of $\|\widehat{\Cauchy}  x - \widehat{h}\|_2^2 + \mu^2 \|x\|_2^2\rightarrow\min$, explicitly computable as
\begin{equation}\label{eq:xmu}
x_\mu = \sum_{i=1}^r \frac{\sigma_i}{\sigma_i^2+\mu^2} (w_i^*\widehat{h}) v_i .
\end{equation}
The parameter $\mu$ can be further adjusted using the Morozov discrepancy principle \cite{morozov}, i.e., to achieve  $\|\widehat{\Cauchy}  x_\mu - \widehat{h}\|_2\approx 
\nu$, where $\nu$ is the estimated level of noise $\delta\widehat{h}$ in the right-hand side, $\nu\approx \|\delta\widehat{h}\|_2$. 

%\bibitem{moro}
%Morozov, V. (1966) On the solution of functional equations by the method of
%regularization. \emph{Soviet Mathematrics Doklady}, ~(7), 414 -- 417.

\begin{example}\label{NumEx-01-B} 
{\em
Here we continue Example \ref{NumEx-01-A}, use the same data, and apply \textsf{vecfit3} and \textsf{mimoVF} where we use the preceding regularization approach in the final residue computation step.
 Using accurate SVD of $\widehat{\Cauchy}$ we compute $x_\mu$ as in (\ref{eq:xmu}) with an ad hoc choice of  $\mu=10^{-3}$. The results after only {\it one} iteration of \textsf{mimoVF} and {\it two} iterations of \textsf{vectfit3} are shown on Figure \ref{FIG-compare-to-vf3-ISS270_2iter_logpoles2}.  \textsf{mimoVF}
 still has smaller \textsf{LS} error;  $4.90\cdot 10^{-3}$ compared to $1.77\cdot 10^{-1}$ due to \textsf{vectfit3}.
 It is important to note that  \textsf{mimoVF} achieves this better performance after only one iteration. \textsf{vectfit3} has the error of $1.77\cdot 10^{-1}$ after two iterations. 
 This shows that a more robust implementation may reduce the total number of iterations needed to reach satisfactory approximation; thus reducing the overall computational complexity.
	\begin{figure}
		\begin{center}
			\includegraphics[width=2.9in,height=2.8in]{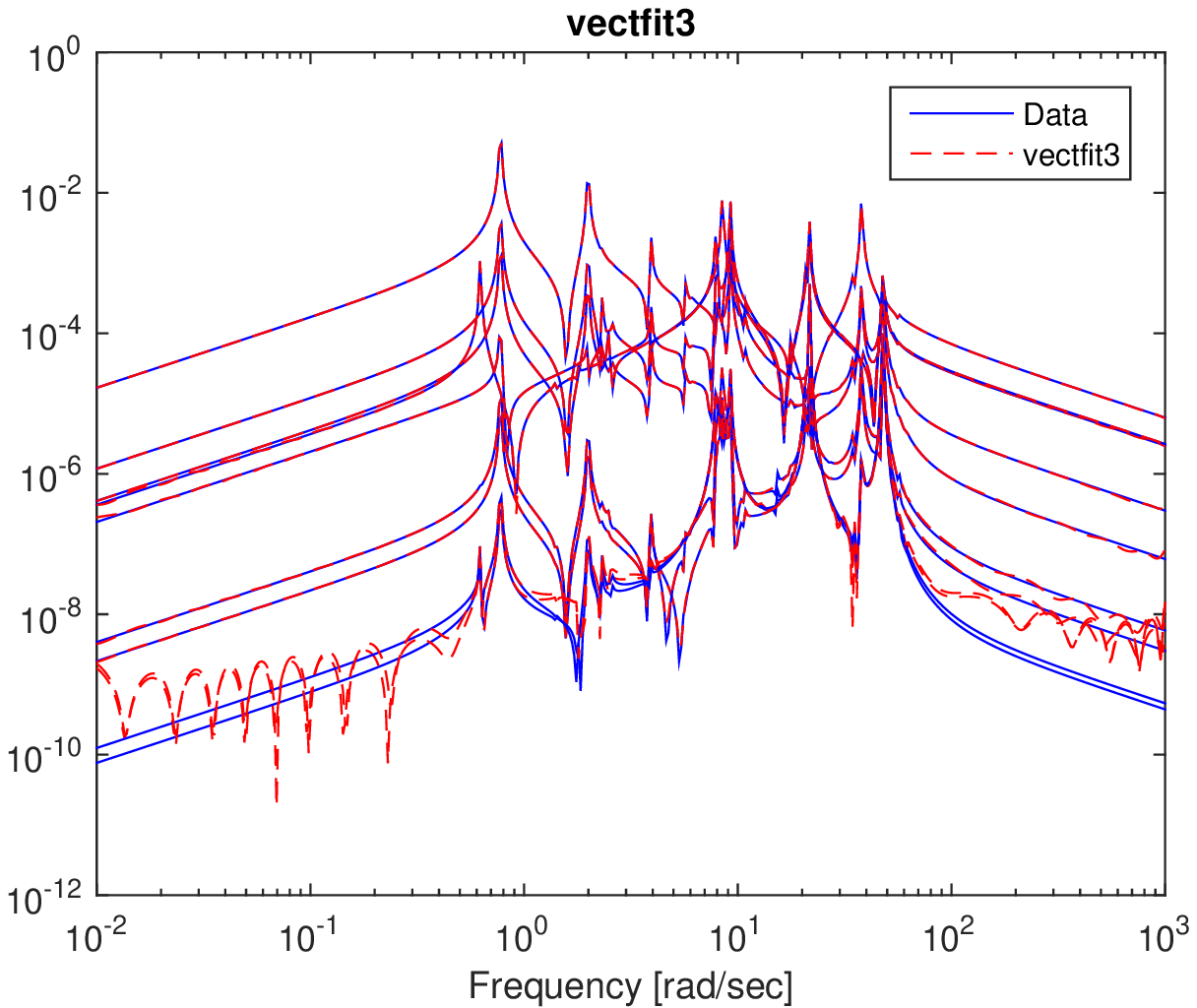}
			\includegraphics[width=2.9in,height=2.8in]{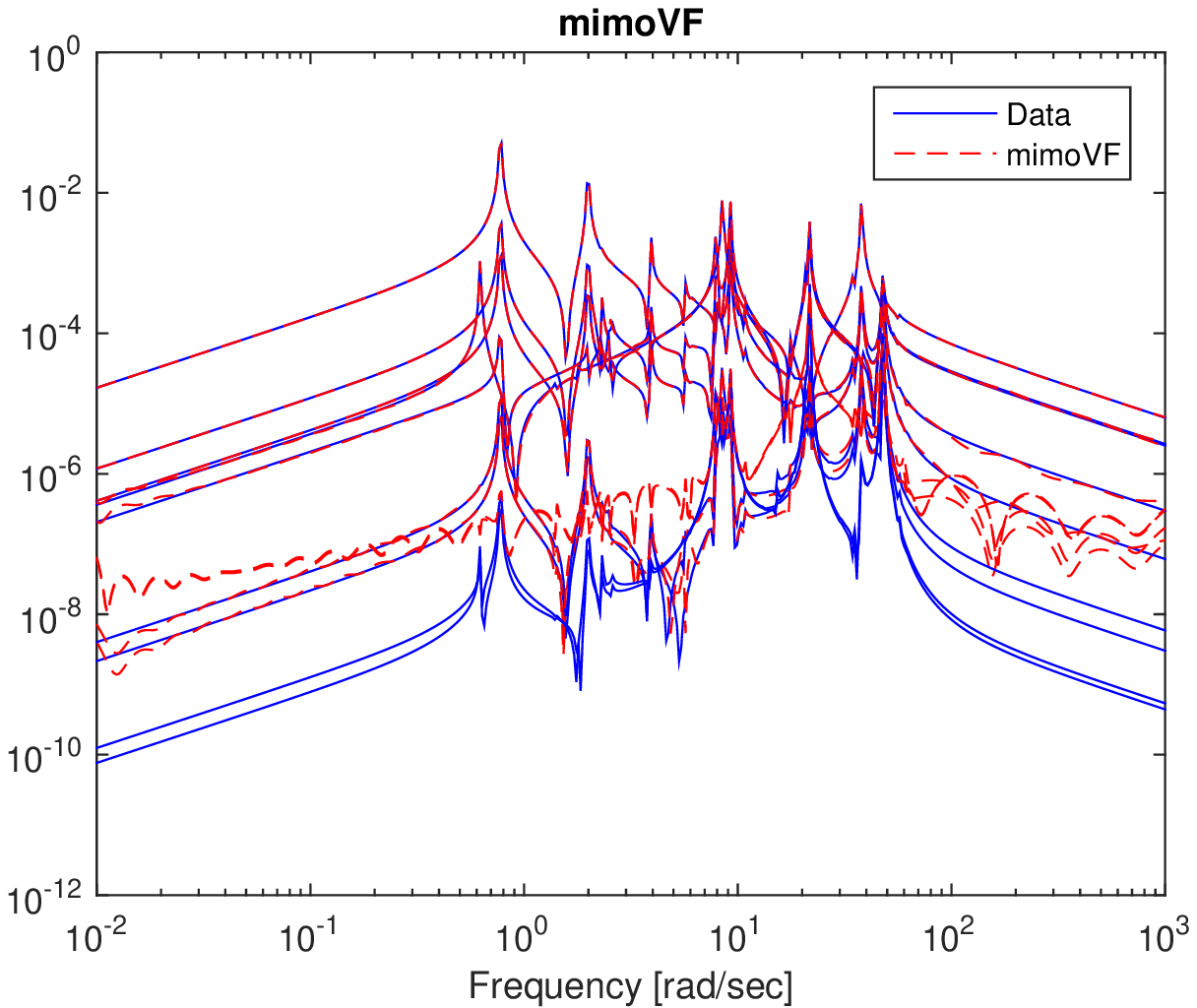}
		\end{center}
		%	\includegraphics[width=5.9in,height=2.8in]{VFZSSL1000r90_truncated.eps}
		%\vspace{-7mm}
		\caption{\label{FIG-compare-to-vf3-ISS270_2iter_logpoles2} 
			Comparison of  \textsf{mimoVF} and the \textsf{vectfit3} on an ISS example with initial poles set as log-spaced (imaginary parts log-spaced over the frequency range).   In this example, $\ell=500$ and $r=100$. The first plot shows the output of \textsf{vectfit3} after \underline{two} iterations
			(with the relative error $\gamma \approx 1.77\cdot 10^{-1}$), and the second of \textsf{mimoVF} ($\gamma \approx 4.90\cdot 10^{-3}$) after \underline{one} iteration.  
		}
	\end{figure} 
	}
\end{example}

\begin{remark}
{\em
	The LDU decomposition, which is the initial step of the accurate SVD, can also be used in the pole identification phase to compute accurate QR factorization of Cauchy matrices at all stages of the computation. We have not included those modifications, for the sake of brevity of the presentation. Interested readers can find the details of such an approach in \cite{drm-00-angles}.	
}
\end{remark}

\subsubsection{Stopping criterion} \label{SS=StopCrit}
It should be noted that the stopping criterion in the \textsf{VF} framework is rather vaguely specified. To the best of our knowledge, the \textsf{VF} literature does not provide a precise and numerically justified strategy of halting the iterations. {For instance, \textsf{vectfit3} allows only running for a given
fixed number of iterations.}
	
Following our recent analysis \cite{Drmac-Gugercin-Beattie:VF-2014-SISC}, we propose to  declare the $|{\dres}_j^{(k+1)}|$s ``small enough" if
\begin{equation}
\sum_{j=1}^r \frac{|{\dres}_j^{(k+1)}|}{|\Re{e}(\pol_j^{(k+1)})|}\equiv \theta^{(k+1)} \leq \bfeps, \;\;\mbox{where $\bfeps$ is a suitable threshold.}
\end{equation}
%
%\marginpar{\raggedleft\textcolor{blue}{\tiny Does it appear \\ that $\theta^k$ \\ is %monotone \\ decreasing in k ? \\ Any exceptions \\ noted ?  \\} } 
%
This seems appropriate because $\max_{s\in\imunit\R}|d^{(k+1)}(s) -1|\leq \theta^{(k+1)}$, and in (\ref{eq:VF-LS-problem}) we can write
$$
 \bfS^{(i)} (1 + \sum_{j=1}^r \frac{{\dres}_j^{(k+1)}}{\xi_i-\pol_j^{(k+1)}}) = 
 \bfS^{(i)} + \Delta \bfS^{(i)},\;\; \|\Delta \bfS^{(i)}\|_F \leq \theta^{(k+1)} \|\bfS^{(i)}\|_F,
$$
thus interpreting this as introducing backward perturbation into the data. In fact, this interpretation can be a guidance for choosing the threshold value $\bfeps$ by following a discrepancy principle, i.e., so that this backward error matches the estimated size of the noise level on the input.

\subsection{Guidelines for efficient implementation}\label{SSS:LessQRs}

We now further discuss implementation details that are relevant for an efficient
software implementation of \textsf{mimoVF} or  Algorithm \ref{zd:ALG:MIMOVF-basic} in general. 
Recall that \textsf{VF} for MIMO systems with $m$ input and $p$ inputs will require 
$p\cdot m$ QR factorizations (\ref{eq:concurrent-QRs}) of size $\ell \times (2r)$ before 
advancing toward finding $\dres_j^{(k+1)}$. It is clear that this is  a demanding computational challenge even for moderate 
$m$ and $p$; e.g., in the case $m=p=100$, it will require $p\cdot m =10000$ QR factorizations of size $\ell \times (2r)$. Since these factorizations are independent, they
can be very efficiently parallelized and the whole computation can be optimized for a multi-core computing machinery.
This has been  nicely described by Chinea and  Grivet-Talocia \cite{Chinea-G.Talocia-2011}, who showed a nearly ideal speedup on
a four quad-core architecture. 

\subsubsection{Efficient computation of  $\mathbf{B}_{[22]}$}\label{SS=B22}
It was pointed out earlier that the QR factorizations in (\ref{eq:concurrent-QRs}) are independent of $(u,v)$ in the first $r$ columns and $(R^{(k+1)})_{11}$ from (\ref{eq:PHIuv}) can be computed by a single QR factorization, optionally with column pivoting, of $\dqw \Cauchy^{(k+1)}$, i.e., 
%
%\marginpar{\raggedleft\textcolor{blue}{\tiny  Why not use \\ previous $Q R$ \\ notation ? \\} } 
%
\textcolor{black}{$\dqw\Cauchy^{(k+1)} \Pi = V^{(k+1)}\left( \begin{smallmatrix} T^{(k+1)} \cr 0 \end{smallmatrix}\right)$.
  } %% composing the diagonally scaled Cauchy matrix $\dqw \Cauchy^{(k+1)}$, . 
% This means that all upper triangular matrices $(R^{(k+1)}_{uv})_{11}$ from (\ref{eq:PHIuv}) (denoted by $\ast$ in (\ref{eq:structure-partialQR})) are equal, and thus {all can be computed by a single QR factorization} of $\dqw \Cauchy^{(k+1)}$, i.e.,
%$\dqw\Cauchy^{(k+1)} = V^{(k+1)}\left( \begin{smallmatrix} T^{(k+1)} \cr 0 \end{smallmatrix}\right)$. 
Further, the introduction of rank-revealing column pivoting $\Pi$ in this factorization incurs a negligible overhead, while preserving the structure (\ref{eq:structure-partialQR}). We discussed in \S \ref{SS=LS+QR} that this pivoting is very important for numerical robustness of the \textsf{LS} solution as well. In an \textsf{LAPACK}-style implementation, the matrix $V^{(k+1)}$ can be computed and stored in form of $r$ Householder vectors (using \textsf{Xgeqp3}), and then, using 
\textsf{Xormqr}, $(V^{(k+1)})^*$ can be concurrently applied to all $-\dqw D^{(uv)}\Cauchy^{(k+1)}$, $u=1,\ldots, p$, $v=1,\ldots, m$. Then, it only remains to compute the QR factorizations of the $(r+1:\ell,r+1:2r)$ submatrices of 
$-(V^{(k+1)})^* \dqw D^{(uv)}\Cauchy^{(k+1)}$. One should note that in a blocked QR factorization a similar computation is done anyway in the process of computing the QR factorizations (\ref{eq:concurrent-QRs}). 
The computed triangular $r\times r$ factors are the $(2,2)$ blocks in (\ref{eq:concurrent-QRs}) that build
the matrix $\mathbf{B}_{[22]}^{(k+1)}$. 
Hence, the saving of this modified approach is equivalent to the cost of $(p m -1)$ QR factorizations of size $\ell\times r$, or, approximately, $const\cdot (pm-1)\ell r^2$. The total work
on the QR factorizations (\ref{eq:concurrent-QRs}) without this modification is $4\cdot const \cdot pm\ell r^2$.

Further, when we are solving only for the $\dres_j^{(k)}$ during the \textsf{VF} iterations,  the elements 
of $\mathcal{Q}^* A$, denoted by $\times$ in (\ref{eq:concurrent-QRs}) and (\ref{eq:structure-partialQR}) are not used in the pole identification phase, but they are computed as $(R^{(k+1)}_{uv})_{12}$ parts of the
QR factorizations (\ref{eq:concurrent-QRs}). On the other hand, once the poles are fixed, the \textsf{LS} problem is solved with the approximant of the form (\ref{eq:H_r(k)}),
but with the unit denominator, $d^{(k)}(s)\equiv 1$, or,  equivalently, with
$\dres_j^{(k+1)}=0$, $j=1,\ldots, r$. This means that, for computing (\ref{eq:H_r(k)-final}) by the Algorithm \ref{zd:ALG:MIMOVF-basic},
we do not compute the matrices  $(R^{(k+1)}_{uv})_{12}$, which further reduces the complexity. 
%
%Implementing this in \textsc{matlab} to reduce runtime is not straightforward, but a \textsf{FORTRAN} or \textsf{C++} procedure can be easily derived from the
%
Our implementation of this more efficient approach is based on adapting the \textsf{LAPACK}'s functions \textsf{Xormqr}, \textsf{Xlarft} and \textsf{Xlarfb}. 
%The reason is that the \textsc{matlab}'s function \textsf{qr} returns only the tall and skinny or the full unitary factor, while the \textsf{LAPACK}'s function returns the Householder vectors and uses a block representation of their product. One can do a reverse engineering on that and produce a required modified routine.

\subsubsection{Locally pivoted factorization}
In the presence of noise and ill-conditioning, pivoting is essential when using the QR factorization.
Hence, we propose to include pivoting in the procedure outlined in \S \ref{SS=B22}.
More precisely, the QR factorization of $\dqw \Cauchy^{(k+1)}$ is computed with column pivoting, i.e.,
$\dqw\Cauchy^{(k+1)}\Pi^{(k+1)} = V^{(k+1)}\left( \begin{smallmatrix} T^{(k+1)} \cr 0 \end{smallmatrix}\right)$. This enhances the accuracy of the computed residues in line 9. of Algorithm \ref{zd:ALG:MIMOVF-basic}. In a software implementation, Line 9. is reshaped into the \textsf{LS} problem with the coefficient matrix $T^{(k+1)}$ and with
$m\cdot p$ right-hand sides, for all input-output pairs. Optionally, one can use the truncation discussed in \S \ref{SS=LS+QR}, or the accurate SVD as explained in \S \ref{SS=HALA}.

Further, we also advocate the use of pivoting when computing the QR factorizations of the $(r+1:\ell,r+1:2r)$ submatrices of 
$-(V^{(k+1)})^* \dqw D^{(uv)}\Cauchy^{(k+1)}$. This increases the accuracy of the computed matrix $\mathbf{B}_{[22]}^{(k+1)}$ in line 6. For details how pivoting influences the accuracy of the rows of the computed triangular QR factor we refer to \cite{Drmac-Bujanovic-2008}. Furthermore, this may allow (in the cases of numerical rank deficiency, as revealed by the pivoted QR factorization and discussed in \S \ref{SS=LS+QR}) to set  certain numbers of rows of  $(R^{(k+1)}_{uv})_{22}$ to zero and thus increased the number of zero rows in $\mathbf{B}_{[22]}^{(k+1)}$ and reduced the complexity of Line 6 in Algorithm \ref{zd:ALG:MIMOVF-basic}, 
where, as described in \S \ref{SS=LS+QR}, the \textsf{LS} solver starts with the QR factorization with column pivoting. 

The pivoted QR factorization of the tall and skinny matrix $\mathbf{B}_{[22]}^{(k+1)}$ can be computed by e.g., first computing the QR without pivoting using the techniques of \cite{DBLP:journals/corr/abs-1301-1071}, and then computing the pivoted QR factorization of the computed $r\times r$ triangular factor, or e.g. as in \cite{Demmel:EECS-2013-46}.		
These approaches become particularly attractive if $p\cdot m$ and $r$ are large.

\begin{remark}
	{\em 
Interestingly, we do not need to compute the $(R^{(k+1)}_{uv})_{22}$s. Instead, we can build the matrix $\mathbf{B}_{[22]}^{(k+1)}$ from the $(r+1:\ell,r+1:2r)$ submatrices of 
$-(V^{(k+1)})^* \dqw D^{(uv)}\Cauchy^{(k+1)}$. This will lead to increased number of rows in $\mathbf{B}_{[22]}^{(k+1)}$, but overall it reduces the complexity with potential gain increased if the QR factorization of $\mathbf{B}_{[22]}^{(k+1)}$ is computed  using the strategies of \cite{DBLP:journals/corr/abs-1301-1071}, \cite{Demmel:EECS-2013-46}.
	}
\end{remark}
%%%%%%%%%%%%%
%\begin{remark}
%	{\em 
%The QR factorizations described here and in \S \ref{SS=B22} can be computed to higher accuracy at a price of additional pivoted LDU decomposition, as discussed in \S \ref{SS=HALA}. Interested readers can find the details of such an approach in \cite{drm-00-angles}.
%}
%\end{remark}

\section{Numerical Quadrature in \textsf{mimoVF} for discretized $\Hardy_2$ approximation}
\label{SS:quad}
The framework for \textsf{mimoVF} is  based on the algebraic least squares (LS) error minimization (\ref{basicOptProb}) where one usually chooses the weights $\rho_j=1$ and the nodes $\xi_j$ are usually selected heuristically. In \cite{Drmac-Gugercin-Beattie:VF-2014-SISC}, for single-input/single-output (SISO) systems, we have shown that 
with the underlying dynamical system in mind, reformulating the discrete \textsf{LS} problem
 as discretization of an underlying continuos  $\Hardy_2$ error measure and then choosing the nodes and weights by an appropriate numerical quadrature improves the performance of \textsf{VF} significantly. The same conclusion holds in the MIMO case as well since, once the common set of poles has been determined,  \textsf{mimoVF} works separately on each input--to-output pair. We illustrate these considerations briefly in this section.

\subsection{$\Hardy_2$ approximation and numerical quadrature}

 The algebraic least squares error is closely related to 
% the continuous $L_2$ norm on $\imunit\mathbb{R}$, that is, 
 the $\Hardy_2$ system norm. 
%{Hr_sum_cbt}
More precisely, consider the space $\Hardy_{2,+}^{p\times m}$ of $p\times m$ matrix functions $\bfM(s)$, analytic in the open right half-plane $\Cplx_+ =\{s\in\Cplx\; : \; \Im(s)> 0\}$, such that $\sup_{x>0}\int_{\infty}^{\infty} \|\bfM(x+\imunit y)\|_F^2 dy < \infty$. The space
% By a Fatou theorem, $M(s)$ can be identified with its boundary function $M(\imunit\omega)$, $\omega\in\mathbb{R}$, and 
$\Hardy_{2,+}^{p\times m}$%, a closed subspace of $L^2(\imunit\mathbb{R},\Cplx^{p\times m})$, 
is a   Hilbert space
 with the associated inner product and norm defined by
\begin{equation}
\langle\bfM_1, \bfM_2\rangle_{\Hardy_2} =\frac{1}{2\pi} \int_{-\infty}^{\infty} \mathrm{Trace}\left(
\cc{\bfM_1(\imunit\omega)}\bfM_2(\imunit\omega)^T\right)\,d\omega,
\;\;\;\|\bfM\|_{\Hardy_2} = \left(\frac{1}{2\pi} \int_{-\infty}^{\infty} \left\|\bfM(\imunit\omega)\right\|_F^2\,d\omega \right)^{1/2} .
\end{equation}
 The 
$\Hardy_2$ approximation problem is, then, to find a degree-$r$ rational approximant, $\bfH_r(s)$, that minimizes the 
$\Hardy_2$ error norm $\|\bfH - \widetilde{\bfH}_r\|_{\Hardy_2}$  over all degree-$r$ rational function  $\widetilde{\bfH}_r(s)$.
Such an optimal rational approximant must satisfy certain Hermite tangential  interpolation conditions;
for details we refer to \cite{Ant2010imr,Gugercin_Antoulas_Beattie:2008:IRKA}. 
  The \emph{Iterative Rational Krylov Algorithm} (\textsf{IRKA}) of Gugercin et al. \cite{Gugercin_Antoulas_Beattie:2008:IRKA}
   is a numerically effective iterative algorithm that constructs degree-$r$ rational approximatants satisfying the $\Hardy_2$-optimality conditions.  
   
 Our goal in this section is to repeat the success of \cite{Drmac-Gugercin-Beattie:VF-2014-SISC} for SISO systems, i.e., 
 improve the performance of \textsf{mimoVF} by  formulating the discrete \textsf{LS} measure as discretization of  the continuous $\Hardy_2$ error. Towards this goal, approximate the $\Hardy_2$ error with a quadrature role to obtain
{\small
 \begin{eqnarray}\label{eq:H2error_discretized}
\int_{-\infty}^{+\infty} \!\!\!\| \bfH(\imunit\omega)-\bfH_r(\imunit\omega)\|_F^2 d\omega &\approx& \!\!
\sum_{j=1}^\ell \rho_j^2  \|\bfH( \xi_j) - \bfH_r( \xi_j)\|_F^2 + \rho_{+}^2\,M_{+}[|\bfH-\bfH_r|^2] +\rho_{-}^2\,M_{-}[|\bfH-\bfH_r|^2]~~~~
\end{eqnarray} }
 where $M_{\pm}[G]$ are linear functionals of $G$ that capture information about asymptotic behavior of $G$ at $\pm\infty$.
Note that the usual \textsf{VF} formulation correspond to  $\rho_{+}=\rho_{-}=0$, with all other $\rho_j=1$,
and  choosing sampling nodes $\xi_j$ to be equidistant and in complex conjugate pairs. Thus, 
 the usual \textsf{VF} objective function is
as a composite trapezoid quadrature rule for the integral in (\ref{eq:H2error_discretized}) 
approximating the $\Hardy_2$ error. As we discussed in  \cite{Drmac-Gugercin-Beattie:VF-2014-SISC}, 
 more effective quadrature options may be considered, e.g., Gauss-Legendre, Gauss-Kronrod, and Gauss-Hermite quadrature rules.  We do not go into the details of what quadrature method to choose here; since our main point is just to illustrate that  \textsf{mimoVF} can perform much better once formulated as a discretized $\Hardy_2$ measure. 

\subsection{A numerical example}
Here, with a simple example, we illustrate the effect of choosing the sampling points and the weights using numerical quadrature.  We use the Clenshaw-Curtis type quadrature rule developed by Boyd \cite{Boyd-1987}. 
We use the ISS 1R module \cite{gugercin2001iss} with $m=3$ inputs and $p=3$ outputs. We  use only $100$ function evaluations ($\ell=200$) and apply
 \textsf{mimoVF} to this data set for different $r$ values. Resulting relative $\mathcal{H}_2$ errors are shown in Table \ref{tab:quad} below.
The quadrature-based selection yields the smallest error in each case; for $r=20$ and $r=30$, it leads to one order of magnitude improvements.
\begin{table}[hh]
\centering
\begin{tabular}{|l||c|c|c|c|} 
\hline
 ~\hspace{0.6cm}Method & $r=10$ & $r=20$ & $r=30$ & $r=40$ \\ \hline 
\textsf{mimoVF} without Quadrature &  $3.4104 \times 10^{-1}$ &  $3.1775 \times 10^{-1}$  & $1.2778 \times 10^{-1}$ & $ 5.0257 \times 10^{-2}$ \\ \hline 
\textsf{mimoVF} with\phantom{out} Quadrature  & $2.5209\times 10^{-1}$  & $4.6074 \times 10^{-2}$& $3.3226 \times 10^{-2}$  & $ 2.1436 \times 10^{-2}$\\ \hline 
\end{tabular}
\caption{Effect of quadrature nodes and weights on \textsf{mimoVF}}
\label{tab:quad}
\end{table}

% % % % % % % % % % % % % % % % % % % % % % % % % % % % % % % % % % % % %  WIP % % % % % % % %
%\section{Computing an $r$-th order approximation} \label{sec_rank1_VFMIMO}
\section{Controlling the McMillan Degree} \label{sec_rank1_VFMIMO}
Let $\{\respm_j\}_1^r\subset \IC^{p \times m}$ and $\{\lambda_j\}_1^r\subset \IC$ denote the final set of matrix residues and poles, respectively, resulting from \textsf{mimoVF}. To ease notational clutter, we drop the iteration index $k$. 
The associated rational matrix approximant can be represented as
 \begin{equation}\label{eq:VF-ss-realization}
\bfH_r(s) =  \sum_{j=1}^r \frac{\respm_j}{s-\pol_j} =  \left(\begin{smallmatrix}
 \Id_p &  \ldots & \Id_p
 \end{smallmatrix}\right)\!\! \left(\begin{smallmatrix}
 \frac{\Id_p}{s-\lambda_1} &  & \cr
 & \ddots & \cr  &  & \frac{\Id_p}{s-\lambda_r}
 \end{smallmatrix}\right) \!\! \left(\begin{smallmatrix}\respm_1 \cr \vdots \cr \respm_r \end{smallmatrix} \right)
 = \left(\begin{smallmatrix}\respm_1 & \ldots & \respm_r \end{smallmatrix}\right)\!\!
 \left(\begin{smallmatrix}
 \frac{\Id_m}{s-\lambda_1} &  & \cr
 & \ddots & \cr  &  & \frac{\Id_m}{s-\lambda_r}
 \end{smallmatrix}\right) \!\! \left(\begin{smallmatrix}
 \Id_m \cr  \vdots \cr \Id_m
 \end{smallmatrix}\right)
 \end{equation}
If $\bfH_r(s)$ has simple poles, then $\bfH_r(s)$ has nominal McMillan degree $\mathsf{deg}(\bfH_r(s))=\sum_{j=1}^r \mathsf{rank}(\respm_j)$. Evidently, $r\leq \mathsf{deg}(\bfH_r(s)) \leq r\min(p,m)$.  Note that $\mathsf{deg}(\bfH_r(s))$ will be strictly larger than $r$ unless the residues, $\respm_j$, \emph{all} have rank $1$, and indeed, $\mathsf{deg}(\bfH_r(s))$ can  be potentially much larger than $r$ if either the input space or output space has significant dimension.  McMillan degree is a proxy for the complexity involved in evaluating an approximant, therefore it is generally desirable and sometimes necessary to reduce the McMillan degree of a rational approximant to the target value of $r$, while retaining its approximating quality as much as possible. 
% The reason for $\bfH_r(s)$ to have high McMillan degree is that the residues $\respm_j$ are not necessarily rank-$1$. 
One straightforward method to accomplish this is to use truncation as suggested in \cite{Gustavsen-Semlyen-2004}: For each $j=1,\,\ldots,\,r$, find the best rank-one
approximation of $\respm_j \approx \ c_j\, b_j^T$ where $c_j\in \IC^{p}$ and $b_j \in \IC^{m}$. 
%where $c_j$ and $\cc{b_j}$ are, respectively,  the left and right singular vectors associated with the largest singular value of $\respm_j$. 
%
%\marginpar{\raggedright\textcolor{blue}{\tiny  This seems like \\ a minor comment \\ that should be \\ omitted since \\ we're not using \\ this approach \\ except (maybe) as\\ initialization \\} } 
%%
%\textcolor{blue}{If the poles are ordered so that $\lambda_{j+1}=\cc{\lambda_j}$, then $\respm_{j+1}=\cc{\respm_{j}}$ and the best rank one approximation of $\respm_{j+1}$ is $\| \respm_j \|\ \cc{c_j}\, \cc{b_j}^T$.  Thus, truncation preserves closure under conjugation.  }  
However, it is often the case that the residues $\respm_j$ are not close to rank--one matrices and so, truncation to rank-one residues can substantially increase the \textsf{LS} error.
 \cite{Gustavsen-Semlyen-2004} suggested using the 
\textsf{SVD} to determine the numerical ranks of the $\respm_j$s and truncating them to their respective best low-rank (not necessarily rank--one) approximations; 
%
%\marginpar{\raggedleft\textcolor{red}{\tiny  It seems that \\ either we should \\ comment further \\ on Gauss-Newton \\ (even if just \\ to say that \\ implementation is \\ not straightforward \\ and convergence \\ behavior may \\ be erratic) \\ or not mention \\ it at all \\} } 
%
{\cite{Gustavsen-Semlyen-2004} also proposed Gauss-Newton correction, but no details on how to proceed in this direction were provided.}

We propose here two different approaches toward retaining rank-$1$ residues, allowing us to achieve a true McMillan degree of $r$ while keeping the approximation quality as high as possible.  The first approach, presented in \S \ref{ss=ALS-correction}, is based on a nonlinear least-squares minimization, the second one, presented in \S \ref{SS::H2-res-corr}, combines \textsf{mimoVF} with well-established optimal systems-theoretic model reduction methodologies.

%\begin{remark}
%{\em 
%The starting approximation (\ref{Hr_sum_cbt}) is not necessarily
%the truncated full \textsf{VF} approximant. We can initialize (\ref{Hr_sum_cbt}) using other available information, or even at random and then deploy an iterative refinement procedure, see \S \ref{ss=ALS-correction}, and \S \ref{SS::H2-res-corr}. 
%}
%\end{remark}
%\begin{remark}
%{\em
%The issue of the residues returned by \textsf{VF} has been also  addressed in \cite{Gustavsen-Semlyen-2004} by the standard technique of using the 
%\textsf{SVD} to determine the numerical ranks of the $\respm_j$s and truncating them to their respective best low-rank (not necessarily rank--one) approximations. Further, \cite{Gustavsen-Semlyen-2004} proposes a Gauss-Newton correction, but without providing details.
%}
%\end{remark}
%$$
%\| \underbrace{(\Cauchy \otimes \Id_p)}_{\mathcal{M}} \begin{pmatrix} c_1 b_1^T \cr \vdots \cr c_rb_r^T\end{pmatrix} - 
%\begin{pmatrix} H(\xi_1) \cr \vdots \cr H(\xi_\ell) \end{pmatrix}\|_F^2
%$$

\subsection{Rank-one residue correction via Alternating Least Squares}\label{ss=ALS-correction}
%\subsection{ALS correction of the truncated \textsf{VF} solution}\label{ss=ALS-correction}
We seek an optimal rational approximant, $\widehat{\bfH}_r(s)$, having the same poles, $\{\lambda_j\}_1^r$, as 
the \textsf{mimoVF} approximant, but taking the form
%Let $c_jb_j^T$ be the optimal rank-$1$ approximation to $\respm_j$ obtained by SVD and consider the 
%truncated \textsf{mimoVF} approximant 
\begin{equation}\label{Hr_sum_cbt}
\widehat{\bfH}_r(s) = \sum_{j=1}^r \frac{c_j b_j^T}{s-\lambda_j} \equiv C (s\Id - \Lambda)^{-1} B^T,
\;\;\mbox{where}\;\;\Lambda=\mathrm{diag}(\lambda_j)_{j=1}^r,\;\;
\begin{array}{l}
C=\begin{pmatrix} c_1 &\ldots & c_r\end{pmatrix} \cr
B=\begin{pmatrix} b_1 &\ldots & b_r\end{pmatrix} .
\end{array}
\end{equation}
Optimality here will mean that $C$ and $B$ are chosen so that $\widehat{\bfH}_r(s)$ satisfies 
%As stated above, when $\respm_j$ far away from being rank-$1$, this truncation might lead to a poor approximation quality. Therefore,  we propose to  take (\ref{Hr_sum_cbt})  as and initial approximation and  iteratively update $B$ and $C$; or equivalently 
%$b_j$ and $c_j$ for $j=1,\ldots,r$,
% to minimize the \textsf{LS} error 
\begin{equation}\label{VFcorrect-LSerror}
\min_{C,\ B} \sum_{i=1}^\ell \| \sum_{j=1}^r  \frac{c_j b_j^T}{\xi_i-\lambda_j} - \bfH(\xi_i)\|_F^2  = 
\min_{C,\ B}  \| \underbrace{(\Cauchy \otimes \Id_p)}_{\mathcal{M}} \begin{pmatrix} c_1 b_1^T \cr \vdots \cr c_rb_r^T\end{pmatrix} - 
\begin{pmatrix} \bfH(\xi_1) \cr \vdots \cr \bfH(\xi_\ell) \end{pmatrix}\|_F^2.
\end{equation}
A weighting factor $\qw_i>0$ can also be attached to
	each sample; yet for simplicity of presentation, we take all $\qw_i=1$. 
In (\ref{VFcorrect-LSerror}), $\Cauchy\in\Cplx^{\ell\times r}$ denotes the Cauchy matrix $\Cauchy_{ij}=1/(\xi_i-\lambda_j)$.
In many applications, $\widehat{\bfH}_r(s)$ should be real-valued for real-valued $s$.  
%such as when approximating a real LTI system, 
In that case, a constraint is added that the poles $\lambda_j$ and the residues
$c_j b_j^T$ must be closed under conjugation: all non-real poles appear in complex conjugate pairs, say $\lambda_j,\lambda_{j+1}=\cc{\lambda_j}$;  $c_j$, $b_j$ are real if $\lambda_j$ is real, otherwise 
$c_{j+1}=\cc{c_j}$, $b_{j+1}=\cc{b_j}$. 
%In that case we require the
%updated residues $c_j b_j^T$ to retain the same conjugacy structure.

The nonlinearity of the \textsf{LS} error (\ref{VFcorrect-LSerror}) with respect to the variables, $B$ and $C$,
can be evaded by reformulating the problem in terms of alternating least squares (\textsf{ALS}): if $B$ (alternatively, $C$) is fixed, then (\ref{VFcorrect-LSerror}) becomes a linear least squares problem  in terms of $C$ (alternatively, $B$). 
So, we minimize alternately with respect to $C$ (holding $B$ fixed) and then with respect to $B$ (holding $C$ fixed), repeating the cycle until convergence.  We provide some algorithmic details below % in \S \ref{SS=CorrectC}  
and illustrate the effectiveness of \textsf{ALS} iteration %in \S \ref{SS=ALSex} 
with an example.   An analogous approach has been used for ``residue correction" in realization-independent (data-driven) approaches to optimal $\Hardy_2$ model reduction \cite{Beattie_Gugercin::RealizationIndependent}.
%
%\subsubsection{Correction of $C$}\label{SS=CorrectC}
%\subsection*{Correction of $C$}\label{SS=CorrectC}
\paragraph{\textbf{Correction of $C$}}%\label{SS=CorrectC}
Assume $B$ is fixed and seek an updated $C$ (with conforming conjugation symmetry) that
will minimize the \textsf{LS} error (\ref{VFcorrect-LSerror}).
 To that end, we vectorize the error matrix column-wise and write the 
$k$th column ($k=1,\ldots, m$)  of the residual matrix in (\ref{VFcorrect-LSerror}) as
\begin{equation}\label{VFcorrect-LSerror-v}
\mathcal{M} \begin{pmatrix} (b_1)_k\Id &   &    \cr 
  &  \ddots  &    \cr
  &    &  (b_r)_k \Id \end{pmatrix} \begin{pmatrix} c_1 \cr \vdots \cr c_r\end{pmatrix}
  - \begin{pmatrix} \bfH(\xi_1)e_k \cr \vdots \cr \bfH(\xi_\ell) e_k \end{pmatrix} = 
  \mathcal{M} (\mathrm{diag}(B(k,:)) \otimes \Id_p)
  \begin{pmatrix} c_1 \cr \vdots \cr c_r\end{pmatrix}
  - \begin{pmatrix} \bfH(\xi_1)e_k \cr \vdots \cr \bfH(\xi_\ell) e_k \end{pmatrix} .
\end{equation}
Stacking  all columns together, the problem becomes: minimize the Euclidean norm of
the residual 
\begin{equation}\label{VFresidualVectorized}
\begin{pmatrix}
\mathcal{M} (\Delta_1 \otimes \Id_p) \cr \hline
\vdots \cr \hline 
\mathcal{M} (\Delta_m \otimes \Id_p) 
\end{pmatrix}
\begin{pmatrix} c_1 \cr \vdots \cr c_r\end{pmatrix} - 
\begin{pmatrix} 
\bfH(\xi_1)e_1 \cr \vdots \cr \bfH(\xi_\ell)e_1 \cr \hline
\vdots \cr \hline 
 \bfH(\xi_1)e_m \cr \vdots \cr \bfH(\xi_\ell)e_m
\end{pmatrix},\;\; \mbox{where}\;\;\Delta_i = \mathrm{diag}(B(i,:)),\;\;i=1,\ldots, m, 
\end{equation}
with conjugation symmetry constraints: $c_j$ is real if $\lambda_j$ is real and
$c_{k}=\cc{c_j}$, if $\lambda_{k}=\cc{\lambda_j}$. 
Note that the number of rows above is $\ell \cdot m \cdot p$, and the number of unknowns
is $p\cdot r$. In practice, $\ell$ is much larger than $m, p, r$, and it is always assumed that $\ell\geq 2r$.

Let $\Cauchy = Q \left(\begin{smallmatrix} R \cr 0\end{smallmatrix}\right) = \widehat{Q}R$ be the QR factorization,
where $Q=\begin{pmatrix} \widehat{Q} & \breve{Q}\end{pmatrix}$, $\widehat{Q}=Q(:,1:r)$.
Then $\mathcal{M} = (Q\otimes\Id_p) (\left(\begin{smallmatrix} R \cr 0\end{smallmatrix}\right)\otimes \Id_p)$
is the QR factorization of $\mathcal{M}$. Multiplying the blocks in the residual (\ref{VFresidualVectorized})
by $(Q^* \otimes \Id_p)$ and using the block-partitioned structure of $Q$, we obtain an equivalent \textsf{LS} problem:
%\begin{equation}
%\begin{pmatrix}
%(\left(\begin{smallmatrix} R \cr 0\end{smallmatrix}\right)\otimes \Id_p) (\mathrm{diag}(B(1,:)) \otimes \Id_p) \cr \hline
%\vdots \cr \hline 
%(\left(\begin{smallmatrix} R \cr 0\end{smallmatrix}\right)\otimes \Id_p) (\mathrm{diag}(B(m,:)) \otimes \Id_p) 
%\end{pmatrix}
%\begin{pmatrix} c_1 \cr \vdots \cr c_r\end{pmatrix} - 
%\begin{pmatrix} 
%(Q^* \otimes \Id_p) \mathrm{vec}(H(:,1,:))\cr \hline
%\vdots \cr \hline 
%(Q^* \otimes \Id_p)\mathrm{vec}(H(:,m,:))
%\end{pmatrix}
%\end{equation}

\begin{equation}
\left\| \begin{pmatrix}
( R \otimes \Id_p) (\Delta_1 \otimes \Id_p) \cr 0 \cr \hline
\vdots \cr \hline 
( R\otimes \Id_p) (\Delta_m \otimes \Id_p) \cr
0
\end{pmatrix}
\begin{pmatrix} c_1 \cr \vdots \cr c_r\end{pmatrix} - 
\begin{pmatrix} 
(\widehat{Q}^* \otimes \Id_p) \mathrm{vec}(\HT(:,1,:))\cr 
(\breve{Q}^* \otimes \Id_p) \mathrm{vec}(\HT(:,1,:))\cr\hline
\vdots \cr \hline 
(\widehat{Q}^* \otimes \Id_p)\mathrm{vec}(\HT(:,m,:)) \cr
(\breve{Q}^* \otimes \Id_p)\mathrm{vec}(\HT(:,m,:))
\end{pmatrix} \right\|_F \longrightarrow \min.
\end{equation}
The blocks $(\breve{Q}^* \otimes \Id_p)\mathrm{vec}(\HT(:,i,:))$, $i=1,\ldots, m$, in the
right-hand side constitute a part of the residual that cannot be influenced with any choice of
the $c_j$s and the corresponding $(\ell-r)\cdot p \cdot m$ equations (with the corresponding zero rows in the
coefficient matrix) are dropped, i.e., only the thin QR factorization $\Cauchy=\widehat{Q}R$ is needed.
This reduces the row dimension of the problem from $\ell\cdot p\cdot m$ to $r\cdot p\cdot m$.
Using the properties of the Kronecker product, we can further simplify it to 
\begin{equation}\label{LS_stackedRs}
\left\| \left[\begin{pmatrix}
 R \Delta_1 \cr \hline
\vdots \cr \hline 
R\Delta_m 
\end{pmatrix} \otimes \Id_p \right]
\begin{pmatrix} c_1 \cr \vdots \cr c_r\end{pmatrix} - 
\begin{pmatrix} 
(\widehat{Q}^* \otimes \Id_p) \mathrm{vec}(\HT(:,1,:))\cr 
\hline
\vdots \cr \hline 
(\widehat{Q}^* \otimes \Id_p)\mathrm{vec}(\HT(:,m,:)) 
\end{pmatrix} \right\|_F \longrightarrow \min.
\end{equation}
%\begin{equation}
%\left\| \begin{pmatrix}
%( R \otimes \Id_p) (\mathrm{diag}(B(1,:)) \otimes \Id_p) \cr \hline
%\vdots \cr \hline 
%( R\otimes \Id_p) (\mathrm{diag}(B(m,:)) \otimes \Id_p) 
%\end{pmatrix}
%\begin{pmatrix} c_1 \cr \vdots \cr c_r\end{pmatrix} - 
%\begin{pmatrix} 
%(\widehat{Q}^* \otimes \Id_p) \mathrm{vec}(H(:,1,:))\cr 
%\hline
%\vdots \cr \hline 
%(\widehat{Q}^* \otimes \Id_p)\mathrm{vec}(H(:,m,:)) 
%\end{pmatrix} \right\|_F \longrightarrow \min.
%\end{equation}
%$i=1,\ldots, m$, and that these blocks remain the same throughout the process, i.e. after updating $B$, these
%will be the same for the next update of $C$. 
To solve (\ref{LS_stackedRs}) we compute the QR factorizations
\begin{equation}
R_{\boxminus} = U \begin{pmatrix} T \cr 0\end{pmatrix},\;\; R_{\boxminus}\otimes \Id_p = (U\otimes \Id_p)
(\begin{pmatrix} T \cr 0\end{pmatrix}\otimes \Id_p),\;\;
\mbox{where}\;\; R_{\boxminus} = \begin{pmatrix}
R \Delta_1 \cr \hline
\vdots \cr \hline 
R\Delta_m 
\end{pmatrix} \in\Cplx^{m\cdot r \times r} ,
\end{equation}
and, using the partition $U=\begin{pmatrix} \widehat{U} & \breve{U}\end{pmatrix}$, we reduce the
problem to solving the triangular system
\begin{equation}\label{eq:TkronI-system}
(T\otimes\Id_p) \begin{pmatrix} c_1 \cr \vdots \cr c_r\end{pmatrix} = 
(\widehat{U}^* \otimes \Id_p)
\begin{pmatrix} 
(\widehat{Q}^* \otimes \Id_p) \mathrm{vec}(\HT(:,1,:))\cr 
\hline
\vdots \cr \hline 
(\widehat{Q}^* \otimes \Id_p)\mathrm{vec}(\HT(:,m,:)) 
\end{pmatrix} .
\end{equation}
Note that only the thin QR factorization $R_{\boxminus} = \widehat{U} T$ is needed.
Folding the unknowns back into the structure of $C$ we obtain, using that 
$(\widehat{Q}^* \otimes \Id_p) \mathrm{vec}(\HT(:,i,:)) = \mathrm{vec}(\HT(:,i,:)\widehat{Q}^{*T})$,
\begin{eqnarray}
\mathrm{vec} ( C T^T )  &=&  (\widehat{U}^* \otimes \Id_p) \begin{pmatrix} 
\mathrm{vec}(\HT(:,1,:)\widehat{Q}^{*T})\cr 
\hline
\vdots \cr \hline 
\mathrm{vec}(\HT(:,m,:)\widehat{Q}^{*T})
\end{pmatrix}
 = (\widehat{U}^* \otimes \Id_p)
 \mathrm{vec}(\begin{pmatrix} \HT(:,1,:)\widehat{Q}^{*T} & \ldots & \HT(:,m,:)\widehat{Q}^{*T}\end{pmatrix}) \nonumber \\
 &=& \mathrm{vec}(\begin{pmatrix} \HT(:,1,:)\widehat{Q}^{*T} & \ldots & \HT(:,m,:)\widehat{Q}^{*T}\end{pmatrix}\widehat{U}^{*T}).
\end{eqnarray}
As an alternative to solving (\ref{eq:TkronI-system}), $C$ can be computed efficiently as the solution of a triangular matrix equation. The formula $C = \begin{pmatrix} \HT(:,1,:)\widehat{Q}^{*T} & \ldots & \HT(:,m,:)\widehat{Q}^{*T}\end{pmatrix}\widehat{U}^{*T} T^{-T}$ is rich in BLAS 3 operations and can be highly optimized. Finally, we note that the QR factorizations involved can be done with pivoting, but we omit details for the sake of simplicity.
\paragraph{\textbf{Correction of $B$}}%\label{SS=CorerctB}
If the matrix $C$ is fixed and we want to update $B$, we use the preceding procedure, with a few
simple modifications. First, transpose the residuals at each $\xi_i$ to get
$\sum_{j=1}^r  \frac{b_j c_j^T}{\xi_i-\lambda_j} - \bfH(\xi_i)^T$. As a consequence, swap the roles of
the $c_j$s and the $b_j$s, and use $\HT(i,:,:)$ instead of $\HT(:,i,:)$. The rest follows
\emph{mutatis mutandis}. 

%\begin{itemize}
%\item We will do $rm \to r$ in the $H_2$ norm; introduce the $H_2$ inner product and norm
%\item But why not then start VFMIMO with the appropriate quadrature nodes; refer to SISO, include the example Section 4.3
%\item Then, go to $rm \to r$ example
%\end{itemize}

%\subsubsection{A numerical example for \textsf{ALS} Correction}\label{SS=ALSex}
\paragraph{\textbf{A numerical example}}\label{SS=ALSex}
%\begin{example}
\label{NumEx-1}
We illustrate the usefulness of the \textsf{ALS} correction process in building a final approximant $\widehat{\bfH}_r$ that has exact McMillan degree $r$.   We use the data of Example \ref{NumEx-01}, and the output of \textsf{mimoVF} after the second iteration. The simple truncation of the residue matrices causes the \textsf{LS} error jump
from $\gamma\approx 6.45 \cdot 10^{-3}$ to $\gamma\approx 2.72$, and one step of 
\textsf{ALS} correction reduces it down to $\gamma\approx 1.41 \cdot 10^{-2}$. 
This improvement is evident in Figure \ref{FIG-compare-to-vf3-ISS270_oneiter_tr}. 
%\textcolor{blue}{It should be mentioned that it is also possible (in case of better initialization or increased number of iterations)
%
%\marginpar{\raggedleft\textcolor{blue}{\tiny  What does \\ this mean ? \\} } 
% 
%for the computed residues allow truncation without significant loss of accuracy.  } 
\begin{figure}
	\begin{center}
		\includegraphics[width=2.9in,height=2.8in]{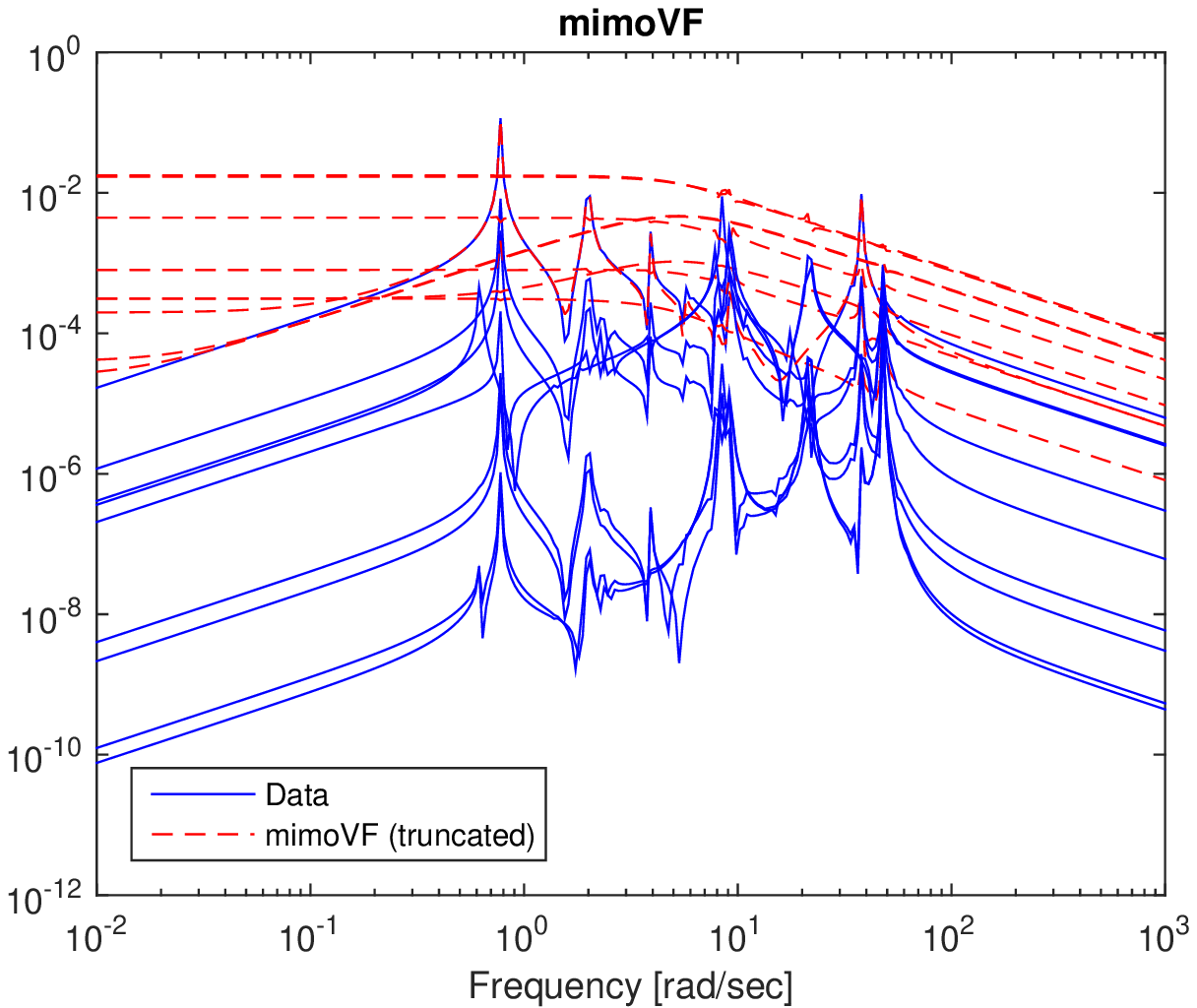}
		\includegraphics[width=2.9in,height=2.8in]{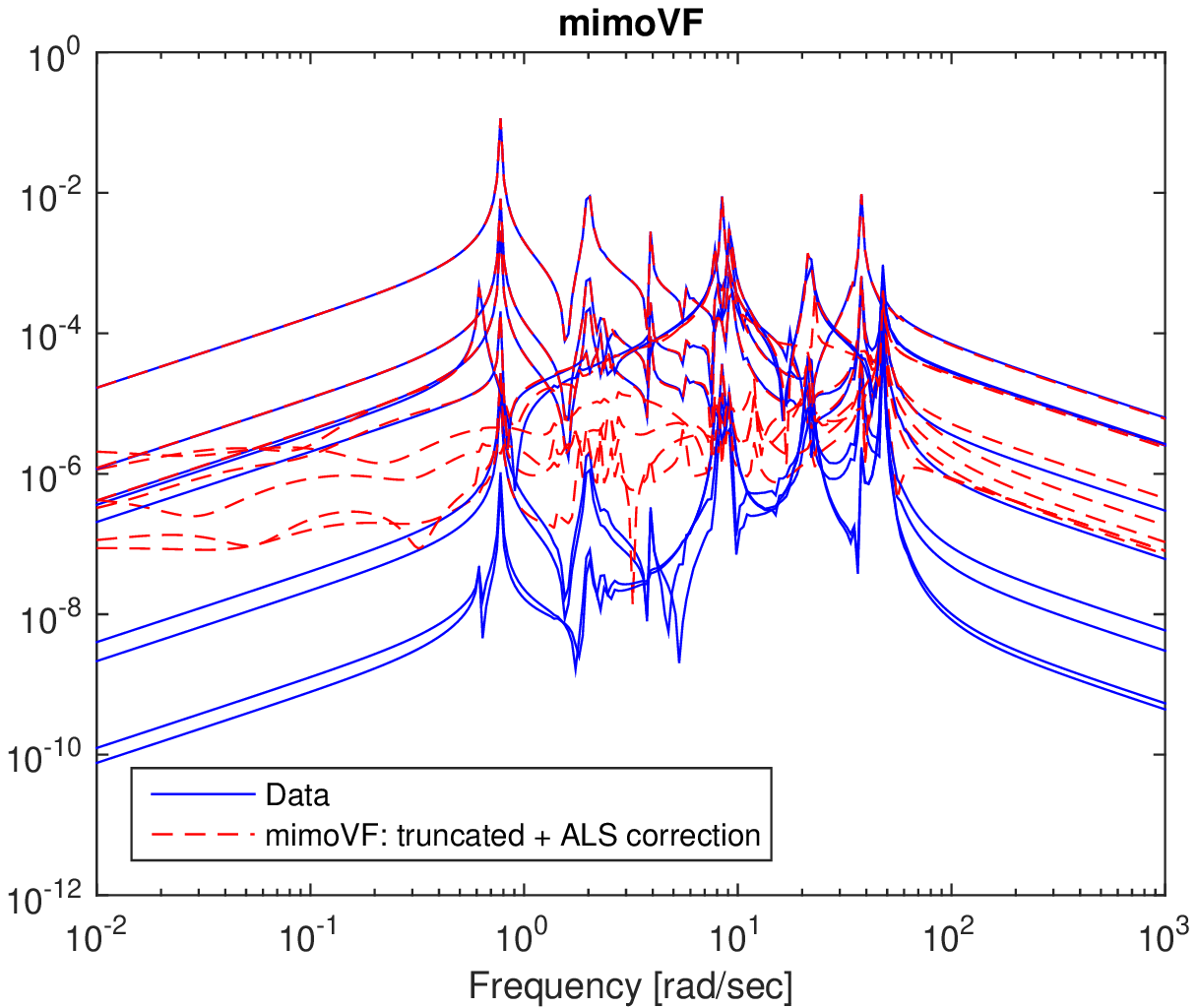}
	\end{center}
	%	\includegraphics[width=5.9in,height=2.8in]{VFZSSL1000r90_truncated.eps}
	%\vspace{-7mm}
	\caption{\label{FIG-compare-to-vf3-ISS270_oneiter_tr} (Example \ref{NumEx-1}.) Illustration of the truncation of \textsf{mimoVF} and the \textsf{ALS} correction. (Truncation applied to the \textsf{mimoVF} output shown on the right graph on Figure \ref{FIG-compare-to-vf3-ISS270_2iter}. Only one \textsf{ALS} iteration is used.) The relative $\Hardy_2$ error of $\widehat{\bfH}_r$ is $\chi\approx 1.50e-01$.
	}
\end{figure} 
%In Figure \ref{FIG-compare-to-vf3-ISS270_Bode11}
%\begin{figure}
%	\begin{center}
%		\includegraphics[width=2.9in,height=2.8in]{ISStwoIterBadPoles_VFZ_Bode11.eps}
%		\includegraphics[width=2.9in,height=2.8in]{ISStwoIterBadPoles_VFZ_truncated_ALS_Bode11.eps}
%	\end{center}
%	%	\includegraphics[width=5.9in,height=2.8in]{VFZSSL1000r90_truncated.eps}
%	%\vspace{-7mm}
%	\caption{\label{FIG-compare-to-vf3-ISS270_Bode11} (Example \ref{NumEx-1}.) The Bode plots for the first input-to-first output of $\bfH_r$ (left) and $\widehat{\bfH}_r$ (right), compared to the original system.
%	}
%\end{figure} 		
%\end{example}

%\begin{remark}
%	{\em
%		If both functions are given the  initial poles of the
%		form $\alpha_j \pm \imunit \beta_j$, where the $\beta_j$s are logarithmically spaced between the minimal and the maximal
%		sampling frequency and $\alpha_j=-\beta_j$\footnote{This choice has been suggested in \cite{}}, and if more iterations are allowed, then both produce good match to the measurement. Interestingly, in that case truncating the residues to their best rank-one approximations also gives a reasonably good approximation, even without the corrections. The purpose of Example \ref{NumEx-01} and Example \ref{NumEx-1} is to show performance in case of an unpropitious distribution of the poles.
%	}
%\end{remark}

\subsection{Rank-one residue correction via $\Hardy_2/\Hardy_{\infty}$ model reduction approaches}\label{SS::H2-res-corr}
%\subsection{Reducing the \textsf{mimoVF} to degree-$r$ using a systems theoretic measure}
The \textsf{ALS} iteration described above is built purely upon algebraic least squares error minimization. 
%rational approximants $\bfH_r$ (\ref{eq:VF-ss-realization}) and $\widehat{\bfH}_r$ (\ref{Hr_sum_cbt})
%%$$
%%\bfH_r^{vf}(s) =  \sum_{j=1}^r \frac{\respm_j^{(k)}}{s-\pol_j^{(k)}}\;\;\mbox{and}\;\;
%%\bfH_r(s) = \sum_{j=1}^r \frac{c_j b_j^T}{s-\lambda_j} \equiv C (s\Id - \Lambda)^{-1} B^T
%%$$
%are  As $\bfH_r$ has more degrees of freedom than $\widehat{\bfH}_r$, truncation inevitably causes loss of
%accuracy, that can be partially (and in some cases very successfully) recovered 
%using the \textsf{ALS} correction, as described in \S \ref{ss=ALS-correction}. 
However, if the underlying context relates the rational approximants to dynamical systems, 
then it may be advantageous to perform this reduction using well-developed systems-theoretic reduction tools. 
Recasting our rank-one residue correction problem into this setting, we consider 
constructing an $r$-th order system $\widehat{\bfH}_r$ that closely
approximates the \textsf{VF} computed model $\bfH_r$ in some appropriate system norm.

The $\Hardy_2$ norm discussed in \S \ref{SS:quad} is the most natural choice and the first one we consider. This approach is compelling  when weights and nodes in \textsf{mimoVF} are chosen using an appropriate quadrature as in \S \ref{SS:quad} and the algebraic \textsf{LS} measure is viewed as a discretized $\Hardy_2$ measure. In this case, the complete procedure; both  \textsf{mimoVF} step and reduction to true McMillan degree-$r$ will be 
performed with the $\Hardy_2$ system norm in mind. To achieve this goal, i.e., to minimize  $\|\bfH_r - \widehat{\bfH}_r \|_{\Hardy_2}$ over all stable $r$th order $\widehat{\bfH}_r$, we will apply the optimal $\Hardy_2$ approximation method  \textsf{IRKA} of \cite{Gugercin_Antoulas_Beattie:2008:IRKA} as modified in \cite{Beattie_Gugercin::RealizationIndependent} for a realization independent procedure.\footnote{It is not required that the underlying transfer function is rational.}  Note that the output of \textsf{mimoVF}, $\bfH_r(s)$,  has the McMillan degree up to $r \min(p,m)$; thus $\bfH_r(s)$ can have a modest state-space dimension. If, for example, $r=80$ and $m=p=50$, $\bfH_r$ can have degree as high as $4000$. Thus, it is important to perform this second reduction step effectively.
 A particularly attractive aspect of the \textsf{IRKA} framework of \cite{Beattie_Gugercin::RealizationIndependent} is that 
 it  needs only function and derivative evaluations at dynamically generated points. This works perfectly in our setting since the explicit state-space form of the \textsf{mimoVF} output in (\ref{eq:VF-ss-realization}) makes these computations 
 trivial. This approach can be viewed as a data driven implementation of the \textsf{IRKA} -- the measurements are fed into  \textsf{mimoVF} to produce an intermediate model, a surrogate of the order  $\widetilde{r}\geq r$, based on measurements, and then this intermediate model is reduced by \textsf{IRKA} to its locally best $r$th order approximant. 

The $\Hardy_\infty$ norm is another commonly used system norm. For a stable dynamical system with transfer function $\bfH(s)$, the $\Hardy_\infty$ norm is defined as ${\| \bfH\|_{\Hardy_\infty} = \sup_{\omega \in \IR} 
\| \bfH(\imunit \omega)\|_2}$. For details, we refer the reader to \cite{ZhDG96}.  The commonly used approach to model reduction towards obtaining a small $\Hardy_\infty$ error measure is Balanced Truncation (\textsf{BT}) \cite{moore1981principal,mullis1976synthesis}. Even though \textsf{BT} requires solving two Lyapunov equations, once again particular state space realization in (\ref{eq:VF-ss-realization}) allows straightforward solution of the Lyapunov equations and makes the \textsf{BT} related computations cheap. Thus, we may also employ \textsf{BT} in reduction 
to true McMillan degree-$r$ without much additional computational cost.

\subsection{An aggregate procedure: \bdgfit} Our overall approach to adapting \textsf{VF} to matrix-valued rational approximation 
consists first of the \textsf{mimoVF} process (described in detail in \S \ref{S=Background} and \S \ref{S=NumericalDetails}), 
 followed by a post-processing step that performs the reduction to true McMillan degree-$r$ with minimal loss of fidelity. 
This post processing stage can be performed either by the \textsf{ALS} correction of \S \ref{ss=ALS-correction} or by systems-theoretic approaches such as  \textsf{IRKA} or \textsf{BT} as described in \S\ref{SS::H2-res-corr}. 
We will refer to this two-step process as \bdgfit.

%\subsection{Numerical Examples for Quadrature}
%Here, with a simple example, we illustrate the effect of choosing the sampling points and the weights using numerical quadrature.  
%We use the ISS 1R module \cite{gugercin2001iss} with $m=3$ inputs and $p=3$ outputs. We  use only $100$ function evaluations ($\ell=200$) and apply
% \textsf{mimoVF} to this data set for different $r$ values. Resulting relative $\mathcal{H}_2$ errors are shown in Table \ref{tab:quad} below.
%\begin{table}[hh]
%\centering
%\begin{tabular}{|l||c|c|c|c|} 
%\hline
% ~\hspace{0.6cm}Method & $r=10$ & $r=20$ & $r=30$ & $r=40$ \\ \hline 
%\textsf{mimoVF} without Quadrature &  $3.4104 \times 10^{-1}$ &  $3.1775 \times 10^{-1}$  & $1.2778 \times 10^{-1}$ & $ 5.0257 \times 10^{-2}$ \\ \hline 
%\textsf{mimoVF} with\phantom{out} Quadrature  & $2.5209\times 10^{-1}$  & $4.6074 \times 10^{-2}$& $3.3226 \times 10^{-2}$  & $ 2.1436 \times 10^{-2}$\\ \hline 
%\end{tabular}
%\caption{Effect of quadrature nodes and weights on \textsf{mimoVF}}
%\label{tab:quad}
%\end{table}
%

\subsection*{Numerical Examples}\label{SS=bdgfit}

Here, we illustrate the performance of  \bdgfit \ with four numerical examples. In each case, we investigate the effect of the methodology employed in the post-processing stage
on the overall approximation quality. 
We also compare the final models produced by \bdgfit \ with the 
optimal-$\Hardy_2$ approximations obtained by \textsf{IRKA}.

\subsubsection{Heat Model}  \label{sec:heat_rank1}

We consider the Heat Model from the \textsf{NICONET} Benchmark collection \cite{NICONET-report}; the model has $m=2$ inputs and 
$p=2$  outputs.   We use only $20$ function evaluations ($\ell=40$ samples due to complex conjugacy) and obtain rational approximations of order $r=6$ and $r=10$. Table \ref{table:heat_rank1} lists the resulting relative $\Hardy_2$ errors due to different approaches.  The first row is the error due to the output of \textsf{mimoVF}. Note that this approximation has order $r\times m=2r$ since it has full-rank residues.  This is not our final approximation and is included here as a reference point. We obtain a true degree-$r$ approximant using four different approaches: (i) simple rank-$1$ truncation of the residues by SVD, (ii) \textsf{ALS} correction of \S \ref{ss=ALS-correction}, (iii) \textsf{IRKA} on the degree-$2r$ output of  \textsf{mimoVF} to reduce it to degree-$r$ (iv)  \textsf{BT} on the degree-$2r$ output of \textsf{mimoVF}  to reduce it to degree-$r$. These four methods are labeled, respectively, as \bdgfit-(\textsf{Trnct}), \bdgfit-(\textsf{ALS}), \bdgfit-(\textsf{IRKA}), and \bdgfit-(\textsf{BT}).  The first observation is that the simple rank-$1$ truncation of the residues by SVD leads to a high loss of accuracy compared to the \textsf{ALS} correction; this is most apparent in the $r=10$ case where \bdgfit-(\textsf{Trnct}) has one order of magnitude higher error than 
 \bdgfit-(\textsf{ALS}). For both cases, \bdgfit-(\textsf{IRKA}) and \bdgfit-(\textsf{BT}) perform extremely well (especially \bdgfit-(\textsf{IRKA})) and even with a true degree-$r$ approximant, they almost match the accuracy of  the degree-$2r$ \textsf{mimoVF} approximant; i.e., reduction from $2r$ to $r$ causes a negligible  loss of  accuracy. The last row indicates the relative $\Hardy_2$ error  associated with the optimal approximant 
 from \textsf{IRKA}. As expected, \textsf{IRKA} yields smaller error; we do not anticipate to beat the continuous optimal approximation via a discretized least-square measure. However,  it is important to note that  \bdgfit-(\textsf{IRKA}) and \bdgfit-(\textsf{BT}) with only 
 $20$ function evaluation yield results close to those obtained by \textsf{IRKA}; this is  especially true for $r=10$.
\begin{table}[hhh]
\centering
\begin{tabular}{|l||c|c|} 
\hline
 ~\hspace{0.6cm}Method & $r=6$ & $r=10$ \\ \hline 
 \textsf{mimoVF}  ({\bf degree: $2r$})& $1.6530\times 10^{-2}$ & $ 1.0759 \times 10^{-3}$  \\ \hline  \hline
 \bdgfit-(\textsf{Trnct}) & $3.7022\times 10^{-2}$ & $ 2.6137 \times 10^{-2}$\\ \hline 
\bdgfit-(\textsf{ALS}) &  $1.8218\times 10^{-2}$& $ 2.8774 \times 10^{-3}$ \\ \hline 
\bdgfit-(\textsf{IRKA}) & $ {1.7359 \times10^{-2}} $ & $ 1.1686 \times 10^{-3}$ \\ \hline 
 \bdgfit-(\textsf{BT}) & $ 3.0604\times10^{-2} $ & $  1.1591 \times 10^{-3}$\\ \hline  \hline
\textsf{IRKA}  &  $ 8.5566\times10^{-3} $  & $ 1.0925 \times 10^{-3}$\\ \hline
\end{tabular}

\vspace{1ex}
\caption{The relative $\Hardy_2$ errors due to \textsf{mimoVF}, \bdgfit \ and \textsf{IRKA}. 
$20$ function evaluations  } 
\label{table:heat_rank1}
\end{table}

\subsubsection{ISS-1R Module}  \label{sec:iss_rank1}
We repeat the above studies for the ISS 1R module \cite{gugercin2001iss} with $m=3$ inputs and $p=3$ outputs.
We use  $100$ function evaluations ($\ell=200$ samples) and obtain rational approximations of order $r=20$ and $r=30$.
 Table \ref{table:iss1r_rank1} depicts the resulting relative $\Hardy_2$ error values for the same methods used in the previous example in \S \ref{sec:heat_rank1}. As in the previous example, both \bdgfit-(\textsf{IRKA}) and \bdgfit-(\textsf{BT}) yield very accurate results and show negligible loss of accuracy in reduction from the intermediate $3r$ approximant to the final degree-$r$ approximant; the order of the \textsf{mimoVF} approximant is reduced three-fold yet not much accuracy is lost.  \bdgfit-(\textsf{IRKA}) and \bdgfit-(\textsf{BT})   again yield approximation errors close to that of \textsf{IRKA}. The main difference from the previous case is that in this case  even the \textsf{ALS} correction suffers from the loss of accuracy as the simple truncation approach. 
  
\begin{table}[hhh]
\centering
\begin{tabular}{|l||c|c|} 
\hline
 ~\hspace{0.6cm}Method & $r=20$ & $r=30$ \\ \hline 
 \textsf{mimoVF}  ({\bf degree: $3r$})& $ 4.6074\times 10^{-2}$ & $ 3.3226  \times 10^{-2}$  \\ \hline  \hline
 \bdgfit-(\textsf{Trnct}) & $ 1.2904 \times 10^{-1}$ & $ 1.2508 \times 10^{-1}$\\ \hline 
\bdgfit-(\textsf{ALS}) &  $ 1.2558 \times 10^{-1}$& $ 1.2223 \times 10^{-1}$ \\ \hline 
\bdgfit-(\textsf{IRKA}) & $ { 7.7305 \times10^{-2}} $ & $ 4.4757 \times 10^{-2}$ \\ \hline 
 \bdgfit-(\textsf{BT}) & $7.7457 \times10^{-2} $ & $  3.3483 \times 10^{-2}$\\ \hline  \hline
\textsf{IRKA}  &  $6.7779 \times10^{-2} $  & $1.1423  \times 10^{-2}$\\ \hline
\end{tabular}

\vspace{1ex}
\caption{The relative $\Hardy_2$ errors due to \textsf{mimoVF}, \bdgfit \ and \textsf{IRKA}.
$100$ function evaluations 
}
\label{table:iss1r_rank1}

\end{table}

\subsubsection{ISS-12A Module}  \label{sec:iss12a_rank1}
We now investigate the larger ISS 12A module \cite{gugercin2001iss} with $m=3$ inputs and $p=3$ outputs. We focus on this problem since the underlying system of degree $n=1412$ is very hard to approximate with a lower order system; it presents significant challenges to model reduction, not necessarily from a computation perspective but from an approximation quality perspective. 
As  illustrated in  \cite{gugercin2001iss}, the Hankel singular values decay rather slowly, so to obtain a reduced model with a relative error tolerance of $10^{-3}$, one needs a reduced model of order at least $226$ even using balanced truncation. We use $250$ function evaluation, obtain rational approximations of order  $r=80$ using  \bdgfit,  and compare the result with the continuous optimal approximation \textsf{IRKA}. 
 Table \ref{table:iss12a} depicts the resulting relative $\Hardy_2$ error values. As before, both \bdgfit-(\textsf{IRKA}) and \bdgfit-(\textsf{BT})  show negligible loss of accuracy in reduction from the intermediate $3r$ approximant to the final degree-$r$ approximant and have approximation errors close to that of \textsf{IRKA}. 
    \bdgfit-(\textsf{Trnct}) and \bdgfit-(\textsf{ALS}) perform reasonably well as well in this case with \bdgfit-(\textsf{Trnct}) having the largest error among the four. 
\begin{table}[hhh]
\centering
\begin{tabular}{|l||c|c|} 
\hline
 ~\hspace{0.6cm}Method & $r=80$   \\ \hline 
 \textsf{mimoVF}  ({\bf degree: $3r$})& $ 1.4678\times 10^{-1}$    \\ \hline  \hline
  \bdgfit-(\textsf{Trnct}) & $ 2.4549 \times 10^{-1}$ \\ \hline 
\bdgfit-(\textsf{ALS}) &  $ 2.2075 \times 10^{-1}$ \\ \hline 
\bdgfit-(\textsf{IRKA}) & $ { 1.5116 \times10^{-1}} $  \\ \hline 
 \bdgfit-(\textsf{BT}) & $1.5130 \times10^{-1} $  \\ \hline  \hline
\textsf{IRKA}  &  $1.1317 \times10^{-1} $   \\ \hline
\end{tabular}

\vspace{1ex}
\caption{The relative $\Hardy_2$ errors due to \textsf{mimoVF}, \bdgfit \ and \textsf{IRKA}.
$250$ function evaluations 
}
\label{table:iss12a}

\end{table}

\subsubsection{A power system example with large input/output space}  \label{sec:4646}
This example results from  small-signal stability studies for large power systems. The specific model we consider here is one of the models from the Brazilian Interconnect Power
System (BIPS); refer to \cite{Rommes2007}\footnote{This model can be downloaded from  {\tt https://sites.google.com/site/rommes/software}} for details.

The underlying dynamical system has dimension $n=13250$ with $m=46$ inputs and $p=46$ outputs. We apply
\bdgfit-(\textsf{BT})  to obtain our rational approximant. We use $\ell=200$ frequency samples.  The \textsf{mimoVF} (first step of \bdgfit) is applied with $r=40$. Due to  $m=p=46$, the output of \textsf{mimoVF} has an effective
McMillan degree of $r\times m = 1840$.  Note that this is only marginally a reduced model, having a McMillan degree roughly $14\%$ of the originally system order. The decay of the leading Hankel singular values of the intermediate model is shown in the upper plot of Figure \ref{fig:4646}. Note the slow decay;  even after the $300^{\rm th}$ one,  the normalized Hankel singular values  are still above the threshold of $10^{-4}$. This system is difficult to reduce. Indeed, in earlier works, even for simpler versions of the model having a smaller number of input and outputs (such as $m=p=28$), a reduced model of degree $291$ was used; see \cite{Rommes2006}. We choose the final degree to be a point at which the normalized Hankel singular values have decayed below $2\times 10^{-4}$, leading to a  final McMillan degree of $253$ (around $2\%$ of the original). The sigma plots, i.e., $\| \bfH(\imunit \omega)\|_2$ vs $\omega \in \IR$, for  the full model and the final  \bdgfit~approximant are shown in the lower plot of Figure \ref{fig:4646}. As the figure illustrates, the \bdgfit~approximant does an excellent job in approximating the underlying dynamics.
\begin{figure}
	\begin{center}
\includegraphics[width=2.9in,height=2.9in]{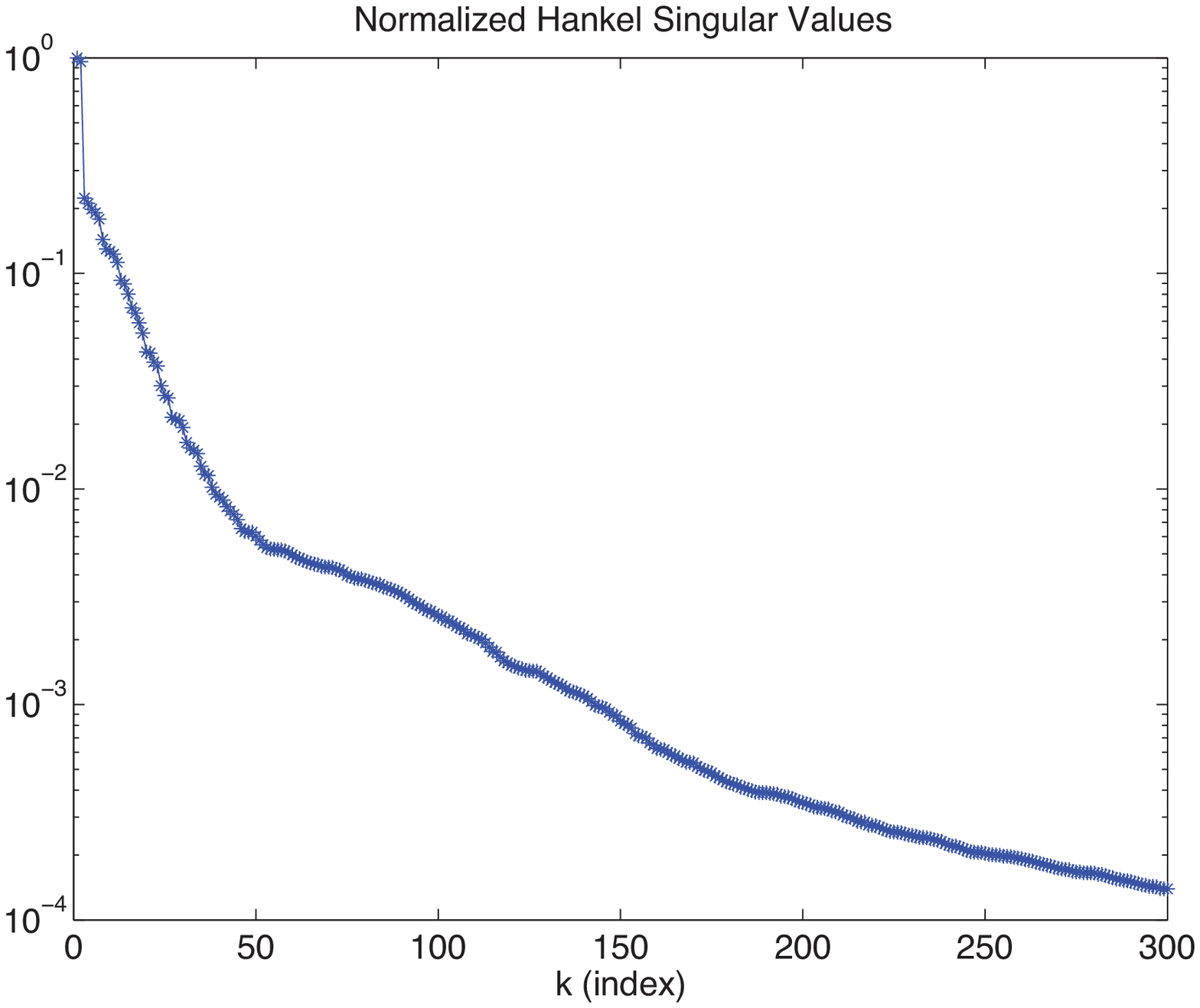} 
\includegraphics[width=2.9in,height=2.9in]{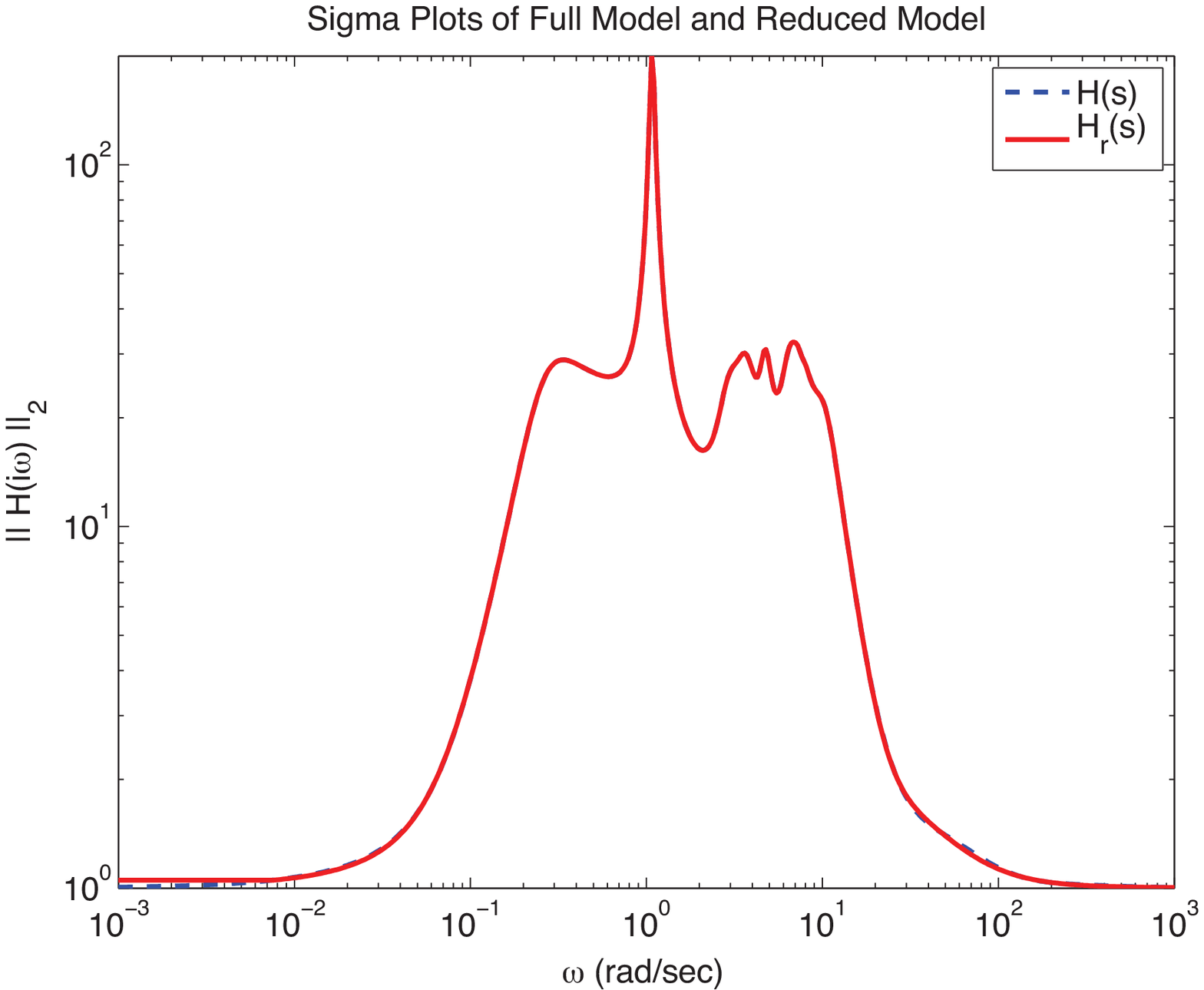} 
	\end{center}
	%	\includegraphics[width=5.9in,height=2.8in]{VFZSSL1000r90_truncated.eps}
	%\vspace{-7mm}
	\caption{\label{fig:4646} (Example \ref{sec:4646}.) Large-scale power system example: Left plot: Decay of the Hankel singular values 
	Right plot:  The sigma plots of the full and reduced model.
	}
\end{figure}

%
%
%\begin{figure}
%	\begin{center}
%		\includegraphics[width=3.5in,height=3.5in]{paper_4646.eps}
%	\end{center}
%	%	\includegraphics[width=5.9in,height=2.8in]{VFZSSL1000r90_truncated.eps}
%	%\vspace{-7mm}
%	\caption{\label{fig:4646} (Example \ref{sec:4646}.) Large-scale power system example: Top plot: Decay of the Hankel singular values 
%	Bottom plot:  The sigma plots of the full and reduced model.
%	}
%\end{figure} 	

%}	
%\end{example}	

\section{Conclusions}
The \textsf{VF} method, as originated by Gustavsen and Semlyen,
 has been and continues to be an important tool 
for rational approximation and many authors have applied, modified, and analyzed this approach.  Although their method is based on successive solution of linear \textsf{LS} problems, which are well understood problems in of themselves, we find that subtleties enter that can degrade both the performance and accuracy of current \textsf{VF} approaches, especially for matrix-valued rational approximation of modest dimension which often produce extremely poorly conditioned problems. 
 These issues include 
 \begin{enumerate}
 \item[1)] balancing the potential conflict between rank revealing pivoting in QR factorizations used in 
\textsf{LS} solvers and column-scaling used to improve conditioning; 
\item[2)] avoiding redundant computation within the multiple subproblems solved as part of the large \textsf{LS} problem that arises; 
\item[3)] the need for rigorous termination criteria and the efficient recovery of the best possible rational approximant in case the iteration must be terminated prematurely; 
\item[4)] the use of regularized least squares and the discrepancy principle, both implemented using high accuracy linear algebra methods; and
\item[5)] control of the McMillan degree of the resulting rational approximant.  
\end{enumerate}
 In this paper, we have considered these issues carefully, together with other more minor ones.  
 We have integrated these developments into a robust, efficient implementation of \textsf{VF} for matrix-valued rational approximation, called \textsf{mimoVF}, that appears to be both faster and more accurate than currently available implementations.  
Further, we have connected the underlying discrete \textsf{LS} approximation problem to a continuous optimal $\Hardy_2$ approximation problem through numerical quadrature, which motivated a reformulation of the original \textsf{VF} objective as a weighted \textsf{LS} problem.    For essentially the same cost as \textsf{mimoVF}, we were then able to use this reformulation to significantly improve the quality of the approximation.   
Finally, we have offered here an aggregate procedure, called \bdgfit,~that
 combines \textsf{mimoVF} with $\Hardy_2$/$\Hardy_\infty$-based approximation methods that yield high-fidelity rational approximants with low McMillan degree.

\bibliography{ModRed}
\bibliographystyle{siam}

%\clearpage
%\setcounter{section}{0}
 %\renewcommand{\thesection}{\Alph{section}}
%\input{vfmimo_over_R}

\end{document}

%% file: vf_def.tex
\newtheorem{example}{Example}[section]
\newtheorem{remark}{Remark}[section]

\newcommand{\R}{\mathbb R}
\newcommand{\Cplx}{\mathbb C}

\newcommand{\Cauchy}{{\mathscr{C}}}

\newcommand{\Hardy}{\mathcal{H}}

\newcommand{\pol}{\lambda}

\newcommand{\cc}[1]{\overline{#1}}

\newcommand{\mb}[1]{\left(\begin{array}{#1}}
\newcommand{\me}{\end{array}\right)}
\newcommand{\eqb}{\begin{equation}}
\newcommand{\eqe}{\end{equation}}
\newcommand{\eqab}{\begin{eqnarray}}
\newcommand{\eqae}{\end{eqnarray}}

\newcommand{\bfgamma}{{\boldsymbol\gamma}}

\newcommand{\bflambda}{{\boldsymbol\lambda}}
\newcommand{\bfpol}{{\boldsymbol\lambda}}

\newcommand{\bfxi}{{\boldsymbol\xi}}

\newcommand{\bfPhi}{{\boldsymbol\Phi}}
\newcommand{\bfvarphi}{{\boldsymbol\varphi}}
\newcommand{\bfeps}{{\boldsymbol\epsilon}}

\def\roff  {\mbox{\boldmath$\varepsilon$}}

\newcommand{\qw}{\rho}
\newcommand{\dqw}{{\mathcal D}_\rho}

\newcommand{\respm}{\bfPhi}
\newcommand{\dres}{\varphi}
\newcommand{\Id}{\mathbb I}
\newcommand{\HT}{\mathbb{S}}

\newcommand{\bfH}{ {\mathbf H} }
\newcommand{\bfB}{ {\mathbf B} }

\newcommand{\bfN}{ {\mathbf N} }
\newcommand{\bfM}{ {\mathbf M} }

\newcommand{\IR}{ {\mathbb R} }
\newcommand{\IC}{ {\mathbb C} }

\def\bfcE{\mbox{\boldmath$\mathcal{E}$}}
\newcommand{\imunit}{{\dot{\imath\!\imath}}}
\newcommand{\bfS}{ {\mathbf S} }

\numberwithin{equation}{section}\numberwithin{equation}{section}

\def\bdgfit{\mbox{\textsf{mimoFIT}}}